\documentclass[11pt]{article}

\usepackage[a4paper,margin=1in]{geometry}
\usepackage{amsmath,amssymb,amsthm,bm}
\usepackage[caption=false,font=normalsize,labelfont=sf,textfont=sf]{subfig}
\usepackage{graphicx}
\usepackage{enumitem}
\usepackage{natbib}
\newtheorem{assumption}{Assumption}
\newtheorem{lemma}{Lemma}
\newtheorem{corollary}{Corollary}
\newtheorem{theorem}{Theorem}
\newtheorem{remark}{Remark}
\newtheorem{proposition}{Proposition}
\usepackage{algorithm}
\usepackage{algpseudocode}
\newcommand{\ba}{\begin{aligned}}
\newcommand{\ea}{\end{aligned}}

\newcommand{\col}{\operatorname{col}}

\newcommand{\ub}{\mathbf{u}}

\newcommand{\xb}{\mathbf{x}}
\newcommand{\yb}{\mathbf{y}}
\newcommand{\zb}{\mathbf{z}}

\newcommand{\Gb}{\mathbf{G}}

\newcommand{\sbf}{\mathbf{s}}

\title{Distributed Nonlinear Equality-Constrained Optimization via Feedback Linearization and Singular Perturbation}
\author{%
    Zihao Ren,
    Lei Wang,
    Hongye Su, and
    Guodong Shi\thanks{Zihao Ren, Lei Wang and Hongye Su are with the College of Control Science and Engineering, Zhejiang University, P.R. China. (E-mail: zhren2000; lei.wangzju; hysu69@zju.edu.cn). Guodong Shi is with  the Australia Centre for Field Robotics, The University of Sydney, Australia. (Email: guodong.shi@sydney.edu.au).  
}}

\date{}

\begin{document}
\maketitle

\begin{abstract}
Distributed optimization becomes particularly challenging when nonconvex
objectives are combined with nonlinear equality constraints: feasibility is
network coupled, while most existing exponentially convergent methods rely on
convexity or affine constraints. This paper introduces a
feedback-linearization and singular-perturbation framework for two
representative problem classes involving, respectively, distributed local
nonlinear equality constraints and nonlinear aggregate equality constraints.
The framework separates the regulation of consensus and feasibility residuals
from optimization along the feasible manifold. Specifically, an ideal
feedback-linearized dynamics is constructed whose output behavior can be
explicitly assigned, while its zero-output dynamics coincides with the
projected gradient flow of the aggregate objective on the feasible manifold.
The ideal feedback-linearizing input, however, is determined by a
state-dependent, network-coupled algebraic equation and is therefore not
directly implementable in a distributed manner. To overcome this obstruction,
we replace the nonlocal algebraic solution with a fast residual-tracking
dynamics, yielding a singular-perturbation realization that uses only local and
neighboring information. Under a local quadratic-growth condition imposed only
on the feasible manifold, we establish local exponential convergence of both
the ideal and distributed dynamics. For sufficiently strong time-scale
separation, the convergence rate of the distributed realization can be chosen
arbitrarily close to that of the ideal dynamics. Explicit Euler
discretizations are also proved to preserve local exponential convergence.
Under stronger manifold regularity conditions, the region of attraction
extends to a tubular neighborhood of the feasible manifold and, for affine
local equality constraints, to the whole admissible state space. Moreover,
classical primal--dual dynamics are recovered as a particular realization of
the proposed framework. Numerical studies validate the
theoretical results.
\end{abstract}

\textbf{Keywords:} distributed optimization, nonconvex constrained optimization,
nonlinear equality constraints.

\section{Introduction}

Distributed optimization plays an increasingly pivotal role in
networked intelligent systems, including multi-agent coordination, smart grids,
and cyber-physical systems
\cite{magnusbook,martinez07,kar12,Rabbat2010}. It is also an important
computational paradigm for decentralized and federated machine learning, where
agents collaboratively train a shared model from locally stored data without
centralizing the raw samples
\cite{mcmahan2017communication,kong2021consensus}. In such problems, a group of
agents cooperatively minimizes an aggregate objective composed of locally
stored functions through local information exchange. Basic combinations of
consensus dynamics and gradient descent lead to distributed gradient methods,
which typically achieve sublinear convergence under convexity assumptions
\cite{DSMF,DSCO}. To improve the convergence rate, gradient-tracking and
primal--dual methods have been developed and can achieve linear convergence
under strong convexity and smoothness conditions
\cite{nedic2017achieving,alghunaim2020linear}. These developments have also
been extended to nonconvex and constrained distributed optimization settings
\cite{XY-LCOF,di2016next,scutari2019distributed}.

An important class of constrained distributed optimization problems involves
local constraints, where each agent is subject to its own feasibility
requirements, such as
\(h_i(\mathbf x_i)=0\) or \(g_i(\mathbf x_i)\leq0\).
Such formulations naturally arise when individual agents have different
physical limitations and local data-consistency requirements,
with representative applications in smart-grid control and distributed
statistical learning
\cite{chang2014distributed,vardakas2015survey,mateos2010distributed,
ji2023distributed}. Projection-based algorithms and
primal--dual dynamics have been extensively studied for distributed local
constraints; see, e.g.,
\cite{nedic2010constrained,zhu2012distributed,AMAS,zeng2017distributed}.

Another important class involves coupled constraints, where feasibility is
determined by an aggregate quantity across the network, such as
\(\sum_{i=1}^{N}h_i(\mathbf x_i)=0\) or
\(\sum_{i=1}^{N}g_i(\mathbf x_i)\leq0\). In this case, no individual agent can
verify feasibility using only its local constraint map, and additional
coordination is required to enforce the aggregate condition. Coupled constraints also commonly
arise in resource allocation, energy management, networked estimation and distributed learning 
\cite{cherukuri2016initialization,cherukuri2018distributed,
gan2013optimal,zhang2013robust,cotter2019training,harder2023hard}. Existing approaches mainly rely on dual
decomposition, primal--dual coordination, and auxiliary-variable techniques;
see, e.g.,
\cite{falsone2017dual,wu2022distributed,notarnicola2020constraint,
camisa2021distributed,falsone2023augmented}.

Despite these advances, distributed optimization with equality constraints
still leaves substantial room for further development, especially compared with
unconstrained distributed optimization. Most existing studies on
equality-constrained distributed optimization are developed under convexity
assumptions on the objective functions
\cite{zhu2012distributed,AMAS,falsone2017dual,wu2022distributed}, while results for nonconvex
objectives are comparatively limited. In addition, the equality constraints
considered in the above existing works are affine, including local constraints of
the form \(\mathbf{A}_i\mathbf{x}_i=\mathbf{b}_i\) and coupled constraints of the form
\(\sum_{i=1}^{N}\mathbf{A}_i\mathbf{x}_i=\mathbf{b}\). Nonlinear equality constraints have received
relatively less attention. To the best of our knowledge,
\cite{matei2013distributed} considered distributed optimization with local
nonlinear equality constraints and established exponential convergence under a
strong convexity-type condition on the Hessian of the Lagrangian. In contrast,
distributed optimization with nonlinear coupled equality constraints remains
much less explored. Overall, for both distributed local equality constraints and coupled equality
constraints, achieving exponential convergence under nonconvex objectives and
nonlinear equality constraints remains a largely open question. This problem class is practically relevant, since nonconvex objectives
and nonlinear equality constraints arise in many applications, particularly in
machine-learning problems
\cite{chen2021decentralized,lee2022fairpca,cotter2019training,
harder2023hard}.

Recently, in the centralized optimization setting, control-theoretic
Lagrangian flows have been developed as an alternative viewpoint for
equality-constrained optimization with nonconvex objective functions and
nonlinear equality constraints
\cite{cerone2025framework,zhang2026constrained,Pirrera2026Global}.
In particular, feedback-linearization-based (FL-based) dynamics regard the Lagrange
multiplier as a control input and the constraint violation as the system
output. By assigning a stable output dynamics, the constraint violation is
actively driven to zero, while the resulting zero dynamics describes the
motion along the constraint manifold. This perspective allows exponential
convergence under a restricted strong-convexity condition on the constraint
manifold, without requiring strong convexity of the objective function over
the whole Euclidean space
\cite{Pirrera2026Global}. Inspired by this line of work, we develop an FL-based design
for the corresponding distributed optimization problems with equality
constraints.

In this paper, we develop a feedback-linearization-based framework and its
distributed singular-perturbation realization for optimization with nonconvex
objectives and nonlinear equality constraints. We consider two representative
problem classes: distributed local nonlinear equality constraints and
nonlinear aggregate equality constraints. At the ideal level, the consensus
and equality-constraint residuals are selected as regulated outputs, while
multiplier-related variables are treated as control inputs. The resulting ideal
FL dynamics admits an explicitly assigned stable output vector field, and its
zero-output dynamics coincides with the projected gradient flow of the
aggregate objective on the corresponding feasible manifold. At the
implementation level, we replace the nonlocal algebraic FL solution with a fast residual-tracking dynamics and introduce a change of variables that yields
a fully distributed 
singular-perturbation (SP) realization using only local and neighboring
information. Under a local quadratic-growth condition imposed only on the
feasible manifold, we establish local exponential convergence of the ideal FL
dynamics, the distributed SP realizations, and their explicit Euler
discretizations. Under stronger manifold regularity conditions, the region of
attraction is further enlarged to a tubular neighborhood of the feasible
manifold and, for affine local equality constraints, to the whole admissible
state space. The main contributions of this paper are summarized as follows.
\begin{itemize}
    \item We develop an FL--SP framework for distributed optimization with
    both local and coupled nonlinear equality constraints. The ideal FL
    formulation separates the regulation of consensus and constraint residuals
    from the projected optimization motion on the feasible manifold.  The resulting
    nonlocal algebraic FL relation is then replaced by a fast
    residual-tracking dynamics, yielding a fully distributed SP realization
    that requires only local and neighboring information.

        \item We establish local exponential convergence under a local
    quadratic-growth condition imposed only on the feasible manifold, rather
    than convexity or strong convexity over an ambient convex domain. This
    permits the objective functions to be nonconvex away from the feasible
    manifold. For a sufficiently small singular-perturbation parameter, the
    convergence rate of the distributed SP realization can be chosen
    arbitrarily close to that of the ideal FL dynamics. Local exponential
    convergence is also proved for the explicit Euler discretizations. Moreover, under
    stronger regularity and manifold PL conditions, the region of attraction
    is enlarged to a tubular neighborhood and, for affine local equality
    constraints under global smoothness conditions, to the whole admissible
    state space.

      \item We reveal the relation between the proposed SP realization and
    classical primal--dual dynamics. A properly scaled linear output feedback
    recovers a classical primal--dual method as a special case, whereas a fixed
    output feedback produces a high-gain primal--dual-type realization that
    approximates the assigned FL output dynamics. 

\end{itemize}

The remainder of this paper is organized as follows.
Section~\ref{sec:problem} formulates the two distributed equality-constrained
optimization problems and introduces the standing assumptions.
Section~\ref{sec:local-equality} develops the ideal FL design and its
distributed SP realization for local equality constraints.
Section~\ref{sec:coupled-equality} extends the framework to coupled equality
constraints through an auxiliary-variable lifting.
Section~\ref{sec:discuss} analyzes the connection with classical primal--dual
dynamics and establishes enlarged-region-of-attraction results.
Section~\ref{sec:discrete-implementation} presents the explicit discrete-time
implementations and their convergence guarantees.
Section~\ref{sec:simulations} provides numerical simulations, and
Section~\ref{sec:conclusion} concludes the paper. All proofs are provided in
the Appendices.

\emph{\bf Notation.}
In this paper, \(\|\cdot\|\) denotes the Euclidean norm for vectors and the
induced spectral norm for matrices.  For a matrix \(\mathbf A\), \(\operatorname{rank}\mathbf A\), \(\ker \mathbf A\), and \(\operatorname{Im}\mathbf A\) denote
its rank, kernel, and image space, respectively. For a symmetric matrix
\(\mathbf A\), \(\mathbf A\succeq 0\) (\(\mathbf A\succ 0\)) means that \(\mathbf A\) is positive
(semi)definite. The notation \(\mathbf A\otimes\mathbf B\) denotes the Kronecker
product of \(\mathbf A\) and \(\mathbf B\).  For matrices \(\mathbf A_1,\ldots,\mathbf A_N\) with compatible
column dimensions, \(\operatorname{col}(\mathbf A_1,\ldots,\mathbf A_N)\)
denotes their vertical concatenation, and \(\operatorname{blkdiag}(\mathbf A_1,\ldots,
\mathbf A_N)\) denotes the block diagonal matrix with \(\mathbf A_i\) as its
\(i\)-th diagonal block. For a differentiable function \(f\), \(\nabla f\) and
\(\nabla^2 f\) denote its gradient and Hessian, respectively. For a vector-valued
map \(\mathbf h\), \(\partial \mathbf h/\partial \mathbf x\) denotes its
Jacobian with respect to \(\mathbf x\). For a closed subspace,
\(\mathcal T\), \(\Pi_{\mathcal T}\) denotes the orthogonal projection onto
\(\mathcal T\).

\section{Problem Formulation}
\label{sec:problem}
In this section, we formulate two classes of distributed optimization problems
with equality constraints, i.e., problems with distributed local equality
constraints and problems with coupled equality constraints. We also introduce
the standing assumptions used throughout the paper.

Consider a network of \(N\) agents indexed by
\(\mathcal V:=\{1,\ldots,N\}\). The communication topology is described by an undirected weighted graph
\(\mathcal G=(\mathcal V,\mathcal E)\), where \(\mathcal E\) is the edge set. Let
\(\mathbf A_{\mathcal G}:=[a_{ij}]\in\mathbb R^{N\times N}\) be the weighted adjacency matrix, where
\(a_{ij}>0\) if agents \(i\) and \(j\) can exchange information, and \(a_{ij}=0\) otherwise. The neighbor set of agent \(i\) is
\(\mathcal N_i:=\{j\in\mathcal V:a_{ij}>0\}\). The graph Laplacian
\(\mathbf L_{\mathcal G}\) is defined by
\([\mathbf L_{\mathcal G}]_{ij}=-a_{ij}\) for \(i\neq j\), and
\([\mathbf L_{\mathcal G}]_{ii}=\sum_{j=1}^{N}a_{ij}\). For the network considered in this paper, we have the following assumption.

\begin{assumption}
\label{asm:graph}
The graph \(\mathcal G\) is undirected and connected.
\end{assumption}

Under Assumption~\ref{asm:graph}, one has
\(\mathbf L_{\mathcal G}=\mathbf L_{\mathcal G}^{\top}\succeq0\),
\(\mathbf L_{\mathcal G}\mathbf 1_N=\mathbf 0_N\) \cite{magnusbook}. 

In this network, each agent \(i\) has local decision variable \(\mathbf x_i\in\mathbb R^d\) and a local objective function
\(f_i:\mathbb R^d\to\mathbb R\). 
This paper considers the following two classes of equality-constrained distributed optimization problems.

First, we consider a distributed equality-constrained problem, where each agent has a local equality constraint
\(\mathbf h_i^{\mathrm d}:\mathbb R^d\to\mathbb R^{r_i}\). The problem is
\begin{equation}
\label{prob:dist}
\begin{aligned}
    \min_{\mathbf x_1,\ldots,\mathbf x_N\in\mathbb R^d}\quad
        & \sum_{i=1}^{N}f_i(\mathbf x_i) \\
    \mathrm{s.t.}\quad
        & \mathbf x_i=\mathbf x_j,\qquad \forall i,j\in\mathcal V,\\
        & \mathbf h_i^{\mathrm d}(\mathbf x_i)=\mathbf 0,\qquad i=1,\ldots,N.
\end{aligned}
\tag{\(\mathcal P_{\mathrm d}\)}
\end{equation}
Let \(\mathbf x:=\operatorname{col}(\mathbf x_1,\ldots,\mathbf x_N)\). The feasible set of the problem is
\(\mathcal X_{\mathrm d}
=\{\mathbf 1_N\otimes\mathbf s:\mathbf s\in\Omega_{\mathrm d}\}\), where the corresponding centralized feasible set is
\(
    \Omega_{\mathrm d}
    :=
    \left\{
    \mathbf s\in\mathbb R^d:
    \mathbf h_i^{\mathrm d}(\mathbf s)=\mathbf 0,\ 
    i=1,\ldots,N
    \right\}
\).

Second, we consider a coupled equality-constrained problem, where each agent has a local map
\(\mathbf h_i^{\mathrm c}:\mathbb R^d\to\mathbb R^q\), and the equality constraint is imposed through their aggregate. The problem is
\begin{equation}
\label{prob:coupled}
\begin{aligned}
    \min_{\mathbf x_1,\ldots,\mathbf x_N\in\mathbb R^d}\quad
        & \sum_{i=1}^{N}f_i(\mathbf x_i) \\
    \mathrm{s.t.}\quad
        & \mathbf x_i=\mathbf x_j,\qquad \forall i,j\in\mathcal V,\\
        & \sum_{i=1}^{N}\mathbf h_i^{\mathrm c}(\mathbf x_i)=\mathbf 0.
\end{aligned}
\tag{\(\mathcal P_{\mathrm c}\)}
\end{equation}
The feasible set of the problem is
\(\mathcal X_{\mathrm c}
=\{\mathbf 1_N\otimes\mathbf s:\mathbf s\in\Omega_{\mathrm c}\}\), where the corresponding centralized feasible set is
\(
    \Omega_{\mathrm c}
    :=
    \left\{
    \mathbf s\in\mathbb R^d:
    \sum_{i=1}^{N}\mathbf h_i^{\mathrm c}(\mathbf s)=\mathbf 0
    \right\}
\).

The problem \(\mathcal P_{\mathrm d}\) has been widely studied in the literature. Existing results include algorithms for affine equality constraints, e.g., \cite{zhu2012distributed,AMAS}, as well as extensions to nonlinear equality constraints, e.g., \cite{matei2013distributed}. In contrast, for the coupled problem \(\mathcal P_{\mathrm c}\), most existing studies focus on coupled affine equality constraints, see, e.g., \cite{LYA,falsone2017dual,wu2022distributed}. Compared with these works, the coupled problem considered here allows a general nonlinear aggregate equality constraint.

In the following, we introduce standing assumptions for
\(\mathcal P_{\mathrm d}\) and \(\mathcal P_{\mathrm c}\).
For compactness, let \(\ell\in\{\mathrm d,\mathrm c\}\), where
\(\ell=\mathrm d\) corresponds to \(\mathcal P_{\mathrm d}\), and
\(\ell=\mathrm c\) corresponds to \(\mathcal P_{\mathrm c}\).
Let \(F(\mathbf s):=\sum_{i=1}^{N}f_i(\mathbf s)\).
\begin{assumption}
\label{asm:smooth}
For each \(i\in\mathcal V\), the objective function
\(f_i:\mathbb R^d\to\mathbb R\) is twice continuously differentiable.
Moreover, for each \(\ell\in\{\mathrm d,\mathrm c\}\), the constraint map
\(\mathbf h_i^{\ell}:\mathbb R^d\to\mathbb R^{r_i}\) (\(\ell=\mathrm d\)) or
\(\mathbf h_i^{\ell}:\mathbb R^d\to\mathbb R^{q}\) (\(\ell=\mathrm c\)) is twice continuously
differentiable, with Jacobian
\(
    \mathbf J_i^{\ell}(\mathbf x_i)
    :=
    \frac{\partial \mathbf h_i^{\ell}(\mathbf x_i)}
    {\partial \mathbf x_i}.
\)
In addition, \(\nabla^2 f_i\) and \(\mathbf J_i^{\ell}\) are locally
Lipschitz continuous.
\end{assumption}

\begin{assumption}
\label{asm:growth}
For problem \(\mathcal P_{\ell}\), there exist a point
\(\mathbf s_{\ell}^{\ast}\in\Omega_{\ell}\), an open neighborhood
\(\mathcal U_{\ell}\subset\mathbb R^d\) of
\(\mathbf s_{\ell}^{\ast}\), and a constant \(\rho_{\ell}>0\) such that
\(\mathbf s_{\ell}^{\ast}\) is a local minimizer of \(F\) on
\(\Omega_{\ell}\cap\mathcal U_{\ell}\). Moreover, \(F\) satisfies the local
quadratic growth condition on the feasible set:
\[
    F(\mathbf s)
    \ge
    F(\mathbf s_{\ell}^{\ast})
    +
    \frac{\rho_{\ell}}{2}
    \|\mathbf s-\mathbf s_{\ell}^{\ast}\|^2,
    \qquad
    \forall
    \mathbf s\in\Omega_{\ell}\cap\mathcal U_{\ell}.
\]
\end{assumption}

Under Assumption~\ref{asm:growth},
\(\mathbf s_{\ell}^{\ast}\) is the locally unique optimal solution to minimize $F$ on
\(\Omega_{\ell}\cap\mathcal U_{\ell}\). Consequently,
\(\mathbf x_{\ell}^{\ast}:=\mathbf 1_N\otimes\mathbf s_{\ell}^{\ast}\)
is the locally unique primal optimizer of 
\(\mathcal P_{\ell}\) in a neighborhood of
\(\mathbf 1_N\otimes\mathcal U_{\ell}\).

\begin{remark}
\label{rem:growth}
Assumption~\ref{asm:growth} is a local quadratic-growth condition on the
feasible set, rather than a domain-wise convexity assumption on $\mathcal U_{\ell}$. This assumption
is less restrictive than requiring strong convexity of the objective function
in a convex neighborhood of the optimizer, as commonly imposed in related
distributed optimization studies
\cite{zhu2012distributed,AMAS,matei2013distributed,
falsone2017dual,wu2022distributed}. Indeed, local strong convexity in a convex neighborhood implies local quadratic growth around the optimizer, but
the converse is not true. Moreover, Assumption~\ref{asm:growth} is imposed
only along feasible points in \(\Omega_\ell\cap \mathcal U_\ell\), and hence
allows the objective function to be nonconvex away from the feasible set
\(\Omega_\ell\).
\end{remark}

\begin{remark}
    In this paper, we 
adopt quadratic growth instead of a strong-convexity-type condition
on \(\Omega_\ell\) because, for a general nonlinear constrained manifold, a
pairwise inequality written with ambient gradients and Euclidean chords is not
intrinsic to the restricted objective. To see this, consider
\(\Omega=\{(s_1,s_2)\in\mathbb R^2:s_2=s_1^2\}\) and
\(F(s_1,s_2)=-M(s_2-s_1^2)\), where \(M>0\). Then all feasible points are optimal and
Assumption~\ref{asm:growth} does not hold at any isolated optimizer. However,
for two feasible points \(\xb=(a,a^2)\) and \(\yb=(b,b^2)\), one has
\(\nabla F(\xb)^\top(\yb-\xb)=-M(b-a)^2\). Moreover, on any local patch
\(\Omega_r=\{(s_1,s_1^2): |s_1|\le r\}\), the strongly-convexity-type inequality
\[
    F(\yb)\ge F(\xb)+\nabla F(\xb)^\top(\yb-\xb)
    +\frac{\rho}{2}\|\yb-\xb\|^2,\qquad \forall \xb,\yb\in\Omega_r,
\]
holds for any \(\rho\in(0,2M/(1+4r^2)]\). Indeed, for
\(\xb=(a,a^2)\) and \(\yb=(b,b^2)\), the right-hand side equals
\((b-a)^2[-M+\frac{\rho}{2}(1+(a+b)^2)]\le 0=F(\yb)\). Thus, such an
ambient-gradient-based pairwise inequality may hold even without an isolated
local optimizer, whereas Assumption~\ref{asm:growth} directly imposes the
quadratic separation needed in our convergence analysis.
\end{remark}

\begin{remark}
\label{rem:relation-nonconvex}
The local quadratic-growth condition in Assumption~\ref{asm:growth} is an
instance of the regularity conditions widely used as alternatives to strong
convexity for establishing exponential convergence of first-order methods in
nonconvex optimization \cite{drusvyatskiy2018error,necoara2019linear}, with other representative examples including the
Polyak--{\L}ojasiewicz condition; see, e.g.,
\cite{karimi2016linear,XY-LCOF}.
The key difference here is that Assumption~\ref{asm:growth} is imposed only on
the equality-constrained feasible manifold \(\Omega_{\ell}\), rather than on a convex space. Therefore, the objective function is less restricted in the sense that it is allowed to
be nonconvex away from the feasible manifold. 
\end{remark}

The next two assumptions impose local regularity of the equality constraints for
\(\mathcal P_{\mathrm d}\) and \(\mathcal P_{\mathrm c}\), respectively.

\begin{assumption}
\label{asm:reg-dist}
For problem \(\mathcal P_{\mathrm d}\), on the set  
\(\mathcal X_{\mathrm d}\cap(\mathbf 1_N\otimes\mathcal U_{\mathrm d})\), the stacked constraint Jacobian has full row rank, i.e.,
\[
    \operatorname{rank}
    \operatorname{col}
    \big(
        \mathbf J_1^{\mathrm d}(\mathbf x_1),\ldots,
        \mathbf J_N^{\mathrm d}(\mathbf x_N)
    \big)
    =
    r,
\]
where \(r:=\sum_{i=1}^{N}r_i\).
\end{assumption}

\begin{assumption}
\label{asm:reg-coupled}
For problem \(\mathcal P_{\mathrm c}\),  on the set  
\(\mathcal X_{\mathrm c}\cap(\mathbf 1_N\otimes\mathcal U_{\mathrm c})\), the aggregate constraint Jacobian has
full row rank, i.e.,
\[
    \operatorname{rank}
   \sum_{i=1}^{N}\mathbf J_i^{\mathrm c}(\mathbf x_i)=q.
\]
\end{assumption}

 
\begin{remark}
\label{rem:reg}
In parallel with the full-rank assumption in the centralized optimization
setting of~\cite{Pirrera2026Global}, Assumptions~\ref{asm:reg-dist} and \ref{asm:reg-coupled} are tailored to distributed scenarios in the sense that they only
require full rank of the equality constraints along the consensus direction,
i.e., no first-order redundancy appears when
\(\xb_1=\xb_2=\cdots=\xb_N\). When the constraint
functions \(h_i\) are affine and have linearly independent rows, the assumption is always satisfied. The assumption fails in singular cases where the constraint map loses first-order information
about the feasible manifold. A typical example is a constraint of the
form \(h(s)=s^2\) at \(s=0\), for which the feasible set is smooth but
the constraint gradient $\nabla h(0) =0$ and therefore cannot characterize the
normal direction.
\end{remark}

\section{Distributed Local Equality Constraints}
\label{sec:local-equality}

In this section, we study the distributed optimization problem with local
equality constraints, $\mathcal{P}_{\rm d}$. We first construct ideal feedback-linearization design
and then develop a singular-perturbation realization that admits a fully
distributed implementation.

\subsection{Ideal Feedback-Linearization Design}
\label{sec:fl-design-dist}

Inspired by \cite{Pirrera2026Global}, we develop an ideal feedback-linearization (FL) design for
problem \(\mathcal P_{\mathrm d}\). 
For simplicity, we 
let \(\mathbf L:=\mathbf L_{\mathcal G}\otimes \mathbf I_d\),
\(f(\mathbf x):=\sum_{i=1}^{N}f_i(\mathbf x_i)\) and
\(\mathbf h^{\mathrm d}(\mathbf x)
:=\operatorname{col}(\mathbf h_1^{\mathrm d}(\mathbf x_1),\ldots,
\mathbf h_N^{\mathrm d}(\mathbf x_N))\).
Then \(\mathcal P_{\mathrm d}\) can be written compactly as
\[
\begin{aligned}
    \min_{\mathbf x\in\mathbb R^{Nd}}\quad
        & f(\mathbf x) \\
    \mathrm{s.t.}\quad
        & \mathbf L\mathbf x=\mathbf 0,\\
        & \mathbf h^{\mathrm d}(\mathbf x)=\mathbf 0 .
\end{aligned}
\]

The Lagrangian associated with the above problem is
\[
    \mathcal L_{\mathrm d}
    (\mathbf x,\bm\lambda,\bm\mu)
    =
    f(\mathbf x)
    +
    \bm\lambda^{\top}\mathbf L\mathbf x
    +
    \bm\mu^{\top}\mathbf h^{\mathrm d}(\mathbf x),
\]
where \(\bm\lambda:=\operatorname{col}(\bm\lambda_1,\ldots,\bm\lambda_N)\in\mathbb R^{Nd}\) is the multiplier associated with the
consensus constraint and
\(\bm\mu:=\operatorname{col}(\bm\mu_1,\ldots,\bm\mu_N)\in\mathbb R^r\) is the multiplier associated with the local equality
constraints. Since \(\mathbf L=\mathbf L^{\top}\), the gradient descent
dynamics of \(\mathcal L_{\mathrm d}\) with respect to \(\mathbf x\) is
\[
    \dot{\mathbf x}
    =
    -\nabla f(\mathbf x)
    -
    \mathbf L\bm\lambda
    -
    \tilde{\mathbf J}^{\mathrm d}(\mathbf x)^{\top}\bm\mu,
\]
where 
\(\tilde{\mathbf J}^{\mathrm d}(\mathbf x)
:=\operatorname{blkdiag}(\mathbf J_1^{\mathrm d}(\mathbf x_1),\ldots,
\mathbf J_N^{\mathrm d}(\mathbf x_N))\).
We let 
\(\mathbf v:=\mathbf L\bm\lambda\in\operatorname{Im}\mathbf L\) to avoid any possible 
$\mathbf{L}^2$ terms later.
Then the primal dynamics becomes
\begin{equation}
\label{eq:fl-x-dynamics-dist}
    \dot{\mathbf x}
    =
    -\nabla f(\mathbf x)
    -
    \mathbf v
    -
    \tilde{\mathbf J}^{\mathrm d}(\mathbf x)^{\top}\bm\mu .
\end{equation}

From a control viewpoint, we regard
\(\mathbf v\in\operatorname{Im}\mathbf L\) and
\(\bm\mu\in\mathbb R^r\) as control inputs, and the output is chosen as the
constraint violation
\begin{equation}
\label{eq:fl-output-dist}
    \mathbf y_{\mathrm d}
    =
    \operatorname{col}(\mathbf y_g,\mathbf y_h)
    :=
    \operatorname{col}
    \big(
    \mathbf L\mathbf x,
    \mathbf h^{\mathrm d}(\mathbf x)
    \big)\in \operatorname{Im}\mathbf L\times\mathbb R^r.
\end{equation}
Here, \(\mathbf y_g\) measures the violation of the consensus constraint, and
\(\mathbf y_h\) measures the violation of the local equality constraints.

Differentiating \eqref{eq:fl-output-dist} along \eqref{eq:fl-x-dynamics-dist} gives
\begin{equation}
\label{eq:fl-output-compact-dist}
    \dot{\mathbf y}_{\mathrm d}
    =
    -
    \mathbf A_{\mathrm d}(\mathbf x)\nabla f(\mathbf x)
    -
    \mathbf M_{\mathrm d}(\mathbf x)\mathbf u_{\mathrm d},
\end{equation}
where \(\mathbf u_{\mathrm d}:=\operatorname{col}(\mathbf v,\bm\mu)\),
\(\mathbf A_{\mathrm d}(\mathbf x)
:=\operatorname{col}(\mathbf L,\tilde{\mathbf J}^{\mathrm d}(\mathbf x))\), $\mathbf B_{\mathrm d}(\mathbf x)=\begin{bmatrix}
    \mathbf{I}&\tilde{\mathbf J}^{\mathrm d}(\mathbf x)^\top
\end{bmatrix}$, and
\[
    \mathbf M_{\mathrm d}(\mathbf x)
    :=\mathbf A_{\mathrm d}(\mathbf x)\mathbf B_{\mathrm d}(\mathbf x)=
    \begin{bmatrix}
        \mathbf L
        &
        \mathbf L\tilde{\mathbf J}^{\mathrm d}(\mathbf x)^{\top}
        \\
        \tilde{\mathbf J}^{\mathrm d}(\mathbf x)
        &
        \tilde{\mathbf J}^{\mathrm d}(\mathbf x)
        \tilde{\mathbf J}^{\mathrm d}(\mathbf x)^{\top}
    \end{bmatrix}.
\]

Let \(\mathbf G_{\mathrm d}(\mathbf y_{\mathrm d})\) be a desired output vector
field, i.e., 
\(
    \dot{\mathbf y}_{\mathrm d}
    =
    \mathbf G_{\mathrm d}(\mathbf y_{\mathrm d}).
\) The output dynamics should be designed such that $\yb_{\rm d}$ converges to zero, meaning all constraints are satisfied. Accordingly, we select the assigned output vector field
\(\mathbf G_{\mathrm d}:
\operatorname{Im}\mathbf L\times\mathbb R^r
\to\operatorname{Im}\mathbf L\times\mathbb R^r\)
to be continuously differentiable and to satisfy
\(\mathbf G_{\mathrm d}(\mathbf 0)=\mathbf 0\). Moreover, the assigned output
dynamics \(\dot{\mathbf y}_{\mathrm d}
=\mathbf G_{\mathrm d}(\mathbf y_{\mathrm d})\) is designed to be globally
exponentially stable at \(\mathbf y_{\mathrm d}=\mathbf 0\) with convergence rate \(\alpha_G>0\).

\begin{remark}
\label{rem:G-choice}
A typical choice satisfying the above design properties is the linear
feedback law
\(
\mathbf G_{\mathrm d}=\operatorname{col}(\mathbf G_g,\mathbf G_h)
    =
    \operatorname{col}( -k_1\mathbf y_g,-k_2\yb_h)
\)
where \(k_1,k_2>0\). In this case, the
convergence rate of each output component can be adjusted. Besides linear feedback,
nonlinear output dynamics can also be adopted. For example,  one may choose \(
    [\mathbf G_{g}(\mathbf y_{g})]_i
    =
    -k_1[\mathbf y_{g}]_i,
\) and
\(
    [\mathbf G_{h}(\mathbf y_{h})]_i
    =
    -k_2[\mathbf y_{h}]_i
    -\gamma[\mathbf y_{h}]_i^3
\) for some $\gamma>0$,
which may provide faster
decay of the corresponding output residual, see numerical simulations in Section~\ref{sec.sim_compare}.
\end{remark}

Combining this desired output
dynamics with \eqref{eq:fl-output-compact-dist}, the ideal feedback-linearizing
input $\ub_{\rm d}$ is designed to satisfy
\begin{equation}
\label{eq:ideal-fl-relation-dist}
    \mathbf M_{\mathrm d}(\mathbf x)\mathbf u_{\mathrm d}
    +
    \mathbf A_{\mathrm d}(\mathbf x)\nabla f(\mathbf x)
    +
    \mathbf G_{\mathrm d}(\mathbf y_{\mathrm d})
    =
    \mathbf 0,
    \qquad
    \mathbf v\in\operatorname{Im}\mathbf L.
\end{equation}
The algebraic relation \eqref{eq:ideal-fl-relation-dist} admits a locally unique solution, as stated in the following lemma.
\begin{lemma}
\label{lem:fl-wp-dist}
Suppose Assumptions~\ref{asm:graph} and \ref{asm:reg-dist} hold. Then, for every \(\mathbf x\) sufficiently close to
\(\mathcal X_{\mathrm d}\cap(\mathbf 1_N\otimes\mathcal U_{\mathrm d})\), the
algebraic relation \eqref{eq:ideal-fl-relation-dist} admits a unique solution
\(\mathbf u_{\mathrm d}=\operatorname{col}(\mathbf v,\bm\mu)\) satisfying
\(\mathbf v\in\operatorname{Im}\mathbf L\).
\end{lemma}

\begin{proof}

We first show that
\(
\operatorname{Im}\tilde{\mathbf J}^{\mathrm d}(\mathbf x)^\top
\cap\operatorname{Im}\mathbf L
=
\{\mathbf 0\}
\).
Suppose that
\(
\tilde{\mathbf J}^{\mathrm d}(\mathbf x)^\top\bm\alpha
\in\operatorname{Im}\mathbf L
\)
for some
\(
\bm\alpha
=
\operatorname{col}(\bm\alpha_1,\ldots,\bm\alpha_N)
\).
Since
\(
\operatorname{Im}\mathbf L
=
\ker(\mathbf 1_N^\top\otimes\mathbf I_d)
\),
premultiplication by
\(
\mathbf 1_N^\top\otimes\mathbf I_d
\)
gives
\(
\sum_{i=1}^{N}
\mathbf J_i^{\mathrm d}(\mathbf x_i)^\top\bm\alpha_i
=
\mathbf 0
\).
By Assumption~\ref{asm:reg-dist}, the matrix
\(
\operatorname{col}
(
\mathbf J_1^{\mathrm d}(\mathbf x_1),\ldots,
\mathbf J_N^{\mathrm d}(\mathbf x_N)
)
\)
has full row rank. Hence its transpose has a trivial kernel, which implies
\(\bm\alpha=\mathbf 0\).

Let
\(
\mathcal A_{\mathrm d}
:=
\operatorname{Im}\mathbf L\times\mathbb R^r
\).
We next prove that the restriction of
\(\mathbf M_{\mathrm d}(\mathbf x)\) to
\(\mathcal A_{\mathrm d}\) is injective. Let
\(
\bar{\mathbf u}_{\mathrm d}
=
\operatorname{col}(\bar{\mathbf v},\bar{\bm\mu})
\in\mathcal A_{\mathrm d}
\)
satisfy
\(
\mathbf M_{\mathrm d}(\mathbf x)\bar{\mathbf u}_{\mathrm d}
=
\mathbf 0
\).
Define
\(
\bar{\mathbf w}
=
\mathbf B_{\mathrm d}(\mathbf x)\bar{\mathbf u}_{\mathrm d}
=
\bar{\mathbf v}
+
\tilde{\mathbf J}^{\mathrm d}(\mathbf x)^\top\bar{\bm\mu}
\).
Since
\(
\mathbf M_{\mathrm d}
=
\mathbf A_{\mathrm d}\mathbf B_{\mathrm d}
\),
we have
\(
\mathbf L\bar{\mathbf w}=\mathbf 0
\)
and
\(
\tilde{\mathbf J}^{\mathrm d}(\mathbf x)\bar{\mathbf w}
=
\mathbf 0
\).
Thus,
\(
\bar{\mathbf w}
\in
\ker\mathbf L
\cap
\ker\tilde{\mathbf J}^{\mathrm d}(\mathbf x)
\).
On the other hand, because
\(\bar{\mathbf v}\in\operatorname{Im}\mathbf L\), one has
\(
\bar{\mathbf w}
\in
\operatorname{Im}\mathbf L
+
\operatorname{Im}\tilde{\mathbf J}^{\mathrm d}(\mathbf x)^\top
\).
Since \(\mathbf L\) is symmetric,
\(
\operatorname{Im}\mathbf L
=
(\ker\mathbf L)^\perp
\),
and therefore
\(
\operatorname{Im}\mathbf L
+
\operatorname{Im}\tilde{\mathbf J}^{\mathrm d}(\mathbf x)^\top
=
\big(
\ker\mathbf L
\cap
\ker\tilde{\mathbf J}^{\mathrm d}(\mathbf x)
\big)^\perp .
\)
Consequently, \(\bar{\mathbf w}\) belongs to a subspace and its orthogonal
complement, and hence \(\bar{\mathbf w}=\mathbf 0\).
It follows that
\(
\bar{\mathbf v}
+
\tilde{\mathbf J}^{\mathrm d}(\mathbf x)^\top\bar{\bm\mu}
=
\mathbf 0
\).
Therefore,
\(
\tilde{\mathbf J}^{\mathrm d}(\mathbf x)^\top\bar{\bm\mu}
\in\operatorname{Im}\mathbf L
\).
By \(
\operatorname{Im}\tilde{\mathbf J}^{\mathrm d}(\mathbf x)^\top
\cap\operatorname{Im}\mathbf L
=
\{\mathbf 0\}
\) established above,
\(\bar{\bm\mu}=\mathbf 0\), and consequently
\(\bar{\mathbf v}=\mathbf 0\). Hence
\(\bar{\mathbf u}_{\mathrm d}=\mathbf 0\), proving injectivity.

Moreover,
\(\mathbf M_{\mathrm d}(\mathbf x)\) maps
\(\mathcal A_{\mathrm d}\) into itself, because the first block of
\(\mathbf M_{\mathrm d}(\mathbf x)\mathbf u_{\mathrm d}\) belongs to
\(\operatorname{Im}\mathbf L\). Since \(\mathcal A_{\mathrm d}\) is
finite-dimensional, the restriction of
\(\mathbf M_{\mathrm d}(\mathbf x)\) to \(\mathcal A_{\mathrm d}\) is
therefore bijective.

Finally,
\(
-\mathbf A_{\mathrm d}(\mathbf x)\nabla f(\mathbf x)
-\mathbf G_{\mathrm d}(\mathbf y_{\mathrm d})
\in\mathcal A_{\mathrm d}
\),
because its first block belongs to \(\operatorname{Im}\mathbf L\).
Thus, \eqref{eq:ideal-fl-relation-dist} admits a unique solution
\(
\mathbf u_{\mathrm d}
=
\operatorname{col}(\mathbf v,\bm\mu)
\in\mathcal A_{\mathrm d}
\),
which completes the proof.
\end{proof}

By Lemma~\ref{lem:fl-wp-dist}, the admissible solution of
\eqref{eq:ideal-fl-relation-dist} is locally unique. We denote this
solution by \(\mathbf u_{\mathrm d}^{\rm q}(\mathbf x)\).
The corresponding ideal FL algorithm is
\begin{equation}
\label{eq:ideal-fl-algorithm-dist}
\begin{aligned}
    \dot{\mathbf x}
    &=
    -\nabla f(\mathbf x)
    -
    \mathbf v
    -
    \tilde{\mathbf J}^{\mathrm d}(\mathbf x)^{\top}\bm\mu,
    \\
   \mathbf u_{\mathrm d}
&=
\mathbf u_{\mathrm d}^{\rm q}(\mathbf x),
\end{aligned}
\end{equation}
where $\mathbf u_{\mathrm d}=\operatorname{col}(\mathbf v,\bm\mu)$.  
The direct implementation of \eqref{eq:ideal-fl-algorithm-dist} may
face difficulty in distributed realization. This issue
will be discussed and resolved in the next section. 
In the remainder of this section, we proceed under the assumption that the
ideal control law \eqref{eq:ideal-fl-algorithm-dist} is implementable.

Under \eqref{eq:ideal-fl-algorithm-dist}, the closed-loop output dynamics is exactly
\(\dot{\mathbf y}_{\mathrm d}
=\mathbf G_{\mathrm d}(\mathbf y_{\mathrm d})\).
It remains to characterize the motion on the zero-output manifold of
\eqref{eq:ideal-fl-algorithm-dist}. In the
present setting, the zero-output manifold is $\mathcal{X}_{\rm d}$. On this manifold, one has \(\mathbf y_{\mathrm d}=\mathbf 0\). Since
\(\mathbf G_{\mathrm d}(\mathbf 0)=\mathbf 0\), the ideal FL law enforces
\(\dot{\mathbf y}_{\mathrm d}=\mathbf 0\). Thus, 
\(\mathcal X_{\mathrm d}\) is a locally invariant manifold of
\eqref{eq:ideal-fl-algorithm-dist}. 
On the zero-output manifold \(\mathcal X_{\mathrm d}\), every state can be
written as
\(
    \mathbf x=\mathbf 1_N\otimes \mathbf s\) for 
    \(\mathbf s\in\Omega_{\mathrm d}.
\)
Moreover, the corresponding tangent velocity satisfies
\(\dot{\mathbf x}=\mathbf 1_N\otimes \dot{\mathbf s}\), with
\(\dot{\mathbf s}\in T_{\mathbf s}\Omega_{\mathrm d}\). Since the ideal FL
input cancels the normal component of \(-\nabla f(\mathbf x)\) relative to
\(\mathcal X_{\mathrm d}\), the dynamics of $\xb(t)$ on $\mathcal{X}_{\rm d}$ can be written equivalently as the
projected gradient flow
\[
    \dot{\mathbf s}
    =
    -\frac{1}{N}
    \Pi_{T_{\mathbf s}\Omega_{\mathrm d}}
    \nabla F(\mathbf s).
\]
This dynamics is the projected gradient flow of the centralized
objective \(F\) on the feasible manifold \(\Omega_{\mathrm d}\). By Assumption~\ref{asm:growth},
\(F\) has quadratic growth on
\(\Omega_{\mathrm d}\cap\mathcal U_{\mathrm d}\) around
\(\mathbf s_{\mathrm d}^{\ast}\). Hence
\(\mathbf s(t)\) locally  converges to
\(\mathbf s_{\mathrm d}^{\ast}\), and consequently
\(\mathbf x(t)=\mathbf 1_N\otimes\mathbf s(t)\) locally 
converges to
\(\mathbf x_{\mathrm d}^{\ast}
=\mathbf 1_N\otimes\mathbf s_{\mathrm d}^{\ast}\)
along the zero-output manifold \(\mathcal X_{\mathrm d}\).

Overall, in dynamics~\eqref{eq:ideal-fl-algorithm-dist}, the control law $\ub_{\rm d}$ is designed such that the output dynamics is governed by
\(\dot{\mathbf y}_{\mathrm d}
=\mathbf G_{\mathrm d}(\mathbf y_{\mathrm d})\).
When the output is driven to zero, the state trajectory $\xb(t)$ remains on the zero-output manifold $\mathcal{X}_{\rm d}$, and its motion on this manifold converges to a local optimizer \(\mathbf x_{\mathrm d}^{\ast}\). Formally, the following proposition
describes the convergence property of the ideal feedback-linearized
dynamics \eqref{eq:ideal-fl-algorithm-dist}.

\begin{proposition}
\label{prop:fl-dist}
Consider problem \(\mathcal P_{\mathrm d}\). Suppose
Assumptions~\ref{asm:graph}--\ref{asm:reg-dist} hold. Then
\(\mathbf x_{\mathrm d}^{\ast}\) is a locally exponentially stable equilibrium
of the ideal FL dynamics \eqref{eq:ideal-fl-algorithm-dist}. Specifically, there
exist constants \(C_{\mathrm{FL}}>0\), \(\alpha_{\mathrm{FL}}>0\), and
\(\delta_{\mathrm{FL}}>0\) such that, for every initial condition satisfying
\(\|\mathbf x(0)-\mathbf x_{\mathrm d}^{\ast}\|<\delta_{\mathrm{FL}}\), the
corresponding solution satisfies
\[
    \|\mathbf x(t)-\mathbf x_{\mathrm d}^{\ast}\|
    +
    \|\mathbf L\mathbf x(t)\|
    +
    \|\mathbf h^{\mathrm d}(\mathbf x(t))\|
    \le
    C_{\mathrm{FL}}e^{-\alpha_{\mathrm{FL}}t}
    \|\mathbf x(0)-\mathbf x_{\mathrm d}^{\ast}\|,
    \qquad
    t\ge0 ,
\]
with admissible rate
\(\alpha_{\mathrm{FL}}\in
(0,\min\{\rho_{\mathrm d}/N,\alpha_G\})\).
\end{proposition}

\begin{remark}
\label{rem:fl-rate}
As shown in Appendix~\ref{app:fl-dist}, the exponential convergence rate
\(\alpha_{\mathrm{FL}}\) in Proposition~\ref{prop:fl-dist} can be chosen as
\(
    \alpha_{\mathrm{FL}}
    \in
    \left(
        0,
        \min\left\{
            \frac{\rho_{\mathrm d}}{N},
            \alpha_G
        \right\}
    \right),
\)
where \(\alpha_G\) denotes the exponential convergence rate of the assigned
output dynamics \(\dot{\mathbf y}=\mathbf G_{\rm d}(\mathbf y)\). In particular, if the output feedback is chosen as the proportional law \(\mathbf G_{\rm d}(\mathbf y)=-k\mathbf y\), then \(\alpha_G=k\).  It is worth noting
that the value of
\(\alpha_{\mathrm{FL}}\) depends on how close \(\mathbf x(t)\) is to
\(\mathbf x_{\mathrm d}^{\ast}\). When \(\delta_{\mathrm{FL}}>0\) in Proposition \ref{prop:fl-dist} is sufficiently small, the trajectory remains in a
sufficiently small neighborhood of \(\mathbf x_{\mathrm d}^{\ast}\). In this
case, \(\alpha_{\mathrm{FL}}\) can be made arbitrarily close to
\(
    \min\left\{
        \frac{\rho_{\mathrm d}}{N},
        \alpha_G
    \right\}.
\) 
Therefore, overall, the convergence rate of \eqref{eq:ideal-fl-algorithm-dist} is
determined by both the restricted quadratic growth rate of \(F\) and the designed output dynamics.

\end{remark}

\subsection{Distributed Singular-Perturbation Realization}
\label{sec:sp-realization-dist}
 
The ideal FL law in Section~\ref{sec:fl-design-dist} prescribes the closed-loop
output dynamics
\(\dot{\mathbf y}_{\mathrm d}=\mathbf G_{\mathrm d}(\mathbf y_{\mathrm d})\)
through the algebraic relation~\eqref{eq:ideal-fl-relation-dist}. In this section, we first explain the obstruction that prevents it from being implemented in a distributed manner, and
then propose a singular-perturbation (SP) realization that approximately recovers the ideal
behavior while allowing for a fully distributed implementation.

Recall the ideal FL relation~\eqref{eq:ideal-fl-relation-dist}. In this dynamics, \eqref{eq:ideal-fl-relation-dist} must be solved for
\(\mathbf u_{\mathrm d}=\operatorname{col}(\mathbf v,\bm\mu)\) at every state
\(\mathbf x\), subject to \(\mathbf v\in\operatorname{Im}\mathbf L\). However, evaluating the unique admissible solution
\(\mathbf u_{\mathrm d}^{\rm q}(\mathbf x)\) requires global coupling
information and therefore cannot be performed in a simple distributed
manner.
 
To circumvent the explicit inversion, we replace the algebraic
solution of \eqref{eq:ideal-fl-relation-dist} by an auxiliary dynamical system that
drives the input \(\mathbf u_{\mathrm d}\) so that
\(\mathbf M_{\mathrm d}(\mathbf x)\mathbf u_{\mathrm d}\) rapidly tracks the
target \(-\mathbf A_{\mathrm d}(\mathbf x)\nabla f(\mathbf x)-\mathbf G_{\mathrm d}(\mathbf y_{\mathrm d})\). Define the FL {residual}
\[
    \mathbf e_{\rm d}
    :=
    \mathbf M_{\mathrm d}(\mathbf x)\mathbf u_{\mathrm d}
    +
    \mathbf A_{\mathrm d}(\mathbf x)\nabla f(\mathbf x)
    +
    \mathbf G_{\mathrm d}(\mathbf y_{\mathrm d}),
\]
so that \(\mathbf e_{\rm d}=\mathbf 0\) recovers exactly the ideal relation
\eqref{eq:ideal-fl-relation-dist}, and, in view of
\eqref{eq:fl-output-compact-dist}, \(\mathbf e_{\rm d}\) coincides with the
mismatch between the actual output rate \(\dot{\mathbf y}_{\mathrm d}\) and the
assigned one \(\mathbf G_{\mathrm d}(\mathbf y_{\mathrm d})\), i.e.
\(\mathbf e_{\rm d}
=-\dot{\mathbf y}_{\mathrm d}+\mathbf G_{\mathrm d}(\mathbf y_{\mathrm d})\).
We then assign to \(\mathbf u_{\mathrm d}\) the fast gradient-type dynamics
\begin{equation}
\label{eq:fast-u-dist}
    \tau\,\dot{\mathbf u}_{\mathrm d}
    =
    -\mathbf K_0\,\mathbf e_{\rm d}
    =
    -\mathbf K_0
    \big(
        \mathbf M_{\mathrm d}(\mathbf x)\mathbf u_{\mathrm d}
        +\mathbf A_{\mathrm d}(\mathbf x)\nabla f(\mathbf x)
        +\mathbf G_{\mathrm d}(\mathbf y_{\mathrm d})
    \big),
\end{equation}
where \(0<\tau\ll1\) is a singular-perturbation parameter that separates the
time scales, \(\mathbf K_0=\operatorname{blkdiag}(K_{0c}\mathbf I,\,
K_{0h}\mathbf I)\succ0\) is a constant gain, and the initial condition is $\mathbf v(0)\in\operatorname{Im}\mathbf L$. System
\eqref{eq:fast-u-dist} together with the primal flow
\eqref{eq:fl-x-dynamics-dist} forms a singularly perturbed
interconnection
\begin{equation}
\label{eq:fast-algorithm-dist}
\begin{aligned}
    \dot{\mathbf x}
    &=
    -\nabla f(\mathbf x)
    -
    \mathbf v
    -
    \tilde{\mathbf J}^{\mathrm d}(\mathbf x)^{\top}\bm\mu,
    \\
     \tau\,\dot{\mathbf u}_{\mathrm d}
   & =
    -\mathbf K_0
    \big(
        \mathbf M_{\mathrm d}(\mathbf x)\mathbf u_{\mathrm d}
        +\mathbf A_{\mathrm d}(\mathbf x)\nabla f(\mathbf x)
        +\mathbf G_{\mathrm d}(\mathbf y_{\mathrm d})
    \big),
\end{aligned}
\end{equation}
where the slow state is \(\mathbf x\) and the fast state is
\(\mathbf u_{\mathrm d}\), and the initial condition is $\mathbf v(0)\in\operatorname{Im}\mathbf L$.
With \(\mathbf x\) frozen, the boundary-layer system of
\eqref{eq:fast-u-dist} has the solution set of \eqref{eq:ideal-fl-relation-dist} as
its quasi-steady-state manifold. As \(\tau\to0\), the slow dynamics reduces to
the ideal FL algorithm~\eqref{eq:ideal-fl-algorithm-dist}. 

To ensure distributed implementability, we further select
\(\mathbf G_{\mathrm d}=\operatorname{col}(\mathbf G_g,\mathbf G_h)\)
to be block-separable, i.e.,
\(\mathbf G_g(\mathbf y_g)
=\operatorname{col}(\mathbf G_{g,1}(\mathbf y_{g,1}),\ldots,
\mathbf G_{g,N}(\mathbf y_{g,N}))\) and
\(\mathbf G_h(\mathbf y_h)
=\operatorname{col}(\mathbf G_{h,1}(\mathbf y_{h,1}),\ldots,
\mathbf G_{h,N}(\mathbf y_{h,N}))\).
With this structure, \eqref{eq:fast-algorithm-dist} only relies on local and
neighboring information and can therefore be implemented in a distributed
manner.

Furthermore, to avoid the term $\mathbf A_{\mathrm d}(\mathbf x)\nabla f(\mathbf x)$ which involves exchanging $\nabla f_i(\xb_i)$ between nodes, we introduce the
 change of variables
\[
    \mathbf z:=\mathbf u_{\mathrm d}-\mathbf K_p\mathbf y_{\mathrm d},
    \qquad
    \mathbf K_p:=\mathbf K_0/\tau,
\]
and differentiating along \eqref{eq:fast-algorithm-dist} cancels this term and yields
\(\dot{\mathbf z}=-\mathbf K_p\mathbf G_{\mathrm d}(\mathbf y_{\mathrm d})\).
Writing \(\mathbf z=\operatorname{col}(\mathbf z_c,\mathbf z_h)\),
\(K_{pc}:=K_{0c}/\tau\), \(K_{ph}:=K_{0h}/\tau\), and recalling
\(\mathbf y_{\mathrm d}=\operatorname{col}(\mathbf L\mathbf x,
\mathbf h^{\mathrm d}(\mathbf x))\), \(\mathbf u_{\mathrm d}
=\operatorname{col}(\mathbf v,\bm\mu)\), \eqref{eq:fast-algorithm-dist} is equivalent to
\begin{equation}
\label{eq:sp-compact-dist}
    \begin{aligned}
        \dot{\mathbf x}
            &=-\nabla f(\mathbf x)-\mathbf v
              -\tilde{\mathbf J}^{\mathrm d}(\mathbf x)^{\top}\bm\mu,\\
        \mathbf v
            &=K_{pc}\mathbf L\mathbf x+\mathbf z_c,\\
        \bm\mu
            &=K_{ph}\mathbf h^{\mathrm d}(\mathbf x)+\mathbf z_h,\\
        \dot{\mathbf z}_c
            &=-K_{pc}\mathbf G_g(\mathbf L\mathbf x),\\
        \dot{\mathbf z}_h
            &=-K_{ph}\mathbf G_h(\mathbf h^{\mathrm d}(\mathbf x)).
    \end{aligned}
\end{equation}
where the initial condition is $\zb_c(0)\in\operatorname{Im}\mathbf L$.  We refer to \eqref{eq:sp-compact-dist} as the SP realization of the ideal FL dynamics \eqref{eq:ideal-fl-algorithm-dist}.
 
By the block structure of \(\mathbf{G}_{\rm d}\),
\(\mathbf L\), \(\tilde{\mathbf J}^{\mathrm d}\), and
\(\mathbf h^{\mathrm d}\), realization~\eqref{eq:sp-compact-dist} is fully
distributed. Each agent \(i\in\mathcal V\) maintains
\(\mathbf x_i,\mathbf z_{c,i}\in\mathbb R^d\),
\(\mathbf z_{h,i}\in\mathbb R^{r_i}\), and runs
\begin{equation}
\label{eq:node-level-dist}
    \begin{aligned}
        \dot{\mathbf x}_i
            &=-\nabla f_i(\mathbf x_i)-\mathbf v_i
              -\mathbf J_i^{\mathrm d}(\mathbf x_i)^{\top}\bm\mu_i,\\
        \mathbf v_i
            &=K_{pc}\!\sum_{j\in\mathcal N_i}\!a_{ij}(\mathbf x_i-\mathbf x_j)
              +\mathbf z_{c,i},\\
        \bm\mu_i
            &=K_{ph}\mathbf h_i^{\mathrm d}(\mathbf x_i)+\mathbf z_{h,i},\\
        \dot{\mathbf z}_{c,i}
            &=-K_{pc}\mathbf G_{g,i}\!\Big(
               \textstyle\sum_{j\in\mathcal N_i}\!a_{ij}(\mathbf x_i-\mathbf x_j)\Big),\\
        \dot{\mathbf z}_{h,i}
            &=-K_{ph}\mathbf G_{h,i}\big(\mathbf h_i^{\mathrm d}(\mathbf x_i)\big),
    \end{aligned}
\end{equation}
using only its own data and the variables \(\mathbf x_j\) of its neighbors
\(j\in\mathcal N_i\).


The following theorem
states that for sufficiently strong time-scale separation the SP realization
\eqref{eq:sp-compact-dist} preserves the local exponential convergence of the
ideal FL dynamics \eqref{eq:ideal-fl-algorithm-dist}.

\begin{theorem}
\label{thm:sp-dist}
Consider problem \(\mathcal P_{\mathrm d}\). Suppose
Assumptions~\ref{asm:graph}--\ref{asm:reg-dist} hold. Then \eqref{eq:sp-compact-dist} admits a locally unique equilibrium \((\mathbf x_{\mathrm d}^{\ast},\zb^\ast)\), which is locally exponentially stable. Specifically, let
\(\alpha_{\mathrm{FL}}\) be any convergence rate admissible in
Proposition~\ref{prop:fl-dist}. For any
\(\alpha_{\mathrm{SP}}\in(0,\alpha_{\mathrm{FL}})\), there exists
\(\tau^\ast>0\) such that, for every fixed
\(\tau\in(0,\tau^\ast)\), there exist constants
\(C_{\mathrm{SP}}>0\) and \(\delta_{\mathrm{SP}}>0\), such that whenever
\(
    \|\mathbf x(0)-\mathbf x_{\mathrm d}^{\ast}\|
    +
    \|\mathbf z(0)-\mathbf z^\ast\|
    <
    \delta_{\mathrm{SP}},
\)
the corresponding
solution satisfies
\begin{equation}
\label{eq:sp-bound-dist}
\begin{aligned}
    &\|\mathbf x(t)-\mathbf x_{\mathrm d}^{\ast}\|
    +\|\mathbf z(t)-\mathbf z^\ast\|
    +\|\mathbf L\mathbf x(t)\|
    +\|\mathbf h^{\mathrm d}(\mathbf x(t))\|  \\
    &\qquad\le
    C_{\mathrm{SP}}e^{-\alpha_{\mathrm{SP}}t}
    \left(
        \|\mathbf x(0)-\mathbf x_{\mathrm d}^{\ast}\|
        +
        \|\mathbf z(0)-\mathbf z^\ast\|
    \right),
    \qquad t\ge0 .
\end{aligned}
\end{equation}
\end{theorem}

\section{Coupled Equality Constraints}
\label{sec:coupled-equality}
In this section, we extend the preceding design to distributed optimization
with coupled equality constraints,  $\mathcal{P}_{\rm c}$. By introducing auxiliary variables, the
coupled constraint is lifted into a distributed form, to which the
feedback-linearization and singular-perturbation designs can be applied.

\subsection{Lifted Feedback-Linearization Design}
\label{sec:fl-design-coupled}

Let 
\(\mathbf L_q:=\mathbf L_{\mathcal G}\otimes\mathbf I_q\), and define the
stacked constraint map
\(
    \mathbf h^{\mathrm c}(\mathbf x)
    :=\operatorname{col}
    (\mathbf h_1^{\mathrm c}(\mathbf x_1),\ldots,
    \mathbf h_N^{\mathrm c}(\mathbf x_N))\in\mathbb R^{Nq}
\).
Introduce an auxiliary variable
\(\bm\zeta:=\operatorname{col}(\bm\zeta_1,\ldots,\bm\zeta_N)\in\mathbb R^{Nq}\),
one block \(\bm\zeta_i\in\mathbb R^{q}\) per agent, and consider the lifted
problem
\begin{equation}
\label{eq:lifted-problem-coupled}
    \begin{aligned}
    \min_{\mathbf x,\bm\zeta}\quad
        & f(\mathbf x):=\sum_{i=1}^N f_i(\mathbf x_i),\\
    \mathrm{s.t.}\quad
        & \mathbf L\mathbf x=\mathbf 0,\\
        & \mathbf h^{\mathrm c}(\mathbf x)+\mathbf L_q\bm\zeta=\mathbf 0.
    \end{aligned}
\end{equation}
The problem \(\mathcal P_{\mathrm c}\) can be solved if we can solve \eqref{eq:lifted-problem-coupled}, as stated in the following lemma. Similar results can be found in the literature, e.g., in
Lemma~1 of \citet{CCDC}.
\begin{lemma}
\label{lem:lift-equiv}
A pair \((\mathbf x,\bm\zeta)\) is a local minimizer of
\eqref{eq:lifted-problem-coupled} if and only if \(\mathbf x\) is a local
minimizer of \(\mathcal P_{\mathrm c}\) and
\(\mathbf h^{\mathrm c}(\mathbf x)+\mathbf L_q\bm\zeta=\mathbf 0\).
\end{lemma}

\begin{proof}
Under Assumption~\ref{asm:graph}, one has
\(\operatorname{Im}\mathbf L_q
=\ker(\mathbf 1_N^\top\otimes\mathbf I_q)\).
Therefore,
\(\sum_{i=1}^N\mathbf h_i^{\mathrm c}(\mathbf x_i)=\mathbf 0\)
if and only if
\((\mathbf 1_N^\top\otimes\mathbf I_q)
\mathbf h^{\mathrm c}(\mathbf x)=\mathbf 0\),
which is equivalent to
\(\mathbf h^{\mathrm c}(\mathbf x)\in\operatorname{Im}\mathbf L_q\).
Hence, the coupled constraint holds if and only if there exists
\(\bm\zeta\in\mathbb R^{Nq}\) such that
\(\mathbf h^{\mathrm c}(\mathbf x)+\mathbf L_q\bm\zeta=\mathbf 0\).
Moreover, since the objective function is independent
of \(\bm\zeta\), the local optimality of \(\mathbf x\) is therefore equivalent
to that of \((\mathbf x,\bm\zeta)\).
\end{proof}

The Lagrangian associated with~\eqref{eq:lifted-problem-coupled} is
\[
    \mathcal L_{\mathrm c}
    (\mathbf x,\bm\zeta,\bm\lambda,\bm\mu)
    =
    f(\mathbf x)
    +
    \bm\lambda^{\top}\mathbf L\mathbf x
    +
    \bm\mu^{\top}
    \big(\mathbf h^{\mathrm c}(\mathbf x)+\mathbf L_q\bm\zeta\big),
\]
where \(\bm\lambda\in\mathbb R^{Nd}\) is the multiplier associated with the
consensus constraint and \(\bm\mu\in\mathbb R^{Nq}\) is the multiplier
associated with the lifted coupling constraint. As in
Section~\ref{sec:fl-design-dist}, we let \(\mathbf v:=\mathbf L\bm\lambda
\in\operatorname{Im}\mathbf L\) to avoid \(\mathbf L^2\) terms. The gradient
descent dynamics of \(\mathcal L_{\mathrm c}\) with respect to the primal
variables \((\mathbf x,\bm\zeta)\) then reads
\begin{equation}
\label{eq:primal-flow-coupled}
    \begin{aligned}
    \dot{\mathbf x}
        &=-\nabla f(\mathbf x)-\mathbf v-\tilde{\mathbf J}^{\mathrm c}(\mathbf x)^\top\bm\mu,\\
    \dot{\bm\zeta}
        &=-\mathbf L_q\bm\mu.
    \end{aligned}
\end{equation}
where $\tilde{\mathbf J}^{\mathrm c}(\mathbf x):=\operatorname{blkdiag}(\mathbf J_1^{\mathrm c}(\mathbf x_1),\ldots,\mathbf J_N^{\mathrm c}(\mathbf x_N))$ the initial condition is taken as
\(\bm\zeta(0)\in\operatorname{Im}\mathbf L_q\), which is imposed only for
convenience, as the component of \(\bm\zeta\) along
\(\ker\mathbf L_q=\mathbf 1_N\otimes\mathbb R^q\) is uncontrollable and does not
affect the remaining dynamics.
Defining
\(\bm\chi:=\operatorname{col}(\mathbf x,\bm\zeta)\in\mathbb R^{N(d+q)}\) and the
inputs \(\mathbf u_{\mathrm c}:=\operatorname{col}(\mathbf v,\bm\mu)\). Moreover,
we take the output to be the constraint violation
\begin{equation}
\label{eq:fl-output-coupled}
    \mathbf y_{\mathrm c}
    =
    \operatorname{col}(\mathbf y_g,\mathbf y_h)
    :=
    \operatorname{col}
    \big(
        \mathbf L\mathbf x,\,
        \mathbf h^{\mathrm c}(\mathbf x)+\mathbf L_q\bm\zeta
    \big)
    \in\operatorname{Im}\mathbf L\times\mathbb R^{Nq},
\end{equation}
where \(\mathbf y_g\) measures the consensus violation and \(\mathbf y_h\)
measures the violation of the lifted coupling constraint. Differentiating
\eqref{eq:fl-output-coupled} along \eqref{eq:primal-flow-coupled} gives
\begin{equation}
\label{eq:fl-output-compact-coupled}
    \dot{\mathbf y}_{\mathrm c}
    =
    -\mathbf A_{\mathrm c}(\bm\chi)\nabla\bar f(\bm\chi)
    -\mathbf M_{\mathrm c}(\bm\chi)\mathbf u_{\mathrm c},
\end{equation}
where \(\bar f(\bm\chi):=f(\mathbf x)\), so that
\(\nabla\bar f(\bm\chi)=\operatorname{col}(\nabla f(\mathbf x),\mathbf 0)\),
\(
    \mathbf A_{\mathrm c}(\bm\chi)
    :=
    \begin{bmatrix}
        \mathbf L & \mathbf 0\\
        \tilde{\mathbf J}^{\mathrm c}(\mathbf x) & \mathbf L_q
    \end{bmatrix}
\), $ \mathbf B_{\mathrm c}(\bm\chi)
    :=
    \begin{bmatrix}
        \mathbf I_{Nd} & \tilde{\mathbf J}^{\mathrm c}(\mathbf x)^\top\\
        \mathbf 0 & \mathbf L_q
    \end{bmatrix}$,
 and
\[
    \mathbf M_{\mathrm c}(\bm\chi)
    :=
    \mathbf A_{\mathrm c}(\bm\chi)\mathbf B_{\mathrm c}(\bm\chi)
    =
    \begin{bmatrix}
        \mathbf L
        &
        \mathbf L\tilde{\mathbf J}^{\mathrm c}(\mathbf x)^\top\\
        \tilde{\mathbf J}^{\mathrm c}(\mathbf x)
        &
        \tilde{\mathbf J}^{\mathrm c}(\mathbf x)\tilde{\mathbf J}^{\mathrm c}(\mathbf x)^\top+\mathbf L_q^2
    \end{bmatrix}.
\]

Let \(\mathbf G_{\mathrm c}\) be the desired output vector field, i.e.,
\(\dot{\mathbf y}_{\mathrm c}
=\mathbf G_{\mathrm c}(\mathbf y_{\mathrm c})\).
Similar to \(\mathbf G_{\mathrm d}\), we select
\(\mathbf G_{\mathrm c}:
\operatorname{Im}\mathbf L\times\mathbb R^{Nq}
\to\operatorname{Im}\mathbf L\times\mathbb R^{Nq}\)
to be continuously differentiable and to satisfy
\(\mathbf G_{\mathrm c}(\mathbf 0)=\mathbf 0\). Moreover, the assigned output
dynamics is designed to be globally exponentially stable at
\(\mathbf y_{\mathrm c}=\mathbf 0\) with
convergence rate \(\alpha_G>0\).

The ideal
feedback-linearizing input is then required to satisfy
\begin{equation}
\label{eq:ideal-fl-relation-coupled}
    \mathbf M_{\mathrm c}(\bm\chi)\mathbf u_{\mathrm c}
    +
    \mathbf A_{\mathrm c}(\bm\chi)\nabla\bar f(\bm\chi)
    +
    \mathbf G_{\mathrm c}(\mathbf y_{\mathrm c})
    =
    \mathbf 0,
    \qquad
    \mathbf v\in\operatorname{Im}\mathbf L.
\end{equation}
Defining the zero-output manifold of
\eqref{eq:ideal-fl-algorithm-coupled} as
\[
    \mathcal Z_{\mathrm c}
    :=
    \{(\mathbf x,\bm\zeta):\mathbf L\mathbf x=\mathbf 0,\ 
    \mathbf h^{\mathrm c}(\mathbf x)+\mathbf L_q\bm\zeta=\mathbf 0\},
\] 
\eqref{eq:ideal-fl-relation-coupled} always admits a locally unique solution close to the zero-output manifold, given by the following lemma. 
\begin{lemma}
\label{lem:fl-wp-coupled}
Suppose Assumptions~\ref{asm:graph} and \ref{asm:reg-coupled} hold. Then, for every \(\bm\chi\) sufficiently close to
\(\mathcal Z_{\mathrm c}\cap(\mathbf 1_N\otimes\mathcal U_{\mathrm c}
\times\mathbb R^{Nq})\), the algebraic relation
\eqref{eq:ideal-fl-relation-coupled} admits a unique solution
\(\mathbf u_{\mathrm c}=\operatorname{col}(\mathbf v,\bm\mu)\) satisfying
\(\mathbf v\in\operatorname{Im}\mathbf L\).
\end{lemma}

\begin{proof}
We first show that
\(
\operatorname{Im}
\begin{bmatrix}
\tilde{\mathbf J}^{\mathrm c}(\mathbf x)^\top\\
\mathbf L_q
\end{bmatrix}
\cap
\operatorname{Im}
\begin{bmatrix}
\mathbf L\\
\mathbf 0
\end{bmatrix}
=
\{\mathbf 0\}
\).
Suppose that
\(
\operatorname{col}
\big(
\tilde{\mathbf J}^{\mathrm c}(\mathbf x)^\top\bm\alpha,
\mathbf L_q\bm\alpha
\big)
\in
\operatorname{Im}\operatorname{col}(\mathbf L,\mathbf 0)
\)
for some
\(
\bm\alpha
=
\operatorname{col}(\bm\alpha_1,\ldots,\bm\alpha_N)
\).
Then
\(
\mathbf L_q\bm\alpha=\mathbf 0
\),
which implies
\(
\bm\alpha=\mathbf 1_N\otimes\bm\beta
\)
for some
\(
\bm\beta\in\mathbb R^q
\).
Moreover,
\(
\tilde{\mathbf J}^{\mathrm c}(\mathbf x)^\top\bm\alpha
\in\operatorname{Im}\mathbf L
\).
Since
\(
\operatorname{Im}\mathbf L
=
\ker(\mathbf 1_N^\top\otimes\mathbf I_d)
\),
premultiplication by
\(
\mathbf 1_N^\top\otimes\mathbf I_d
\)
gives
\(
\sum_{i=1}^{N}
\mathbf J_i^{\mathrm c}(\mathbf x_i)^\top\bm\beta
=
\mathbf 0
\).
By Assumption~\ref{asm:reg-coupled}, the matrix
\(
\sum_{i=1}^{N}\mathbf J_i^{\mathrm c}(\mathbf x_i)
\)
has full row rank. Hence its transpose has a trivial kernel, which implies
\(
\bm\beta=\mathbf 0
\),
and consequently
\(
\bm\alpha=\mathbf 0
\).

Let
\(
\mathcal A_{\mathrm c}
:=
\operatorname{Im}\mathbf L\times\mathbb R^{Nq}
\).
We next prove that the restriction of
\(
\mathbf M_{\mathrm c}(\bm\chi)
\)
to
\(
\mathcal A_{\mathrm c}
\)
is injective. Let
\(
\bar{\mathbf u}_{\mathrm c}
=
\operatorname{col}(\bar{\mathbf v},\bar{\bm\mu})
\in\mathcal A_{\mathrm c}
\)
satisfy
\(
\mathbf M_{\mathrm c}(\bm\chi)\bar{\mathbf u}_{\mathrm c}
=
\mathbf 0
\).
Define
\(
\bar{\mathbf w}
=
\mathbf B_{\mathrm c}(\bm\chi)\bar{\mathbf u}_{\mathrm c}
=
\operatorname{col}
\big(
\bar{\mathbf v}
+
\tilde{\mathbf J}^{\mathrm c}(\mathbf x)^\top\bar{\bm\mu},
\mathbf L_q\bar{\bm\mu}
\big)
\).
Since
\(
\mathbf M_{\mathrm c}
=
\mathbf A_{\mathrm c}\mathbf B_{\mathrm c}
\),
we have
\(
\mathbf A_{\mathrm c}(\bm\chi)\bar{\mathbf w}
=
\mathbf 0
\).
Thus,
\(
\bar{\mathbf w}
\in
\ker\mathbf A_{\mathrm c}(\bm\chi)
\).
On the other hand, because
\(
\bar{\mathbf v}\in\operatorname{Im}\mathbf L
\),
there exists
\(
\bm\xi
\)
such that
\(
\bar{\mathbf v}=\mathbf L\bm\xi
\).
Hence
\(
\bar{\mathbf w}
=
\mathbf A_{\mathrm c}(\bm\chi)^\top
\operatorname{col}(\bm\xi,\bar{\bm\mu})
\),
and therefore
\(
\bar{\mathbf w}
\in
\operatorname{Im}\mathbf A_{\mathrm c}(\bm\chi)^\top
=
(\ker\mathbf A_{\mathrm c}(\bm\chi))^\perp
\).
Consequently, \(\bar{\mathbf w}\) belongs to a subspace and its orthogonal
complement, and hence
\(
\bar{\mathbf w}=\mathbf 0
\).
It follows that
\(
\bar{\mathbf v}
+
\tilde{\mathbf J}^{\mathrm c}(\mathbf x)^\top\bar{\bm\mu}
=
\mathbf 0
\)
and
\(
\mathbf L_q\bar{\bm\mu}
=
\mathbf 0
\).
Therefore,
\(
\operatorname{col}
\big(
\tilde{\mathbf J}^{\mathrm c}(\mathbf x)^\top\bar{\bm\mu},
\mathbf L_q\bar{\bm\mu}
\big)
=
\operatorname{col}(-\bar{\mathbf v},\mathbf 0)
\in
\operatorname{Im}\operatorname{col}(\mathbf L,\mathbf 0)
\).
By
\(
\operatorname{Im}
\begin{bmatrix}
\tilde{\mathbf J}^{\mathrm c}(\mathbf x)^\top\\
\mathbf L_q
\end{bmatrix}
\cap
\operatorname{Im}
\begin{bmatrix}
\mathbf L\\
\mathbf 0
\end{bmatrix}
=
\{\mathbf 0\}
\)
established above,
\(
\bar{\bm\mu}=\mathbf 0
\),
and consequently
\(
\bar{\mathbf v}=\mathbf 0
\).
Hence
\(
\bar{\mathbf u}_{\mathrm c}=\mathbf 0
\),
proving injectivity.

Moreover,
\(
\mathbf M_{\mathrm c}(\bm\chi)
\)
maps
\(
\mathcal A_{\mathrm c}
\)
into itself, because the first block of
\(
\mathbf M_{\mathrm c}(\bm\chi)\mathbf u_{\mathrm c}
\)
belongs to
\(
\operatorname{Im}\mathbf L
\).
Since
\(
\mathcal A_{\mathrm c}
\)
is finite-dimensional, the restriction of
\(
\mathbf M_{\mathrm c}(\bm\chi)
\)
to
\(
\mathcal A_{\mathrm c}
\)
is therefore bijective.

Finally,
\(
-\mathbf A_{\mathrm c}(\bm\chi)\nabla\bar f(\bm\chi)
-\mathbf G_{\mathrm c}(\mathbf y_{\mathrm c})
\in\mathcal A_{\mathrm c}
\),
because its first block belongs to
\(
\operatorname{Im}\mathbf L
\).
Thus, \eqref{eq:ideal-fl-relation-coupled} admits a unique solution
\(
\mathbf u_{\mathrm c}
=
\operatorname{col}(\mathbf v,\bm\mu)
\in\mathcal A_{\mathrm c}
\),
which completes the proof.
\end{proof}

By Lemma~\ref{lem:fl-wp-coupled}, the admissible solution of
\eqref{eq:ideal-fl-relation-coupled} is locally unique. We denote this
solution by \(\mathbf u_{\mathrm c}^{\rm q}(\bm\chi)\).
The corresponding ideal FL algorithm is
\begin{equation}
\label{eq:ideal-fl-algorithm-coupled}
    \begin{aligned}
    \dot{\bm\chi}
        &=
        -\nabla\bar f(\bm\chi)
        -\mathbf B_{\mathrm c}(\bm\chi)\mathbf u_{\mathrm c},\\
 \mathbf u_{\mathrm c}
&=
\mathbf u_{\mathrm c}^{\rm q}(\bm\chi),
    \end{aligned}
\end{equation}
with \(\mathbf u_{\mathrm c}=\operatorname{col}(\mathbf v,\bm\mu)\),
\(\mathbf y_{\mathrm c}=\operatorname{col}(\mathbf L\mathbf x,
\mathbf h^{\mathrm c}(\mathbf x)+\mathbf L_q\bm\zeta)\), and initial condition
\(\bm\zeta(0)\in\operatorname{Im}\mathbf L_q\).

Similar to the analysis for \eqref{eq:ideal-fl-algorithm-dist} in Section~\ref{sec:fl-design-dist}, under \eqref{eq:ideal-fl-relation-coupled}, the closed-loop output dynamics is exactly
\(\dot{\mathbf y}_{\mathrm c}=\mathbf G_{\mathrm c}(\mathbf y_{\mathrm c})\). Moreover, the zero-output manifold is locally invariant, and
the state \(\mathbf x(t)\) locally converges
to \(
\mathbf x_{\mathrm c}^{\ast}\).
The following proposition summarizes the convergence of the ideal FL dynamics \eqref{eq:ideal-fl-algorithm-coupled} for
\(\mathcal P_{\mathrm c}\).

\begin{proposition}
\label{prop:fl-coupled}
Consider problem \(\mathcal P_{\mathrm c}\). Suppose
Assumptions~\ref{asm:graph}--\ref{asm:growth} and
\ref{asm:reg-coupled} hold. Then
\eqref{eq:ideal-fl-algorithm-coupled} admits a locally unique equilibrium
\((\mathbf x_{\mathrm c}^{\ast},\bm\zeta^\ast)\), which is locally
exponentially stable. Specifically, there exist constants
\(C_{\mathrm{FL}}>0\), \(\alpha_{\mathrm{FL}}>0\), and
\(\delta_{\mathrm{FL}}>0\) such that, for every initial condition satisfying
\(
    \|\mathbf x(0)-\mathbf x_{\mathrm c}^{\ast}\|
    +\|\mathbf h^{\mathrm c}(\mathbf x(0))+\mathbf L_q\bm\zeta(0)\|
    <\delta_{\mathrm{FL}},
\)
the corresponding solution satisfies
\[
\begin{aligned}
    &\|\mathbf x(t)-\mathbf x_{\mathrm c}^{\ast}\|
    +\|\bm\zeta(t)-\bm\zeta^\ast\|
    +\|\mathbf L\mathbf x(t)\|
    +\|\mathbf h^{\mathrm c}(\mathbf x(t))+\mathbf L_q\bm\zeta(t)\|  \\
    &\qquad\le
    C_{\mathrm{FL}}e^{-\alpha_{\mathrm{FL}}t}
    \left(
        \|\mathbf x(0)-\mathbf x_{\mathrm c}^{\ast}\|
        +
        \|\bm\zeta(0)-\bm\zeta^\ast\|
    \right),
    \qquad t\ge0 ,
\end{aligned}
\]
with admissible rate
\(\alpha_{\mathrm{FL}}\in
(0,\min\{\rho_{\mathrm c}/N,\alpha_G\})\).
\end{proposition}

Proposition~\ref{prop:fl-coupled} establishes that, for the lifted problem
\eqref{eq:lifted-problem-coupled}, \(\mathbf x(t)\) converges exponentially to
\(\mathbf x_{\mathrm c}^{\ast}\) and the constraint residuals vanish
exponentially. By Lemma~\ref{lem:lift-equiv}, the same conclusion holds
for the coupled problem \(\mathcal P_{\mathrm c}\).

\subsection{Distributed Singular-Perturbation Realization}
\label{sec:sp-realization-coupled}

As in Section~\ref{sec:sp-realization-dist}, the ideal FL relation
\eqref{eq:ideal-fl-relation-coupled} cannot be implemented directly,
since evaluating the unique admissible solution
\(\mathbf u_{\mathrm c}^{\rm q}(\bm\chi)\) requires global coupling
information. We therefore proceed to replace the algebraic
solution by a fast auxiliary dynamics that drives the residual to zero, and then
remove the term \(\mathbf A_{\mathrm c}(\bm\chi)\nabla\bar f(\bm\chi)\) by a
change of variables.

Define the FL residual
\[
    \mathbf e_{\rm c}
    :=
    \mathbf M_{\mathrm c}(\bm\chi)\mathbf u_{\mathrm c}
    +
    \mathbf A_{\mathrm c}(\bm\chi)\nabla\bar f(\bm\chi)
    +
    \mathbf G_{\mathrm c}(\mathbf y_{\mathrm c}),
\]
so that \(\mathbf e_{\rm c}=\mathbf 0\) recovers
\eqref{eq:ideal-fl-relation-coupled}, and, in view of
\eqref{eq:fl-output-compact-coupled}, \(\mathbf e_{\rm c}
=-\dot{\mathbf y}_{\mathrm c}+\mathbf G_{\mathrm c}(\mathbf y_{\mathrm c})\). Together
with the primal flow \eqref{eq:primal-flow-coupled}, this forms the singularly
perturbed interconnection
\begin{equation}
\label{eq:fast-algorithm-coupled}
    \begin{aligned}
        \dot{\bm\chi}
        &=
        -\nabla\bar f(\bm\chi)
        -\mathbf B_{\mathrm c}(\bm\chi)\mathbf u_{\mathrm c},\\
        \tau\,\dot{\mathbf u}_{\mathrm c}
        &=
        -\mathbf K_0
        \big(
            \mathbf M_{\mathrm c}(\bm\chi)\mathbf u_{\mathrm c}
            +\mathbf A_{\mathrm c}(\bm\chi)\nabla\bar f(\bm\chi)
            +\mathbf G_{\mathrm c}(\mathbf y_{\mathrm c})
        \big),
    \end{aligned}
\end{equation}
with slow state \(\bm\chi=\operatorname{col}(\mathbf x,\bm\zeta)\) and fast state
\(\mathbf u_{\mathrm c}\). The initial condition is $\mathbf v(0)\in\operatorname{Im}\mathbf L$ and \(\bm\zeta(0)\in\operatorname{Im}\mathbf L_q\).

For distributed implementation, we further select
\(\mathbf G_{\mathrm c}=\operatorname{col}(\mathbf G_g,\mathbf G_h)\)
to be block-separable, i.e.,
\(\mathbf G_g(\mathbf y_g)
=\operatorname{col}(\mathbf G_{g,1}(\mathbf y_{g,1}),\ldots,
\mathbf G_{g,N}(\mathbf y_{g,N}))\) and
\(\mathbf G_h(\mathbf y_h)
=\operatorname{col}(\mathbf G_{h,1}(\mathbf y_{h,1}),\ldots,
\mathbf G_{h,N}(\mathbf y_{h,N}))\).

Furthermore, to avoid the term
\(\mathbf A_{\mathrm c}(\bm\chi)\nabla\bar f(\bm\chi)\), which involves
exchanging \(\nabla f_i(\mathbf x_i)\) between nodes, we introduce the change of
variables
\[
    \mathbf z:=\mathbf u_{\mathrm c}-\mathbf K_p\mathbf y_{\mathrm c},
    \qquad
    \mathbf K_p:=\mathbf K_0/\tau .
\]
Differentiating along \eqref{eq:fast-algorithm-coupled} cancels this term and yields
\(\dot{\mathbf z}=-\mathbf K_p\mathbf G_{\mathrm c}(\mathbf y_{\mathrm c})\).
Writing \(\mathbf z=\operatorname{col}(\mathbf z_c,\mathbf z_h)\),
\(K_{pc}:=K_{0c}/\tau\), \(K_{ph}:=K_{0h}/\tau\), and recalling
\(\mathbf y_{\mathrm c}=\operatorname{col}(\mathbf L\mathbf x,
\mathbf h^{\mathrm c}(\mathbf x)+\mathbf L_q\bm\zeta)\),
\(\mathbf u_{\mathrm c}=\operatorname{col}(\mathbf v,\bm\mu)\), the
interconnection \eqref{eq:fast-algorithm-coupled} is equivalent to
\begin{equation}
\label{eq:sp-compact-coupled}
    \begin{aligned}
        \dot{\mathbf x}
            &=-\nabla f(\mathbf x)-\mathbf v
              -\tilde{\mathbf J}^{\mathrm c}(\mathbf x)^{\top}\bm\mu,\\
        \dot{\bm\zeta}
            &=-\mathbf L_q\bm\mu,\\
        \mathbf v
            &=K_{pc}\mathbf L\mathbf x+\mathbf z_c,\\
        \bm\mu
            &=K_{ph}\big(\mathbf h^{\mathrm c}(\mathbf x)+\mathbf L_q\bm\zeta\big)+\mathbf z_h,\\
       \dot{\mathbf z}_c
    &=-K_{pc}\mathbf G_g(\mathbf L\mathbf x),\\
\dot{\mathbf z}_h
    &=-K_{ph}\mathbf G_h\big(\mathbf h^{\mathrm c}(\mathbf x)
      +\mathbf L_q\bm\zeta\big),
    \end{aligned}
\end{equation}
with initial conditions \(\mathbf z_c(0)\in\operatorname{Im}\mathbf L\) and
\(\bm\zeta(0)\in\operatorname{Im}\mathbf L_q\). We refer to
\eqref{eq:sp-compact-coupled} as the SP realization of the ideal FL dynamics
\eqref{eq:ideal-fl-algorithm-coupled}.

Under the block structure of \(\mathbf{G}_{\rm c}\), \(\mathbf L\),
\(\mathbf L_q\), \(\tilde{\mathbf J}^{\mathrm c}\), and \(\mathbf h^{\mathrm c}\), realization
\eqref{eq:sp-compact-coupled} is fully distributed. Each agent
\(i\in\mathcal V\) maintains \(\mathbf x_i,\mathbf z_{c,i}\in\mathbb R^d\),
\(\bm\zeta_i,\mathbf z_{h,i}\in\mathbb R^{q}\), and runs
\begin{equation}
\label{eq:node-level-coupled}
    \begin{aligned}
        \dot{\mathbf x}_i
            &=-\nabla f_i(\mathbf x_i)-\mathbf v_i
              -\mathbf J_i^{\mathrm c}(\mathbf x_i)^{\top}\bm\mu_i,\\
        \dot{\bm\zeta}_i
            &=-\textstyle\sum_{j\in\mathcal N_i}a_{ij}(\bm\mu_i-\bm\mu_j),\\
        \mathbf v_i
            &=K_{pc}\!\sum_{j\in\mathcal N_i}\!a_{ij}(\mathbf x_i-\mathbf x_j)
              +\mathbf z_{c,i},\\
        \bm\mu_i
            &=K_{ph}\Big(\mathbf h_i^{\mathrm c}(\mathbf x_i)
              +\textstyle\sum_{j\in\mathcal N_i}a_{ij}(\bm\zeta_i-\bm\zeta_j)\Big)
              +\mathbf z_{h,i},\\
        \dot{\mathbf z}_{c,i}
            &=-K_{pc}\mathbf G_{g,i}\!\Big(
               \textstyle\sum_{j\in\mathcal N_i}\!a_{ij}(\mathbf x_i-\mathbf x_j)\Big),\\
        \dot{\mathbf z}_{h,i}
            &=-K_{ph}\mathbf G_{h,i}\Big(\mathbf h_i^{\mathrm c}(\mathbf x_i)
              +\textstyle\sum_{j\in\mathcal N_i}a_{ij}(\bm\zeta_i-\bm\zeta_j)\Big),
    \end{aligned}
\end{equation}
using only its own data and the variables \(\mathbf x_j,\bm\zeta_j,\bm\mu_j\) of
its neighbors \(j\in\mathcal N_i\).


The following theorem
states that for sufficiently strong time-scale separation the SP realization
\eqref{eq:sp-compact-coupled} preserves the local exponential convergence of the
ideal FL dynamics \eqref{eq:ideal-fl-algorithm-coupled}.

\begin{theorem}
\label{thm:sp-coupled}
Consider problem \(\mathcal P_{\mathrm c}\). Suppose
Assumptions~\ref{asm:graph}--\ref{asm:growth} and 
\ref{asm:reg-coupled} hold. Then \eqref{eq:sp-compact-coupled} admits a
locally unique equilibrium
\((\mathbf x_{\mathrm c}^{\ast},\bm\zeta^\ast,\mathbf z^\ast)\), which is
locally exponentially stable. Specifically, let \(\alpha_{\mathrm{FL}}\) be any
convergence rate admissible in Proposition~\ref{prop:fl-coupled}. For any
\(\alpha_{\mathrm{SP}}\in(0,\alpha_{\mathrm{FL}})\), there exists
\(\tau^\ast>0\) such that, for every fixed \(\tau\in(0,\tau^\ast)\), there
exist constants \(C_{\mathrm{SP}}>0\) and \(\delta_{\mathrm{SP}}>0\) such that
whenever
\(
    \|\mathbf x(0)-\mathbf x_{\mathrm c}^{\ast}\|
    +\|\mathbf h^{\mathrm c}(\mathbf x(0))+\mathbf L_q\bm\zeta(0)\|
    +\|\mathbf z(0)-\mathbf z^\ast\|<\delta_{\mathrm{SP}},
\)
the corresponding solution satisfies
\begin{equation}
\label{eq:sp-bound-coupled}
\begin{aligned}
    &\|\mathbf x(t)-\mathbf x_{\mathrm c}^{\ast}\|
    +\|\bm\zeta(t)-\bm\zeta^\ast\|
    +\|\mathbf z(t)-\mathbf z^\ast\|
    +\|\mathbf L\mathbf x(t)\|
    +\|\mathbf h^{\mathrm c}(\mathbf x(t))+\mathbf L_q\bm\zeta(t)\|  \\
    &\qquad\le
    C_{\mathrm{SP}}e^{-\alpha_{\mathrm{SP}}t}
    \left(
        \|\mathbf x(0)-\mathbf x_{\mathrm c}^{\ast}\|
        +
        \|\bm\zeta(0)-\bm\zeta^\ast\|
        +
        \|\mathbf z(0)-\mathbf z^\ast\|
    \right),
    \qquad t\ge0 .
\end{aligned}
\end{equation}

\end{theorem}

\section{Discussion}
\label{sec:discuss}
\subsection{Connection between SP Realization and Primal--Dual Dynamics}
\label{sec:pd-connection-dist}

We further discuss the connection between the proposed SP realization and
classical primal--dual dynamics. Consider \(\mathcal P_{\mathrm d}\) as an example, whose augmented Lagrangian is given by
\[
    \mathcal L_{\mathrm d}^{\mathrm a}
    (\mathbf x,\mathbf z_c,\mathbf z_h)
    =
    f(\mathbf x)
    +
    \mathbf z_c^\top\mathbf x
    +
    \mathbf z_h^\top\mathbf h^{\mathrm d}(\mathbf x)
    +
    \frac{K_{pc}}{2}\mathbf x^\top\mathbf L\mathbf x
    +
    \frac{K_{ph}}{2}
    \|\mathbf h^{\mathrm d}(\mathbf x)\|^2,
\]
where \(\mathbf z_c\in\operatorname{Im}\mathbf L\) and
\(\mathbf z_h\in\mathbb R^r\) are multiplier-related variables. The associated primal--dual-type dynamics, which is also considered in
\cite{matei2013distributed}, is
\[
\begin{aligned}
    \dot{\mathbf x}
    &=
    -\nabla f(\mathbf x)
    -K_{pc}\mathbf L\mathbf x
    -\mathbf z_c
    -\tilde{\mathbf J}^{\mathrm d}(\mathbf x)^\top
    \big(
        K_{ph}\mathbf h^{\mathrm d}(\mathbf x)+\mathbf z_h
    \big),\\
    \dot{\mathbf z}_c
    &=
    \mathbf L\mathbf x,\\
    \dot{\mathbf z}_h
    &=
    \mathbf h^{\mathrm d}(\mathbf x).
\end{aligned}
\]
Moreover, when the assigned
output dynamics in \eqref{eq:sp-compact-dist} is chosen as
\(
    \mathbf G_g(\mathbf y_g)
    =
    -\frac{1}{K_{pc}}\mathbf y_g\) and \(
    \mathbf G_h(\mathbf y_h)
    =
    -\frac{1}{K_{ph}}\mathbf y_h
\), the resulting SP realization coincides with the primal--dual dynamics above. Therefore, the proposed
SP realization recovers a primal--dual-type dynamics in this special case.
This shows that classical primal--dual algorithms can be interpreted as
singular-perturbation implementations of the ideal feedback-linearizing law.

On the other hand, for the SP realization, a more typical choice is the
constant linear feedback law
\(
    \mathbf G_{\mathrm d}(\mathbf y_{\mathrm d})
    =
    -k\mathbf y_{\mathrm d},
\)
where \(k>0\) is independent of the singular-perturbation parameter \(\tau\).
In this case,
\(
    \dot{\mathbf z}_c
    =
    kK_{pc}\mathbf L\mathbf x,\) and 
\(    \dot{\mathbf z}_h
    =
    kK_{ph}\mathbf h^{\mathrm d}(\mathbf x).
\)
Since \(K_{pc}=K_{0c}/\tau\) and \(K_{ph}=K_{0h}/\tau\), the resulting
dual-variable dynamics has \(\tau\)-dependent high gains. Thus, compared with the fixed-gain primal--dual dynamics, the SP realization
uses \(\tau\)-dependent high-gain multiplier dynamics to rapidly enforce the
ideal feedback-linearizing relation, thereby recovering the assigned output
dynamics \(\dot{\mathbf y}_{\mathrm d}=-k\mathbf y_{\mathrm d}\) in the
singular-perturbation limit.

\subsection{Enlarged Regions of Attraction}
\label{sec:enlarged-roa}

The convergence results established in the previous sections are local in the
standard sense that the initial condition is required to be sufficiently close to
the optimizer. This is natural under the standing local assumptions imposed
around \(\xb_\ell^\ast\). Nevertheless, the feedback-linearization viewpoint
also suggests a possible enlargement of the region of attraction. Indeed, the
output part of the dynamics is responsible for driving the consensus and
constraint residuals to zero, while the zero-output dynamics is a projected
gradient flow on the feasible manifold. Therefore, under stronger regularity
conditions in a neighborhood of the feasible manifold, one may enlarge the
admissible initial set from a neighborhood of the optimizer to a tubular
neighborhood of the constraint manifold.

In this section, we discuss this extension for the FS-NDE realization of
\(\mathcal P_{\mathrm d}\). The same idea applies to the coupled problem
\(\mathcal P_{\mathrm c}\) after using the augmented formulation introduced in
Section~\ref{sec:coupled-equality}.
Let \(\Omega_{\mathrm d}^{\ast}\) denote the connected component of
\(\Omega_{\mathrm d}\) that contains \(\sbf_{\mathrm d}^{\ast}\), and define
\[
    \mathcal X_{\mathrm d}^{\ast}
    :=
    \left\{
        \mathbf 1_N\otimes\sbf:
        \sbf\in\Omega_{\mathrm d}^{\ast}
    \right\}.
\]
For \(\varepsilon>0\), let
\[
    \mathcal N_{\mathrm d}(\varepsilon)
    :=
    \left\{
        \xb\in\mathbb R^{Nd}:
        \operatorname{dist}
        \left(
            \xb,\mathcal X_{\mathrm d}^{\ast}
        \right)
        <
        \varepsilon
    \right\}.
\]
The following assumptions are stronger versions of
Assumptions~\ref{asm:smooth}, \ref{asm:growth}, and
\ref{asm:reg-dist}. They are imposed in a neighborhood of the feasible
manifold rather than only around the optimizer.

\begin{assumption}
\label{asm:tube-smooth-dist}
For problem \(\mathcal P_{\mathrm d}\), there exists
\(\varepsilon>0\) such that, for each \(i\in\mathcal V\),
\(\nabla^2 f_i\) and \(\mathbf J_i^{\mathrm d}\) are Lipschitz continuous on
the set $\mathcal N_{\mathrm d}(\varepsilon)$.
\end{assumption}

\begin{assumption}
\label{asm:tube-reg-dist}
For problem \(\mathcal P_{\mathrm d}\), on the set
\(\mathcal N_{\mathrm d}(\varepsilon)\), the stacked constraint Jacobian is
uniformly full row rank, i.e., there exists
\(\underline\sigma_{\mathrm d}>0\) such that
\[
    \sigma_{\min}
    \left(
    \operatorname{col}
    \big(
        \mathbf J_1^{\mathrm d}(\xb_1),\ldots,
        \mathbf J_N^{\mathrm d}(\xb_N)
    \big)
    \right)
    \ge
    \underline\sigma_{\mathrm d},
    \qquad
    \forall \xb\in\mathcal N_{\mathrm d}(\varepsilon),
\]
where \(r:=\sum_{i=1}^{N}r_i\).
\end{assumption}

\begin{assumption}
\label{asm:manifold-pl-dist}
For problem \(\mathcal P_{\mathrm d}\), there exists
\(\mu_{\mathrm d}>0\) such that \(F\) satisfies the manifold
Polyak--{\L}ojasiewicz condition on \(\Omega_{\mathrm d}^{\ast}\), i.e.,
\[
    \frac{1}{2}
    \left\|
        \Pi_{T_{\sbf}\Omega_{\mathrm d}}
        \nabla F(\sbf)
    \right\|^2
    \ge
    \mu_{\mathrm d}
    \left(
        F(\sbf)-F(\sbf_{\mathrm d}^{\ast})
    \right),
    \qquad
    \forall \sbf\in\Omega_{\mathrm d}^{\ast}.
\]
\end{assumption}

\begin{remark}
Assumption~\ref{asm:tube-smooth-dist} strengthens the local Lipschitz
requirement in Assumption~\ref{asm:smooth} from a neighborhood of
\(\xb_{\mathrm d}^{\ast}\) to a tubular neighborhood of the feasible manifold
component \(\mathcal X_{\mathrm d}^{\ast}\). Assumption~\ref{asm:tube-reg-dist}
is the corresponding version of Assumption~\ref{asm:reg-dist} that requires uniform full rank on the feasible manifold. For Assumption \ref{asm:manifold-pl-dist}, the Polyak--{\L}ojasiewicz condition is a standard regularity condition used
to establish exponential convergence for nonconvex optimization problems. In the
present setting, we only require this condition to hold for the restricted
objective on the feasible manifold component \(\Omega_{\mathrm d}^{\ast}\),
rather than in the whole Euclidean space.

\end{remark}

\begin{theorem}
\label{thm:sp-dist-enlarged-roa}
Consider problem \(\mathcal P_{\mathrm d}\). Suppose Assumptions~\ref{asm:graph}, \ref{asm:growth},
\ref{asm:tube-smooth-dist}, 
\ref{asm:tube-reg-dist} and \ref{asm:manifold-pl-dist} hold. Then \eqref{eq:sp-compact-dist} admits a locally unique equilibrium \((\mathbf x_{\mathrm d}^{\ast},\zb^\ast)\), which is locally exponentially stable. For any
\(
    \alpha_{\mathrm{SP}}
    \in
    \left(
        0,
        \min
        \left\{
            \frac{2\mu_{\mathrm d}}{N},
            \alpha_G
        \right\}
    \right),
\)
there exist constants \(\tau^\ast>0\), \(C_{\mathrm{SP}}>0\), and
\(\delta_z>0\) such that, for every
\(\tau\in(0,\tau^\ast)\), if
\[
    \xb(0)\in \mathcal N_{\mathrm d}(\varepsilon),
    \qquad
    \left\|
        \zb(0)-\zb_{\mathrm d}^{\mathrm q}(\xb(0))
    \right\|
    <
    \delta_z,
\]
where $\zb_{\mathrm d}^{\mathrm q}(\xb(0)):=\ub_{\mathrm d}^{\mathrm q}(\xb(0)) - \mathbf K_p\yb_{\mathrm d}(\xb(0))$ with $\ub_{\mathrm d}^{\mathrm q}(\xb(0))$ being the unique solution of \eqref{eq:ideal-fl-relation-dist},
and \(\zb_c(0)\in\operatorname{Im}\mathbf L\), then the solution of
\eqref{eq:sp-compact-dist}  satisfies
\[
\begin{aligned}
    &\|\xb(t)-\xb_{\mathrm d}^{\ast}\|
    +
    \|\mathbf L\xb(t)\|
    +
    \|\mathbf h^{\mathrm d}(\xb(t))\|
    +
    \left\|
        \zb(t)-\zb^\ast
    \right\|                                      \\
    &\qquad\le
    C_{\mathrm{SP}}e^{-\alpha_{\mathrm{SP}}t}
    \Big(
        \|\xb(0)-\xb_{\mathrm d}^{\ast}\|
        +
        \|\mathbf L\xb(0)\|
        +
        \|\mathbf h^{\mathrm d}(\xb(0))\|
        +
        \left\|
            \zb(0)-\zb^\ast
        \right\|
    \Big),
    \qquad t\ge0 .
\end{aligned}
\]
\end{theorem}

Theorem~\ref{thm:sp-dist-enlarged-roa} shows that the admissible initial
condition for the primal state can be enlarged from a small neighborhood of
\(\xb_{\mathrm d}^{\ast}\) to a tubular neighborhood of the feasible manifold
patch \(\mathcal X_{\mathrm d}^\ast\). 
This observation is useful when an initial primal point satisfying the equality
constraints is available. For example, suppose
\(\xb(0)=\mathbf 1_N\otimes\sbf(0)\in\mathcal X_{\mathrm d}^\ast\),
where \(\sbf(0)\in\Omega_{\mathrm d}^{\ast}\). Then initialize
\(
    \zb(0)
    =
    \ub_{\mathrm d}^{\mathrm q}(\xb(0))-\mathbf K_p\yb_{\mathrm d}(\xb(0))= \ub_{\mathrm d}^{\mathrm q}(\xb(0))
\), and the initial condition is satisfied. Therefore, the FS-NDE realization
converges to \(\xb_{\mathrm d}^{\ast}\) even when \(\xb(0)\) is not
sufficiently close to \(\xb_{\mathrm d}^{\ast}\). 

In particular, we consider the case where the local equality constraints
are affine and Assumption~\ref{asm:tube-smooth-dist} is strengthened to
hold globally. Then we have the following corollary.

\begin{corollary}
\label{cor:sp-dist-global-affine}
Consider problem \(\mathcal P_{\mathrm d}\), and suppose all the conditions
of Theorem~\ref{thm:sp-dist-enlarged-roa} hold. In addition, suppose that
\(\mathbf h_i^{\mathrm d}(\mathbf x_i)
=\mathbf A_i^{\mathrm d}\mathbf x_i-\mathbf b_i^{\mathrm d}\) for all
\(i\in\mathcal V\), and that
Assumption~\ref{asm:tube-smooth-dist} holds globally on
\(\mathbb R^{Nd}\). Then, for any
\(\alpha_{\mathrm{SP}}\in
(0,\min\{2\mu_{\mathrm d}/N,\alpha_G\})\), there exists
\(\tau^\ast>0\) such that, for every fixed
\(\tau\in(0,\tau^\ast)\), the equilibrium
\((\mathbf x_{\mathrm d}^{\ast},\mathbf z^\ast)\) of
\eqref{eq:sp-compact-dist} is globally exponentially stable with rate
\(\alpha_{\mathrm{SP}}\). 
\end{corollary}

\begin{proof}
Under affine constraints, the constraint Jacobians and the matrices
defining the boundary-layer dynamics are constant. Hence,
Assumption~\ref{asm:tube-reg-dist} and the boundary-layer stability hold
globally. Given that Assumption~\ref{asm:tube-smooth-dist} holds globally, the result then follows by applying the proof of
Theorem~\ref{thm:sp-dist-enlarged-roa} to the whole admissible state space.
\end{proof}

Furthermore, for general nonlinear equality constraints, if
Assumptions~\ref{asm:tube-smooth-dist} and
\ref{asm:tube-reg-dist} are imposed globally, then the same argument may
also yield global exponential convergence. However, for nonconvex objectives
with nonlinear equality constraints, global Lipschitz conditions on
\(\mathbb R^{Nd}\) and global uniform full-rank conditions are very
strong. Such assumptions exclude many nonlinear constraint structures
with singularities or unbounded curvature. For this reason, we do not
state global convergence under general nonlinear equality constraints as
a formal result.

\section{Discrete-Time Implementation of the Proposed SP Realizations}
\label{sec:discrete-implementation}

The proposed SP realizations in Sections~\ref{sec:sp-realization-dist} and
\ref{sec:sp-realization-coupled} are continuous-time distributed dynamics.
For implementation on digital platforms, we present their explicit Euler
discretizations in this section. Let \(T>0\) denote the sampling period. For any continuous-time
signal \(\xi(t)\), we write \(\xi^k\) for its discrete-time approximation at
\(t=kT\). 

We first consider the SP realization for \(\mathcal P_{\mathrm d}\). 
The explicit discrete-time implementation of \eqref{eq:node-level-dist} is given in
Algorithm~\ref{alg:euler-sp-dist}: Feedback linearization and Singular perturbation-based distributed nonconvex optimization with Nonlinear
Distributed Equality constraints (FS-NDE).

\begin{algorithm}[t]
\caption{Feedback linearization and Singular perturbation-based distributed nonconvex optimization with Nonlinear Distributed Equality constraints (FS-NDE)}
\label{alg:euler-sp-dist}
\begin{algorithmic}[1]
\State \textbf{Initialization:} Each agent \(i\) chooses
\(\mathbf x_i^0\in\mathbb R^d\), \(\mathbf z_{c,i}^0\in\mathbb R^d\), and
\(\mathbf z_{h,i}^0\in\mathbb R^{r_i}\), with
\(\mathbf z_c^0\in\operatorname{Im}\mathbf L\).
\For{\(k=0,1,2,\ldots\)}
    \State Each agent \(i\) receives \(\mathbf x_j^k\) from all
    \(j\in\mathcal N_i\).
    \State Compute
    \[
    \ba
        &\mathbf y_{g,i}^k
        =
        \sum_{j\in\mathcal N_i}a_{ij}(\mathbf x_i^k-\mathbf x_j^k),
        \qquad
        &&\mathbf y_{h,i}^k
        =
        \mathbf h_i^{\mathrm d}(\mathbf x_i^k).\\
        &\mathbf v_i^k
        =
        K_{pc}\mathbf y_{g,i}^k+\mathbf z_{c,i}^k,
        \qquad
        &&\bm\mu_i^k
        =
        K_{ph}\mathbf y_{h,i}^k+\mathbf z_{h,i}^k .
        \ea
    \]
    \State Update
    \[
    \begin{aligned}
        \mathbf x_i^{k+1}
        &=
        \mathbf x_i^k
        -
        T\Big(
            \nabla f_i(\mathbf x_i^k)
            +
            \mathbf v_i^k
            +
            \mathbf J_i^{\mathrm d}(\mathbf x_i^k)^{\top}\bm\mu_i^k
        \Big),\\
        \mathbf z_{c,i}^{k+1}
        &=
        \mathbf z_{c,i}^k
        -
        T K_{pc}\mathbf G_{g,i}(\mathbf y_{g,i}^k),\\
        \mathbf z_{h,i}^{k+1}
        &=
        \mathbf z_{h,i}^k
        -
        T K_{ph}\mathbf G_{h,i}(\mathbf y_{h,i}^k).
    \end{aligned}
    \]
\EndFor
\end{algorithmic}
\end{algorithm}

The following theorem states that the discrete-time implementation, FS-NDE, preserves the local
exponential convergence of the continuous-time SP dynamics.

\begin{theorem}
\label{thm:euler-sp-dist}
Consider problem \(\mathcal P_{\mathrm d}\). Suppose the conditions of
Theorem~\ref{thm:sp-dist} hold. Fix any
\(\tau\in(0,\tau^\ast)\), where \(\tau^\ast\) is given in
Theorem~\ref{thm:sp-dist}, and let \(\alpha_{\mathrm{SP}}\) be an admissible
continuous-time convergence rate in Theorem~\ref{thm:sp-dist}. Then, for any
\(\alpha_{\mathrm E}\in(0,\alpha_{\mathrm{SP}})\), there exists
\(T_{\mathrm d}^\ast>0\) such that, for every
\(T\in(0,T_{\mathrm d}^\ast)\), FS-NDE admits a
locally unique equilibrium
\((\mathbf x_{\mathrm d}^{\ast},\mathbf z^\ast)\), which is locally
exponentially stable. Specifically, there exist constants
\(C_{\mathrm E}>0\) and \(\delta_{\mathrm E}>0\) such that, if
\[
    \|\mathbf x^0-\mathbf x_{\mathrm d}^{\ast}\|
    +
    \|\mathbf z^0-\mathbf z^\ast\|
    <
    \delta_{\mathrm E},
\]
then the discrete-time trajectory satisfies
\[
\begin{aligned}
    &\|\mathbf x^k-\mathbf x_{\mathrm d}^{\ast}\|
    +
    \|\mathbf z^k-\mathbf z^\ast\|
    +
    \|\mathbf L\mathbf x^k\|
    +
    \|\mathbf h^{\mathrm d}(\mathbf x^k)\|  \\
    &\qquad\le
    C_{\mathrm E}e^{-\alpha_{\mathrm E}kT}
    \left(
        \|\mathbf x^0-\mathbf x_{\mathrm d}^{\ast}\|
        +
        \|\mathbf z^0-\mathbf z^\ast\|
    \right),
    \qquad k=0,1,2,\ldots .
\end{aligned}
\]
\end{theorem}

We next discretize the SP realization for the coupled problem
\(\mathcal P_{\mathrm c}\). 
The explicit discrete-time implementation of \eqref{eq:node-level-coupled} is given in
Algorithm~\ref{alg:euler-sp-coupled}: Feedback linearization and Singular perturbation-based distributed nonconvex
optimization with Nonlinear Coupled Equality constraints (FS-NCE).

\begin{algorithm}[t]
\caption{Feedback linearization and Singular perturbation-based distributed nonconvex optimization with Nonlinear Coupled Equality constraints (FS-NCE)}
\label{alg:euler-sp-coupled}
\begin{algorithmic}[1]
\State \textbf{Initialization:} Each agent \(i\) chooses
\(\mathbf x_i^0\in\mathbb R^d\), \(\bm\zeta_i^0\in\mathbb R^q\),
\(\mathbf z_{c,i}^0\in\mathbb R^d\), and
\(\mathbf z_{h,i}^0\in\mathbb R^q\), with
\(\bm\zeta^0\in\operatorname{Im}\mathbf L_q\) and
\(\mathbf z_c^0\in\operatorname{Im}\mathbf L\).
\For{\(k=0,1,2,\ldots\)}
    \State Each agent \(i\) receives \(\mathbf x_j^k\) and \(\bm\zeta_j^k\) from all \(j\in\mathcal N_i\).
    \State Compute
    \[
    \ba
        &\mathbf y_{g,i}^k
        =
        \sum_{j\in\mathcal N_i}a_{ij}(\mathbf x_i^k-\mathbf x_j^k),\qquad
        &&\mathbf y_{h,i}^k
        =
        \mathbf h_i^{\mathrm c}(\mathbf x_i^k)
        +
        \sum_{j\in\mathcal N_i}a_{ij}(\bm\zeta_i^k-\bm\zeta_j^k).\\
        &\mathbf v_i^k
        =
        K_{pc}\mathbf y_{g,i}^k+\mathbf z_{c,i}^k,
        \qquad
        &&\bm\mu_i^k
        =
        K_{ph}\mathbf y_{h,i}^k+\mathbf z_{h,i}^k .
        \ea
    \]
    \State Each agent \(i\) receives 
    \(\bm\mu_j^k\) from all \(j\in\mathcal N_i\).
    \State Update
    \[
    \begin{aligned}
        \mathbf x_i^{k+1}
        &=
        \mathbf x_i^k
        -
        T\Big(
            \nabla f_i(\mathbf x_i^k)
            +
            \mathbf v_i^k
            +
            \mathbf J_i^{\mathrm c}(\mathbf x_i^k)^{\top}\bm\mu_i^k
        \Big),\\
        \bm\zeta_i^{k+1}
        &=
        \bm\zeta_i^k
        -
        T\sum_{j\in\mathcal N_i}a_{ij}(\bm\mu_i^k-\bm\mu_j^k),\\
        \mathbf z_{c,i}^{k+1}
        &=
        \mathbf z_{c,i}^k
        -
        T K_{pc}\mathbf G_{g,i}(\mathbf y_{g,i}^k),\\
        \mathbf z_{h,i}^{k+1}
        &=
        \mathbf z_{h,i}^k
        -
        T K_{ph}\mathbf G_{h,i}(\mathbf y_{h,i}^k).
    \end{aligned}
    \]
\EndFor
\end{algorithmic}
\end{algorithm}

The convergence of FS-NCE is stated below.

\begin{theorem}
\label{thm:euler-sp-coupled}
Consider problem \(\mathcal P_{\mathrm c}\). Suppose the conditions of
Theorem~\ref{thm:sp-coupled} hold. Fix any
\(\tau\in(0,\tau^\ast)\), where \(\tau^\ast\) is given in
Theorem~\ref{thm:sp-coupled}, and let \(\alpha_{\mathrm{SP}}\) be an admissible
continuous-time convergence rate in Theorem~\ref{thm:sp-coupled}. Then, for any
\(\alpha_{\mathrm E}\in(0,\alpha_{\mathrm{SP}})\), there exists
\(T_{\mathrm c}^\ast>0\) such that, for every
\(T\in(0,T_{\mathrm c}^\ast)\), FS-NCE admits a
locally unique equilibrium
\((\mathbf x_{\mathrm c}^{\ast},\bm\zeta^\ast,\mathbf z^\ast)\), which is
locally exponentially stable. Specifically, there exist constants
\(C_{\mathrm E}>0\) and \(\delta_{\mathrm E}>0\) such that, if
\[
    \|\mathbf x^0-\mathbf x_{\mathrm c}^{\ast}\|
    +
    \|\bm\zeta^0-\bm\zeta^\ast\|
    +
    \|\mathbf z^0-\mathbf z^\ast\|
    <
    \delta_{\mathrm E},
\]
then the discrete-time trajectory satisfies
\[
\begin{aligned}
    &\|\mathbf x^k-\mathbf x_{\mathrm c}^{\ast}\|
    +
    \|\bm\zeta^k-\bm\zeta^\ast\|
    +
    \|\mathbf z^k-\mathbf z^\ast\|
    +
    \|\mathbf L\mathbf x^k\| 
    +
    \|\mathbf h^{\mathrm c}(\mathbf x^k)+\mathbf L_q\bm\zeta^k\|  \\
    &\qquad\le
    C_{\mathrm E}e^{-\alpha_{\mathrm E}kT}
    \left(
        \|\mathbf x^0-\mathbf x_{\mathrm c}^{\ast}\|
        +
        \|\bm\zeta^0-\bm\zeta^\ast\|
        +
        \|\mathbf z^0-\mathbf z^\ast\|
    \right),
    \qquad k=0,1,2,\ldots .
\end{aligned}
\]
\end{theorem}

\section{Numerical Simulations}
\label{sec:simulations}

In this section, we provide numerical simulations to illustrate the proposed
FL dynamics and its SP realization. 
In the following simulations, the ideal FL dynamics \eqref{eq:ideal-fl-algorithm-dist} and \eqref{eq:ideal-fl-algorithm-coupled} are implemented through their Euler-discretized forms, although the explicit discrete-time algorithm is not presented in the paper. Meanwhile, the proposed SP realizations are also implemented in their Euler-discretized forms, i.e., FS-NDE and FS-NCE.

\subsection{Small-scale implementation}
\label{subsec:small-scale-sim}

\begin{figure}[t]
    \centering
    \begin{subfloat}[\(\mathcal P_{\mathrm d}\) solved by FS-NDE.]
{\includegraphics[width=0.48\linewidth]{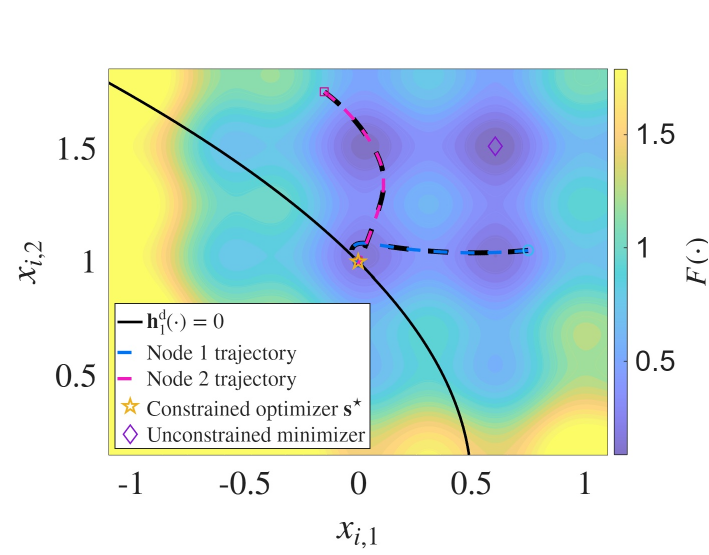}
        \label{fig:Pd-small}}
    \end{subfloat}
    \hfill
    \begin{subfloat}[\(\mathcal P_{\mathrm c}\) solved by FS-NCE.]
{\includegraphics[width=0.48\linewidth]{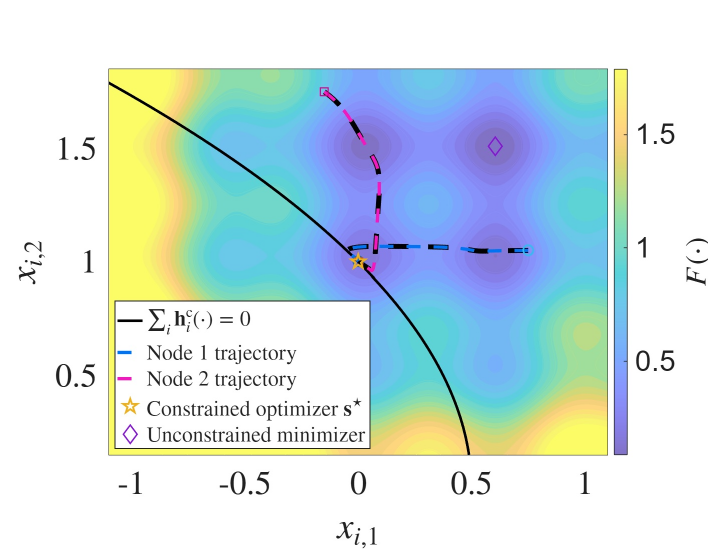}
        \label{fig:Pc-small}}
    \end{subfloat}
    \caption{Small-scale simulations of the proposed SP realizations with \(N=2\),
\(d=2\), and \(\tau=4\times10^{-2}\).}
    \label{fig:small-scale}
\end{figure}

To clearly illustrate the convergence behavior of the proposed algorithms, we first
consider small-scale numerical examples for \(\mathcal P_{\mathrm d}\) and
\(\mathcal P_{\mathrm c}\).
We consider a small-scale example with \(N=2\) agents and local dimension
\(d=2\). For each agent, let
\(
    \mathbf x_i
    =
    \col(x_{i,1},x_{i,2})\in\mathbb R^2
\) for $i=1,2$.
The objective function is selected as
\[
    f_i(\mathbf x_i)
    =
    \frac{1}{2}
    (\mathbf x_i-\mathbf a_i)^\top
    \mathbf Q
    (\mathbf x_i-\mathbf a_i)
    +
    A_1\big(1-\cos(k_1x_{i,1})\big)
    +
    A_2\big(1-\cos(k_2(x_{i,2}-1))\big),
\]
where
\(
    \mathbf Q=\mathrm{diag}(0.8,0.8),\ 
    A_1=0.09,\  A_2=0.07,\ 
    k_1=10,\ k_2=12,\ 
    \mathbf a_1=\col(0.43,1.31),\ 
    \mathbf a_2=\col(0.27,1.39).
\)
Due to the sinusoidal terms, the objective function is nonconvex and has
multiple local peaks and valleys in the displayed region. Moreover, the
unconstrained minimizer is not located on the equality-constrained feasible
manifold.

For problem \(\mathcal P_{\mathrm d}\), only the first agent is subject to a
local nonlinear equality constraint, i.e.,
\(
    h_1^{\mathrm d}(\mathbf x_1)
    =
    x_{1,1}
    +
    \frac{1}{2}x_{1,2}^2
    -
    \frac{1}{2}
    =
    0,
\)
while the second agent has no local equality constraint. Together with the
consensus constraint \(\mathbf x_1=\mathbf x_2\), the corresponding centralized
feasible set is
\[
    \Omega_{\mathrm d}
    =
    \left\{
    \mathbf s\in\mathbb R^2:
    s_1+\frac{1}{2}s_2^2-\frac{1}{2}=0
    \right\}.
\]
This feasible set is a one-dimensional nonlinear manifold. 

For problem \(\mathcal P_{\mathrm c}\), both agents have scalar nonlinear maps
and the equality constraint is imposed through their sum. Specifically, define
\(
    g(\mathbf s)
    =
    s_1+\frac{1}{2}s_2^2-\frac{1}{2}.
\)
We set
\(
    h_1^{\mathrm c}(\mathbf x_1)
    =
    0.6\,g(\mathbf x_1)+0.05,
\) and \(
    h_2^{\mathrm c}(\mathbf x_2)
    =
    0.4\,g(\mathbf x_2)-0.05 .
\)
The aggregate constraint is
\(
    h_1^{\mathrm c}(\mathbf x_1)
    +
    h_2^{\mathrm c}(\mathbf x_2)
    =
    0.
\)
Hence \(\mathcal P_{\mathrm c}\) has the same feasible
manifold as \(\mathcal P_{\mathrm d}\). 

Moreover, we set the desired output field $\mathbf G_{\mathrm d}(\mathbf y_{\mathrm d})=-k\mathbf y_{\mathrm d}$ and $\mathbf G_{\mathrm c}(\mathbf y_{\mathrm c})=-k\mathbf y_{\mathrm c}$ for $k=2$.
Under the above settings, it can be verified that assumptions used in this paper hold.

The SP realization, FS-NDE and FS-NCE, are applied to both problems with singular-perturbation parameter $\tau=4\times10^{-2}$ to solve \(\mathcal P_{\mathrm d}\) and \(\mathcal P_{\mathrm c}\), respectively. The simulation results are
shown in Fig.~\ref{fig:small-scale}. The background contours represent the
centralized objective value \(F(\mathbf s)=f_1(\mathbf s)+f_2(\mathbf s)\), the
solid black curve represents the equality-constrained manifold, and the two
curves show the state trajectories of the two agents. For both
\(\mathcal P_{\mathrm d}\) and \(\mathcal P_{\mathrm c}\), the agents start
away from the constraint manifold and then move toward consensus while being
driven to the constrained optimizer. Although the objective landscape is
nonconvex and contains several local peaks and valleys, the trajectories
converge to the desired constrained equilibrium in a neighborhood of
\(\mathbf s^{\ast}\). This confirms the local convergence behavior predicted
by the proposed theory in Theorems \ref{thm:euler-sp-dist} and \ref{thm:euler-sp-coupled}.

\subsection{Large-scale implementation}
\label{subsec:large-scale-sim}

We further consider application-oriented examples to show how the proposed
algorithms can be used to solve a practical distributed optimization problem.

\subsubsection{Nonlinear system identification problem solved by FS-NDE}
For problem \(\mathcal P_{\mathrm d}\), motivated by the centralized counterpart \cite{cerone2025framework},  we consider a distributed nonlinear
system identification problem. The output of the nonlinear model is described by
\[
    y(k)
    =
    \theta_1 e^{-y^2(k-1)}
    +\theta_2 u^2(k-1)
    +\theta_3 u(k-2)y(k-1)
    +\theta_4 u(k-2)^{\theta_5},
\]
where \(\bm\theta=\col(\theta_1,\ldots,\theta_5)\) is the unknown parameter vector,
and \(u(k)\) and \(y(k)\) denote the input and the noise-free output, respectively.
Given the measured output \(\tilde y(k)\), the identification objective is to fit
the output \(y(k)\) to \(\tilde y(k)\) while enforcing the above model
equations as nonlinear equality constraints.

To cast this problem into the form of \(\mathcal P_{\mathrm d}\), each agent
maintains a local copy
\(
    \mathbf x_i=\col(\bm\theta_i,y_i(1),\ldots,y_i(M)),
\)
and the data indices are partitioned among the agents. Specifically, the local
cost function is chosen as
\[
    f_i(\mathbf x_i)
    =
    \frac{1}{2}\sum_{k\in\mathcal K_i}
    \big(y_i(k)-\tilde y(k)\big)^2,
\]
where \(\mathcal K_i\) is the set of data samples assigned to agent \(i\). The
local nonlinear equality constraints are given by the model residuals
\[
\begin{aligned}
    h_{i,k}^{\mathrm d}(\mathbf x_i)
    =
    &-y_i(k)
    +\theta_{i,1} e^{-y_i^2(k-1)}
    +\theta_{i,2} u^2(k-1) +\theta_{i,3}u(k-2)y_i(k-1)
    +\theta_{i,4}u(k-2)^{\theta_{i,5}},
    \quad k\in\mathcal K_i .
\end{aligned}
\]
Thus, along with the consensus constraints, this example is a nonlinear equality-constrained distributed
optimization problem of the form \(\mathcal P_{\mathrm d}\). As the consensus constraint guarantees that all
agents identify the same parameter vector and the same noise-free output
sequence, solving this problem
will recover the unknown parameter vector and reconstruct the noise-free output
trajectory that best fits the measured data while satisfying the nonlinear model
equations.

We compare the ideal FL dynamics with the SP realization under different
values of \(\tau\). In the SP realization, the proportional gains are scaled
with \(\tau\), and three values of \(\tau\) are tested. The simulation results
are reported in Fig.~\ref{fig:Pd-sysid}. Two quantities are plotted: the state
error \(\|\mathbf x^k-\mathbf x_{\mathrm d}^{\ast}\|\) and the objective gap
\(|f(\mathbf x^k)-f(\mathbf x_{\mathrm d}^{\ast})|\).

As shown in Fig.~\ref{fig:Pd-sysid}, the ideal FL dynamics drives the state to
the optimal solution and all residuals to zero exponentially, which verifies
Proposition~\ref{prop:fl-dist}, although the proposition is established for the
continuous-time dynamics. Moreover, as \(\tau\) decreases, the trajectories of
the SP realization approach those of the ideal FL dynamics, which is consistent
with the SP analysis in Theorem~\ref{thm:sp-dist}. These results demonstrate
that the proposed SP realization provides an implementable approximation of the
ideal FL dynamics while preserving the desired convergence behavior, verifying
Theorem~\ref{thm:euler-sp-dist}.

\subsubsection{\(SE(3)\)-motivated rigid-registration problem solved by FS-NCE}

For problem \(\mathcal P_{\mathrm c}\), we consider a distributed rigid-registration
problem motivated by pose estimation on \(SE(3)\). Each agent has a subset of
three-dimensional point correspondences and aims to estimate a common rigid-body
transformation
\[
    \mathbf T=
    \begin{bmatrix}
        \mathbf R & \mathbf p\\
        \mathbf 0 & 1
    \end{bmatrix}\in SE(3),
\]
where \(\mathbf R\in SO(3)\) and \(\mathbf p\in\mathbb R^3\). For each local
correspondence \((\mathbf a_{im},\mathbf b_{im})\), the measurement model is
\(\mathbf b_{im}=\mathbf R\mathbf a_{im}+\mathbf p+\boldsymbol\varepsilon_{im}\).

To cast this problem into the form of \(\mathcal P_{\mathrm c}\), each agent
maintains a local copy
\(\mathbf x_i=\col(\operatorname{vec}(\mathbf R_i),\mathbf p_i)\in\mathbb R^{12}\),
where \(\mathbf R_i\in\mathbb R^{3\times 3}\) and
\(\mathbf p_i\in\mathbb R^3\). The local cost function is chosen as
\[
    f_i(\mathbf x_i)
    =
    \frac{1}{2}
    \left\|
        \mathbf R_i\mathbf a_{im}+\mathbf p_i-\mathbf b_{im}
    \right\|^2.
\]
Let \(\mathbf r_{i,1},\mathbf r_{i,2},\mathbf r_{i,3}\) be the three
columns of \(\mathbf R_i\), and define
\[
\begin{aligned}
    \mathbf h_{\mathrm{SO}}(\mathbf R_i)
    =
    \col\big(
    &\mathbf r_{i,1}^{\top}\mathbf r_{i,1}-1,\,
    \mathbf r_{i,2}^{\top}\mathbf r_{i,2}-1,\,
    \mathbf r_{i,3}^{\top}\mathbf r_{i,3}-1,\\
    &\mathbf r_{i,1}^{\top}\mathbf r_{i,2},\,
    \mathbf r_{i,1}^{\top}\mathbf r_{i,3},\,
    \mathbf r_{i,2}^{\top}\mathbf r_{i,3}
    \big).
\end{aligned}
\]
The local nonlinear map is selected as
\(\mathbf h_i^{\mathrm c}(\mathbf x_i)=\omega_i
\mathbf h_{\mathrm{SO}}(\mathbf R_i)\), where
\(\omega_i>0\) and \(\sum_{i=1}^N\omega_i=1\). Thus, along with the consensus constraints, this example is a nonlinear equality-constrained distributed
optimization problem of the form \(\mathcal P_{\mathrm c}\). For such a problem,
under the consensus constraint, one has \(\mathbf R_1=\cdots=\mathbf R_N=\mathbf R\)
and \(\mathbf p_1=\cdots=\mathbf p_N=\mathbf p\). Hence
\(\sum_{i=1}^N\mathbf h_i^{\mathrm c}(\mathbf x_i)
=\mathbf h_{\mathrm{SO}}(\mathbf R)\), and the coupled equality constraint becomes
\(\mathbf R^\top\mathbf R=\mathbf I_3\). If the initial condition is chosen in
the local branch satisfying \(\det(\mathbf R)>0\), this constraint locally
represents \(\xb_i\in SO(3)\times\mathbb R^3\). Therefore, solving this problem estimates the common
rigid-body transformation that best fits the distributed point correspondences
while satisfying the nonlinear \(SE(3)\)-related feasibility constraint.

We compare the ideal FL dynamics with the SP realization under different
values of \(\tau\). In the SP realization, the proportional gains are scaled
with \(\tau\), and three values of \(\tau\) are tested. The simulation results
are reported in Fig.~\ref{fig:Pc-large}. Two quantities are plotted: the state
error \(\|\mathbf x^k-\mathbf x_{\mathrm c}^{\ast}\|\) and the objective gap
\(|f(\mathbf x^k)-f(\mathbf x_{\mathrm c}^{\ast})|\).
As shown in Fig.~\ref{fig:Pc-large}, the ideal FL dynamics converges exponentially to
the optimal solution. Moreover, as \(\tau\) decreases, the SP trajectories
approach the ideal FL trajectory, which is consistent with
Theorem~\ref{thm:sp-coupled} and verifies the discrete-time implementation in
Theorem~\ref{thm:euler-sp-coupled}.

\begin{figure}[t]
    \centering
    \includegraphics[width=0.90\linewidth]{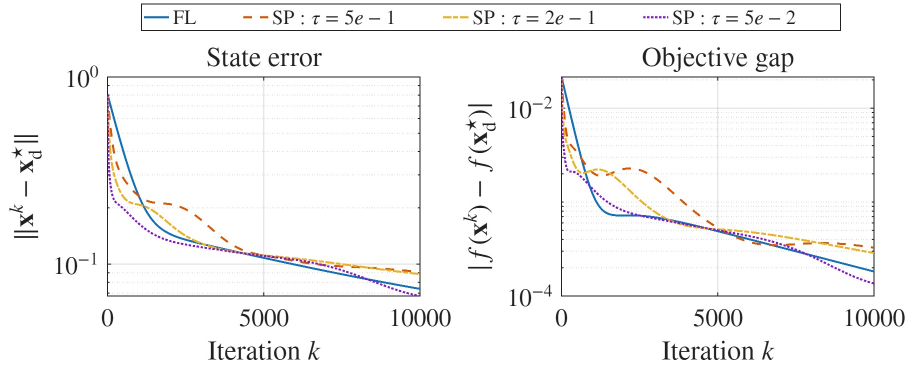}
    \caption{Nonlinear system identification solved by FL dynamics \eqref{eq:ideal-fl-algorithm-dist} and FS-NDE with different $\tau$.}
    \label{fig:Pd-sysid}
\end{figure}

\begin{figure}[t]
    \centering
    \includegraphics[width=0.90\linewidth]{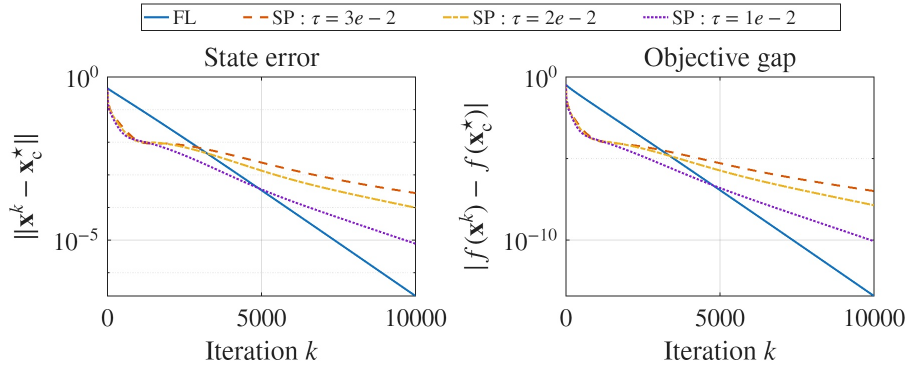}
    \caption{\(SE(3)\)-motivated rigid-registration problem solved by FL dynamics \eqref{eq:ideal-fl-algorithm-coupled} and FS-NCE with different $\tau$.}
    \label{fig:Pc-large}
\end{figure}

\subsubsection{Distributed principal component learning problem solved by FS-NCE}

For problem \(\mathcal P_{\mathrm c}\), we
consider distributed principal component learning, motivated by \cite{chen2021decentralized}. Each agent stores a local
subset of the MNIST handwritten-digit dataset
\cite{lecun1998gradient} and aims to learn a common leading principal
direction. We randomly select \(400\) training images from each digit class,
resulting in \(4000\) samples in total. Each \(28\times28\) image is
downsampled to \(8\times8\), vectorized, and mean centered, yielding a data
vector \(\mathbf a_{im}\in\mathbb R^{64}\). To obtain a label-skewed non-IID
partition, the samples are divided into \(20\) label-homogeneous shards, and
each of the \(N=10\) agents is randomly assigned two shards.

Let \(\mathcal D_i\) denote the sample set stored by agent \(i\), with
\(M_i=|\mathcal D_i|\), and define its local empirical covariance matrix as
\(
    \mathbf C_i
    =
    \frac{1}{M_i}
    \sum_{m\in\mathcal D_i}
    \mathbf a_{im}\mathbf a_{im}^{\top}.
\)
Each agent maintains a local principal-direction estimate
\(\mathbf w_i\in\mathbb R^{64}\), and its local objective function is selected
as
\[
    f_i(\mathbf w_i)
    =
    -\frac{1}{2}
    \mathbf w_i^{\top}\mathbf C_i\mathbf w_i.
\]
The local nonlinear map is defined as
\[
    h_i^{\mathrm c}(\mathbf w_i)
    =
    \omega_i
    \big(
        \mathbf w_i^{\top}\mathbf w_i-1
    \big),
    \qquad
    \omega_i=\frac{1}{N}.
\]
Together with the consensus constraints, this gives a distributed
nonconvex optimization problem of the form \(\mathcal P_{\mathrm c}\).
In particular, the objective is concave quadratic and is therefore nonconvex
when minimized, while the normalization condition is imposed through a
nonlinear coupled equality constraint.

Under the consensus constraint
\(\mathbf w_1=\cdots=\mathbf w_N=\mathbf w\), one has
\(
    \sum_{i=1}^{N}
    h_i^{\mathrm c}(\mathbf w_i)
    =
    \mathbf w^{\top}\mathbf w-1.
\)
Since all agents contain the same number of samples, define the global
empirical covariance matrix as
\(
    \mathbf C
    =
    \frac{1}{N}\sum_{i=1}^{N}\mathbf C_i.
\)
The resulting centralized problem is equivalent, up to a positive scaling of
the objective, to
\(
    \min_{\|\mathbf w\|=1}
    -\frac{1}{2}\mathbf w^{\top}\mathbf C\mathbf w,
\) which is exactly the objective of the centralized counterpart.

To evaluate the learned principal direction, define
\(
    \overline{\mathbf w}^{\,k}
    =
    \frac{1}{N}\sum_{i=1}^{N}\mathbf w_i^k,\) and \(
    \widehat{\mathbf w}^{\,k}
    =
    \frac{\overline{\mathbf w}^{\,k}}
    {\|\overline{\mathbf w}^{\,k}\|}.
\)
The simulation results in Fig.~\ref{fig:Pc-pca} report the sign-invariant
principal-direction error
\(
    e_{\mathrm{dir}}^k
    =
    1-
    \left(
        (\mathbf w^\star)^\top
        \widehat{\mathbf w}^{\,k}
    \right)^2
\)
and the explained-variance gap
\(
    e_{\mathrm{var}}^k
    =
    \lambda_1(\mathbf C)
    -
    (\widehat{\mathbf w}^{\,k})^\top
    \mathbf C
    \widehat{\mathbf w}^{\,k}.
\)
As shown in Fig.~\ref{fig:Pc-pca}, the ideal FL dynamics drives both the
principal-direction error and the explained-variance gap to zero
exponentially. FS-NCE exhibits the same convergence behavior, and its
trajectories approach those of the ideal FL dynamics more closely as
\(\tau\) decreases. These results are consistent with
Theorem~\ref{thm:sp-coupled} and verify the discrete-time implementation in
Theorem~\ref{thm:euler-sp-coupled}. They also demonstrate the applicability of
the proposed method to distributed machine-learning problems involving a
nonconvex objective and nonlinear normalization constraints.

\begin{figure}[t]
    \centering
    \includegraphics[width=0.90\linewidth]
    {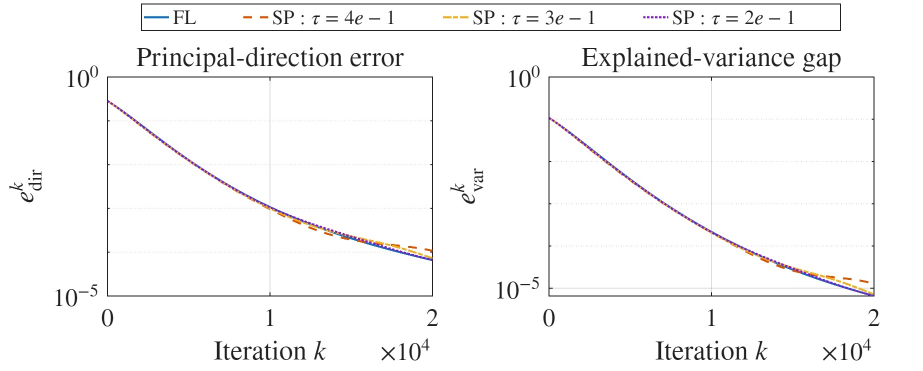}
    \caption{Distributed principal component learning on MNIST solved by
    FL dynamics \eqref{eq:ideal-fl-algorithm-coupled} and FS-NCE with
    different values of \(\tau\).}
    \label{fig:Pc-pca}
\end{figure}

\subsection{Comparison with Existing Methods and Different Choices of Output  Dynamics}
\label{sec.sim_compare}

\begin{figure}[t]
    \centering
    \subfloat[Comparison results for problem \(\mathcal P_{\mathrm d}\).]
    {
        \includegraphics[width=0.48\linewidth]{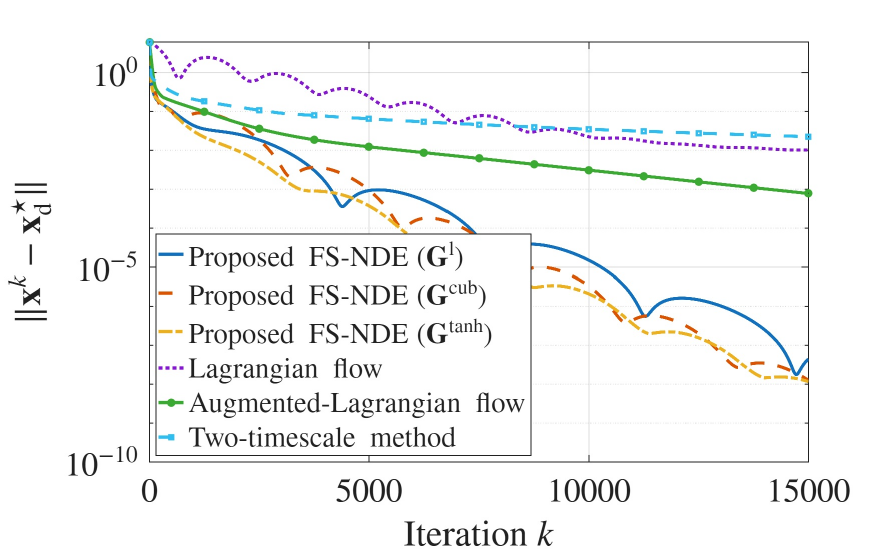}
        \label{fig:compare-pd}
    }
    \hfill
    \subfloat[Comparison results for problem \(\mathcal P_{\mathrm c}\).]
    {
        \includegraphics[width=0.48\linewidth]{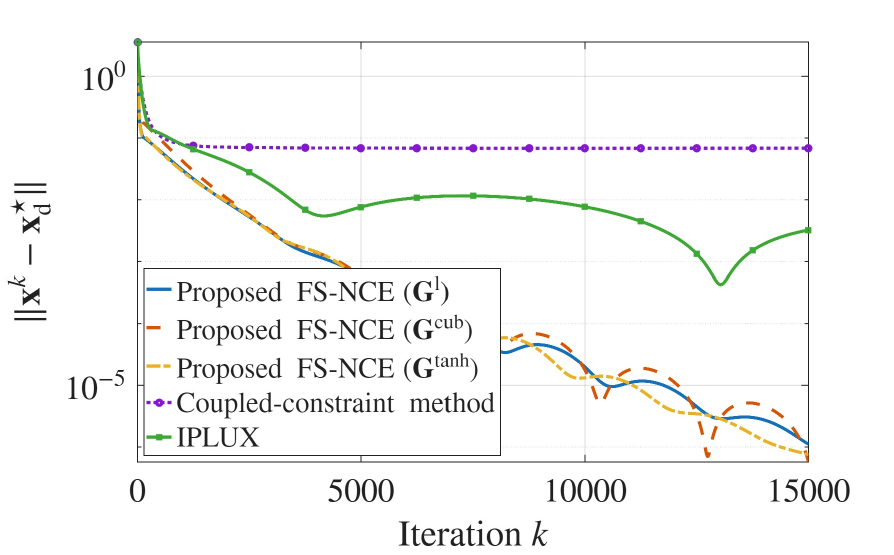}
        \label{fig:compare-pc}
    }
    \caption{Comparison of different dynamics $\Gb$ with existing distributed optimization algorithms.}
    \label{fig:compare-literature}
\end{figure}

In this subsection, we compare the proposed FS-NDE and FS-NCE with existing methods, and further investigate the influence of different choices of the nonlinear dynamics \(\mathbf G\). Specifically, three choices of \(\mathbf G\) are considered:
\[
\begin{aligned}
    \mathbf G^{\mathrm l}(\mathbf y) &= -k\mathbf y,\\
  [\mathbf G_{\mathrm c}^{\mathrm{cub}}(\mathbf y_{\mathrm c})]_{j}
    &=
    -k[\mathbf y_{\mathrm c}]_{j},\quad
     &&[\mathbf G_{\mathrm h}^{\mathrm{cub}}(\mathbf y_{\mathrm h})]_{j}
   =
    -k[\mathbf y_{\mathrm h}]_{j}
    -
    k[\mathbf y_{\mathrm h}]_{j}^{3},\\
     [\mathbf G_{\mathrm c}^{\tanh}(\mathbf y_{\mathrm c})]_{j}
    &=
    -k[\mathbf y_{\mathrm c}]_{j},\quad
    &&[\mathbf G_{\mathrm h}^{\tanh}(\mathbf y_{\mathrm h})]_{j}
    =
    -k[\mathbf y_{\mathrm h}]_{j}
    -
    k\delta
    \tanh\!\left(
        \frac{[\mathbf y_{\mathrm h}]_{j}}{\delta}
    \right).
\end{aligned}
\]
where \(k,\delta>0\). We consider a larger network with \(N=10\) agents and local dimension
\(d=5\). The communication graph is selected as a connected undirected ring
graph.

For problem \(\mathcal P_{\mathrm d}\), we compare the proposed FS-NDE with the existing Lagrangian flow and augmented-Lagrangian flow in \cite{matei2013distributed}, and the two-timescale method in \cite{shi2025distributed}. Since the comparison methods are developed for convex optimization problems, we set the objective function to be strongly convex quadratic functions, rather than using the nonconvex objectives allowed in this paper, while retaining nonlinear equality constraints. The comparison results are shown in Fig.~\ref{fig:compare-literature}(a). It can be observed that the proposed FS-NDE achieves a significantly faster convergence rate for all three choices of \(\mathbf G\), even compared with the augmented-Lagrangian flow in \cite{matei2013distributed}, which can be regarded as a special case of FS-NDE under the proportional choice \(\mathbf G^{\mathrm l}\), as shown in Section~\ref{sec:pd-connection-dist}. As discussed in Section~\ref{sec:pd-connection-dist}, this improvement is mainly due to the high-gain dual dynamics induced by the FS-NDE realization. The high-gain mechanism drives the constraint residuals rapidly toward the equality manifold, thereby accelerating the overall convergence.

For problem \(\mathcal P_{\mathrm c}\), we further compare the proposed FS-NCE with the existing methods, IPLUX in \cite{wu2022distributed} and coupled-constrained method in \cite{liang2020distributed}. The same three choices of output dynamics, \(\mathbf G^{\mathrm l}\), \(\mathbf G^{\mathrm{cub}}\), and \(\mathbf G^{\tanh}\), are adopted. Since the compared literature mainly considers convex optimization problems with linear equality constraints, we also use a strongly convex objective function and linear equality constraints in this comparison. As shown in Fig.~\ref{fig:compare-literature}(b), the proposed FS-NCE achieves exponential convergence and exhibits faster decay than the existing baseline methods with subexponential convergence rates. 

Furthermore, the results also show the influence of different nonlinear choices of \(\mathbf G\). In particular, \(\mathbf G^{\tanh}\) and \(\mathbf G^{\mathrm{cub}}\) lead to faster convergence than the linear choice \(\mathbf G^{\mathrm l}\). This acceleration effect is particularly clear in the FS-NDE case, while it is less pronounced in the FS-NCE case. More generally, these observations indicate that the proposed FS-NDE/FS-NCE framework has the potential to realize different transient and convergence behaviors through the design of \(\mathbf G\), which will be investigated in future work.

\section{Conclusions}
\label{sec:conclusion}
This paper developed a feedback-linearization and singular-perturbation
framework for distributed optimization with equality constraints. Two classes
of problems were considered, i.e., distributed local equality constraints and
coupled equality constraints. Under a
local quadratic-growth condition on the feasible manifold, local exponential
convergence was established for the ideal dynamics, the distributed
singular-perturbation realizations, and their explicit Euler discretizations.

Future work will investigate extensions of the proposed framework to more
general constraint structures, including mixed equality-inequality constraints
and convex-set constraints. Another direction is to study how to choose different output dynamics $\Gb$ for various transient and convergence behaviors.
It is also of interest to further explore manifold-control tools from nonlinear
control theory for distributed optimization over more complicated constraint
geometries.

\appendix
\section{Proof of Proposition~\ref{prop:fl-dist}}
\label{app:fl-dist}

We first construct local coordinates around
\(\mathbf x_{\mathrm d}^{\ast}
=\mathbf 1_N\otimes\mathbf s_{\mathrm d}^{\ast}\). Since
\(\mathbf L\) is a Laplacian matrix, it is rank deficient. Thus, for the
coordinate construction only, choose a constant matrix
\(\mathbf C\in\mathbb R^{(N-1)d\times Nd}\) such that
\[
    \ker\mathbf C=\ker\mathbf L
    =
    \{\mathbf 1_N\otimes\zeta:\zeta\in\mathbb R^d\},
    \qquad
    \operatorname{rank}\mathbf C=(N-1)d .
\]
For instance, \(\mathbf C\) can be chosen as any full-row-rank selection of
independent rows of \(\mathbf L\). Since \(\mathbf C\) and \(\mathbf L\) have
the same kernel, there exists \(c_L>0\) such that
\begin{equation}
\label{eq:L-C-bound}
    \|\mathbf L\mathbf x\|
    \le
    c_L\|\mathbf C\mathbf x\|,
    \qquad
    \forall \mathbf x\in\mathbb R^{Nd}.
\end{equation}
Define
\[
    \mathbf y_{\mathrm r}(\mathbf x)
    :=
    \operatorname{col}
    \big(
        \mathbf C\mathbf x,
        \mathbf h^{\mathrm d}(\mathbf x)
    \big),
    \qquad
    \mathbf A_{\mathrm r}(\mathbf x)
    :=
    \frac{\partial\mathbf y_{\mathrm r}(\mathbf x)}{\partial\mathbf x}
    =
    \operatorname{col}
    \big(
        \mathbf C,
        \tilde{\mathbf J}^{\mathrm d}(\mathbf x)
    \big).
\]
By Assumption~\ref{asm:reg-dist},
\(\operatorname{col}(\mathbf J_1^{\mathrm d}(\mathbf x_1),\ldots,
\mathbf J_N^{\mathrm d}(\mathbf x_N))\) has full row rank \(r\) locally.
Hence, following the same proof of Lemma \ref{lem:fl-wp-dist}, \(\mathbf A_{\mathrm r}(\mathbf x)\) has full row rank
\((N-1)d+r\) in a sufficiently small neighborhood of
\(\mathbf x_{\mathrm d}^{\ast}\).  Moreover,
\[
    \{\mathbf y_{\mathrm r}=\mathbf 0\}
    =
    \{\mathbf L\mathbf x=\mathbf 0,\ 
      \mathbf h^{\mathrm d}(\mathbf x)=\mathbf 0\}
    =
    \mathcal X_{\mathrm d},
\]
and, on \(\mathcal X_{\mathrm d}\),
\begin{equation}
\label{eq:kerAr-tangent-dist}
    \ker\mathbf A_{\mathrm r}(\mathbf x)
    =
    T_{\mathbf x}\mathcal X_{\mathrm d}.
\end{equation}

Since \(\mathbf y_{\mathrm r}\) is a submersion, the local submersion theorem
gives a complementary coordinate
\(\boldsymbol\xi(\mathbf x)\in\mathbb R^{d-r}\) such that
\(\hat{\bm\Phi}_{\mathrm d}(\mathbf x)
=(\mathbf y_{\mathrm r}(\mathbf x),\boldsymbol\xi(\mathbf x))\) is a local
diffeomorphism around \(\mathbf x_{\mathrm d}^{\ast}\). Define
\(\hat{\mathbf p}_{\mathrm d}(\boldsymbol\xi)
:=\hat{\bm\Phi}_{\mathrm d}^{-1}(\mathbf 0,\boldsymbol\xi)\) and
\(\hat{\mathbf P}_{\mathrm d}(\boldsymbol\xi)
:=\partial\hat{\mathbf p}_{\mathrm d}(\boldsymbol\xi)/
\partial\boldsymbol\xi\). Since
\(\mathbf y_{\mathrm r}(\hat{\mathbf p}_{\mathrm d}(\boldsymbol\xi))=\mathbf 0\),
differentiating with respect to \(\boldsymbol\xi\) gives
\[
    \mathbf A_{\mathrm r}
    \big(
        \hat{\mathbf p}_{\mathrm d}(\boldsymbol\xi)
    \big)
    \hat{\mathbf P}_{\mathrm d}(\boldsymbol\xi)
    =
    \mathbf 0 .
\]
Thus, by \eqref{eq:kerAr-tangent-dist}, we have
\[
    \operatorname{Im}
    \hat{\mathbf P}_{\mathrm d}(\boldsymbol\xi)
    =
    \ker
    \mathbf A_{\mathrm r}
    \big(
        \hat{\mathbf p}_{\mathrm d}(\boldsymbol\xi)
    \big)
    =
    T_{\hat{\mathbf p}_{\mathrm d}(\boldsymbol\xi)}
    \mathcal X_{\mathrm d}.
\]
Let
\(\hat{\mathbf H}_{\star}
:=\hat{\mathbf P}_{\mathrm d}(\boldsymbol\xi_{\mathrm d}^{\ast})^\top
\hat{\mathbf P}_{\mathrm d}(\boldsymbol\xi_{\mathrm d}^{\ast})\). With the full-row-rank property of \(\mathbf A_{\mathrm r}\) and, we have
\(\hat{\mathbf H}_{\star}\succ0\). Set
\(\mathbf T_{\star}:=\hat{\mathbf H}_{\star}^{1/2}\) and
\(\bm\eta:=\mathbf T_{\star}
(\boldsymbol\xi-\boldsymbol\xi_{\mathrm d}^{\ast})\) where $\boldsymbol\xi_{\mathrm d}^{\ast}
:=\boldsymbol\xi(\mathbf x_{\mathrm d}^{\ast})$. Define
\[
    \mathbf p_{\mathrm d}(\bm\eta)
    :=
    \hat{\mathbf p}_{\mathrm d}
    \big(
        \boldsymbol\xi_{\mathrm d}^{\ast}
        +
        \mathbf T_{\star}^{-1}\bm\eta
    \big),
    \qquad
    \mathbf P_{\mathrm d}(\bm\eta)
    :=
    \frac{\partial\mathbf p_{\mathrm d}(\bm\eta)}
    {\partial\bm\eta}.
\]
Then, at \(\bm\eta_{\mathrm d}^{\ast}:=\mathbf 0\),
\begin{equation}
\label{eq:P-normalization-dist}
    \mathbf P_{\mathrm d}(\bm\eta_{\mathrm d}^{\ast})^\top
    \mathbf P_{\mathrm d}(\bm\eta_{\mathrm d}^{\ast})
    =
    \mathbf I,
\end{equation}
and locally
\begin{equation}
\label{eq:Project}
\ba
   & \operatorname{Im}\mathbf P_{\mathrm d}(\bm\eta)
    =
    T_{\mathbf p_{\mathrm d}(\bm\eta)}\mathcal X_{\mathrm d},\\
    &\Pi_{T_{\mathbf x}\mathcal X_{\mathrm d}}
    =
    \mathbf P_{\mathrm d}
    \big(\mathbf P_{\mathrm d}^{\top}\mathbf P_{\mathrm d}\big)^{-1}
    \mathbf P_{\mathrm d}^{\top}.
    \ea
\end{equation}
Let \(\bm\Phi_{\mathrm d}(\mathbf x)
:=(\mathbf y_{\mathrm r}(\mathbf x),\bm\eta(\mathbf x))\) denote the resulting
local coordinate map.

We next characterize the zero-output dynamics. On \(\mathcal X_{\mathrm d}\),
one has \(\mathbf y_{\mathrm d}=\mathbf 0\). Since
\(\mathbf G_{\mathrm d}(\mathbf 0)=\mathbf 0\), the ideal FL law gives
\(\dot{\mathbf y}_{\mathrm d}=\mathbf 0\). Thus \(\mathcal X_{\mathrm d}\) is
locally invariant. On this manifold, the ideal input cancels the normal
component of \(-\nabla f(\mathbf x)\), and the motion is
\begin{equation}
\label{eq:zero-x-dynamics-dist}
    \dot{\mathbf x}
    =
    -
    \Pi_{T_{\mathbf x}\mathcal X_{\mathrm d}}
    \nabla f(\mathbf x).
\end{equation}
Writing \(\mathbf x=\mathbf p_{\mathrm d}(\bm\eta)\),
one has \(\dot{\mathbf x}=\mathbf P_{\mathrm d}(\bm\eta)\dot{\bm\eta}\).
Substituting \eqref{eq:Project} into \eqref{eq:zero-x-dynamics-dist} and using the
chain rule
\(\nabla_{\bm\eta}\widetilde F_{\mathrm d}(\bm\eta)
=\mathbf P_{\mathrm d}(\bm\eta)^\top
\nabla f(\mathbf p_{\mathrm d}(\bm\eta))\), where \(\widetilde F_{\mathrm d}(\bm\eta)
:=f(\mathbf p_{\mathrm d}(\bm\eta))\), we obtain
\[
    \dot{\bm\eta}
    =
    -
    \mathbf H_{\mathrm d}(\bm\eta)^{-1}
    \nabla_{\bm\eta}
    \widetilde F_{\mathrm d}(\bm\eta),
\]
where
\(\mathbf H_{\mathrm d}(\bm\eta)
:=
\mathbf P_{\mathrm d}(\bm\eta)^\top
\mathbf P_{\mathrm d}(\bm\eta)\).

Since
\(\mathbf p_{\mathrm d}(\bm\eta)\in\mathcal X_{\mathrm d}\), there is a local
parametrization \(\mathbf s_{\mathrm d}(\bm\eta)\in\Omega_{\mathrm d}\) such
that
\(
    \mathbf p_{\mathrm d}(\bm\eta)
    =
    \mathbf 1_N\otimes\mathbf s_{\mathrm d}(\bm\eta).
\)
Let
\(\mathbf J_s(\bm\eta)
:=\partial\mathbf s_{\mathrm d}(\bm\eta)/\partial\bm\eta\). Then
\(\mathbf P_{\mathrm d}(\bm\eta)=\mathbf 1_N\otimes\mathbf J_s(\bm\eta)\).
By \eqref{eq:P-normalization-dist},
\(\mathbf J_s(\bm\eta_{\mathrm d}^{\ast})^\top
\mathbf J_s(\bm\eta_{\mathrm d}^{\ast})=\frac{1}{N}\mathbf I\).

Define
\(\widetilde F_{\mathrm d}(\bm\eta)
:=f(\mathbf p_{\mathrm d}(\bm\eta))
=F(\mathbf s_{\mathrm d}(\bm\eta))\).
For any fixed direction \(\zeta\) and sufficiently small \(\tau\), the restricted
quadratic growth condition in Assumption~\ref{asm:growth} gives
\[
    \widetilde F_{\mathrm d}(\bm\eta_{\mathrm d}^{\ast}+\tau\zeta)
    -
    \widetilde F_{\mathrm d}(\bm\eta_{\mathrm d}^{\ast})
    \ge
    \frac{\rho_{\mathrm d}}{2}
    \|\mathbf s_{\mathrm d}(\bm\eta_{\mathrm d}^{\ast}+\tau\zeta)
    -\mathbf s_{\mathrm d}(\bm\eta_{\mathrm d}^{\ast})\|^2 .
\]
On the other hand, since \(\bm\eta_{\mathrm d}^{\ast}\) is a local minimizer of
\(\widetilde F_{\mathrm d}\), one has
\(\nabla_{\bm\eta}\widetilde F_{\mathrm d}(\bm\eta_{\mathrm d}^{\ast})=\mathbf 0\), and Taylor's formula yields
\[
    \widetilde F_{\mathrm d}(\bm\eta_{\mathrm d}^{\ast}+\tau\zeta)
    -
    \widetilde F_{\mathrm d}(\bm\eta_{\mathrm d}^{\ast})
    =
    \frac{\tau^2}{2}
    \zeta^\top
    \nabla_{\bm\eta}^2\widetilde F_{\mathrm d}(\bm\eta_{\mathrm d}^{\ast})
    \zeta
    +
    o(\tau^2).
\]
Moreover,
\[
    \mathbf s_{\mathrm d}(\bm\eta_{\mathrm d}^{\ast}+\tau\zeta)
    -
    \mathbf s_{\mathrm d}(\bm\eta_{\mathrm d}^{\ast})
    =
    \tau\mathbf J_s(\bm\eta_{\mathrm d}^{\ast})\zeta
    +
    o(\tau),
\]
and hence
\[
    \|\mathbf s_{\mathrm d}(\bm\eta_{\mathrm d}^{\ast}+\tau\zeta)
    -\mathbf s_{\mathrm d}(\bm\eta_{\mathrm d}^{\ast})\|^2
    =
    \tau^2
    \zeta^\top
    \mathbf J_s(\bm\eta_{\mathrm d}^{\ast})^\top
    \mathbf J_s(\bm\eta_{\mathrm d}^{\ast})
    \zeta
    +
    o(\tau^2).
\]
Comparing the preceding estimates and using
\(\mathbf J_s(\bm\eta_{\mathrm d}^{\ast})^\top
\mathbf J_s(\bm\eta_{\mathrm d}^{\ast})=\frac{1}{N}\mathbf I\), we obtain
\[
    \nabla_{\bm\eta}^2
    \widetilde F_{\mathrm d}(\bm\eta_{\mathrm d}^{\ast})
    \succeq
    \frac{\rho_{\mathrm d}}{N}\mathbf I .
\]

A Taylor
expansion of the gradient gives
\[
    \nabla_{\bm\eta}\widetilde F_{\mathrm d}(\bm\eta)
    =
    \nabla_{\bm\eta}^2
    \widetilde F_{\mathrm d}(\bm\eta_{\mathrm d}^{\ast})
    {\bm\eta}
    +
    o(\|{\bm\eta}\|).
\]
Also, by the continuity of
\(\mathbf H_{\mathrm d}(\bm\eta)^{-1}\) and
\(\mathbf H_{\mathrm d}(\bm\eta_{\mathrm d}^{\ast})=\mathbf I\), one has
\(\mathbf H_{\mathrm d}(\bm\eta)^{-1}=\mathbf I+o(1)\) as
\(\bm\eta\to\bm\eta_{\mathrm d}^{\ast}\). Therefore,
\[
    {\bm\eta}^{\top}
    \mathbf H_{\mathrm d}(\bm\eta)^{-1}
    \nabla_{\bm\eta}\widetilde F_{\mathrm d}(\bm\eta)
    \ge
    \frac{\rho_{\mathrm d}}{N}
    \|{\bm\eta}\|^2
    +
    o(\|{\bm\eta}\|^2).
\]
Consequently, for any \(\rho_\eta\in(0,\rho_{\mathrm d}/N)\), we have
\begin{equation}
\label{eq:eta-dissipation-dist}
    {\bm\eta}^{\top}
    \mathbf H_{\mathrm d}(\bm\eta)^{-1}
    \nabla_{\bm\eta}\widetilde F_{\mathrm d}(\bm\eta)
    \ge
    \rho_\eta\|{\bm\eta}\|^2 .
\end{equation}

It remains to combine the output dynamics and the internal dynamics. Since
\(\mathbf C\) and \(\mathbf L\) have the same kernel and \(\mathbf C\) has full
row rank, there exist constant matrices \(\mathbf T_C\) and \(\mathbf T_L\) such
that \(\mathbf C=\mathbf T_C\mathbf L\) and
\(\mathbf L=\mathbf T_L\mathbf C\). Define
\(\mathbf S_r:=\operatorname{blkdiag}(\mathbf T_C,\mathbf I_r)\) and
\(\mathbf T_r:=\operatorname{blkdiag}(\mathbf T_L,\mathbf I_r)\). Then
\(\mathbf y_{\mathrm r}=\mathbf S_r\mathbf y_{\mathrm d}\) and
\(\mathbf y_{\mathrm d}=\mathbf T_r\mathbf y_{\mathrm r}\). Hence the output
subsystem in the reduced output coordinate is
\[
    \dot{\mathbf y}_{\mathrm r}
    =
    \mathbf G_{\mathrm r}(\mathbf y_{\mathrm r})
    :=
    \mathbf S_r\mathbf G_{\mathrm d}(\mathbf T_r\mathbf y_{\mathrm r}).
\]
By the exponential stability of the assigned output dynamics, the origin of this output subsystem is
globally exponentially stable. Thus, 
there exist a Lyapunov function \(V_2\) and constants \(c_1,c_2>0\) such that
\[
    c_1\|\mathbf y_{\mathrm r}\|^2
    \le
    V_2(\mathbf y_{\mathrm r})
    \le
    c_2\|\mathbf y_{\mathrm r}\|^2,
    \qquad
    \dot V_2
    \le
    -2\alpha_G V_2 .
\]

In the coordinates \((\mathbf y_{\mathrm r},\bm\eta)\), the
\(\bm\eta\)-dynamics has the form
\[
    \dot{\bm\eta}
    =
    -
    \mathbf H_{\mathrm d}(\bm\eta)^{-1}
    \nabla_{\bm\eta}\widetilde F_{\mathrm d}(\bm\eta)
    +
    \boldsymbol\delta(\bm\eta,\mathbf y_{\mathrm r}),
\]
where \(\boldsymbol\delta(\bm\eta,\mathbf 0)=\mathbf 0\). By local
Lipschitzness, there exists \(L_1>0\) such that locally
\(\|\boldsymbol\delta(\bm\eta,\mathbf y_{\mathrm r})\|
\le L_1\|\mathbf y_{\mathrm r}\|\). Let
\(V_1:=\frac12\|{\bm\eta}\|^2\). From
\eqref{eq:eta-dissipation-dist} and Young's inequality, for any
\(\varepsilon\in(0,\rho_\eta)\) to be determined later,
\[
    \dot V_1
    \le
    -(\rho_\eta-\varepsilon)\|{\bm\eta}\|^2
    +
    \frac{L_1^2}{4\varepsilon}\|\mathbf y_{\mathrm r}\|^2
    \le
    -2(\rho_\eta-\varepsilon)V_1
    +
    c_q V_2,
\]
where \(c_q:=L_1^2/(4\varepsilon c_1)\).

Fix any \(\alpha_{\mathrm{FL}}\in(0,\min\{\rho_\eta,\alpha_G\})\). Choose
\(\varepsilon\in(0,\rho_\eta-\alpha_{\mathrm{FL}})\), and choose
\(\mu>0\) such that \(\mu c_q\le 2(\alpha_G-\alpha_{\mathrm{FL}})\). Define the
composite Lyapunov function \(W:=\mu V_1+V_2\). Then
\[
    \dot W
    \le
    -2\mu(\rho_\eta-\varepsilon)V_1
    -
    (2\alpha_G-\mu c_q)V_2
    \le
    -2\alpha_{\mathrm{FL}}W .
\]
Moreover, with \(m_W:=\min\{\mu/2,c_1\}\) and
\(M_W:=\max\{\mu/2,c_2\}\), one has
\[
    m_W
    \left\|
        \operatorname{col}({\bm\eta},\mathbf y_{\mathrm r})
    \right\|^2
    \le
    W
    \le
    M_W
    \left\|
        \operatorname{col}({\bm\eta},\mathbf y_{\mathrm r})
    \right\|^2 .
\]
Therefore,
\[
    \left\|
        \operatorname{col}
        ({\bm\eta}(t),\mathbf y_{\mathrm r}(t))
    \right\|
    \le
    c_0e^{-\alpha_{\mathrm{FL}}t}
    \left\|
        \operatorname{col}
        ({\bm\eta}(0),\mathbf y_{\mathrm r}(0))
    \right\|,
\]
where \(c_0:=\sqrt{M_W/m_W}\).

Since \(\bm\Phi_{\mathrm d}\) is a local diffeomorphism, there exist constants
\(L_\Psi,L_\Phi>0\) such that locally
\begin{equation}
    \label{eq:diffeo-bound-dist}
    \|\mathbf x-\mathbf x_{\mathrm d}^{\ast}\|
    \le
    L_\Psi
    \left\|
        \operatorname{col}({\bm\eta},\mathbf y_{\mathrm r})
    \right\|,
    \qquad
    \left\|
        \operatorname{col}({\bm\eta},\mathbf y_{\mathrm r})
    \right\|
    \le
    L_\Phi
    \|\mathbf x-\mathbf x_{\mathrm d}^{\ast}\|.
\end{equation}
Moreover, by \eqref{eq:L-C-bound}, one can take
\(c_y:=\sqrt{c_L^2+1}\) such that locally
\(
    \|\mathbf L\mathbf x\|
    +
    \|\mathbf h^{\mathrm d}(\mathbf x)\|
    \le
    c_y\|\mathbf y_{\mathrm r}(\mathbf x)\| .
\)
Combining the above estimates gives
\[
\begin{aligned}
    &\|\mathbf x(t)-\mathbf x_{\mathrm d}^{\ast}\|
    +
    \|\mathbf L\mathbf x(t)\|
    +
    \|\mathbf h^{\mathrm d}(\mathbf x(t))\|\le
    C_{\mathrm{FL}}e^{-\alpha_{\mathrm{FL}}t}
    \|\mathbf x(0)-\mathbf x_{\mathrm d}^{\ast}\|,
\end{aligned}
\]
where \(
    C_{\mathrm{FL}}
    :=(L_\Psi+c_y)c_0L_\Phi
.
\)

Since \(\rho_\eta\in(0,\rho_{\mathrm d}/N)\) is arbitrary, the above argument
allows any
\(\alpha_{\mathrm{FL}}<\min\{\rho_{\mathrm d}/N,\alpha_G\}\) by taking a
sufficiently small neighborhood. Finally, choose \(\delta_{\mathrm{FL}}>0\)
small enough so that the ball
\(\|\mathbf x-\mathbf x_{\mathrm d}^{\ast}\|<\delta_{\mathrm{FL}}\) is mapped
by \(\bm\Phi_{\mathrm d}\) into the local coordinate region where all the above
bounds hold. Since
\(\dot W\le -2\alpha_{\mathrm{FL}}W\), the corresponding sublevel set of \(W\)
is forward invariant, and thus every trajectory starting from
\(\|\mathbf x(0)-\mathbf x_{\mathrm d}^{\ast}\|<\delta_{\mathrm{FL}}\) remains
in the region where the estimates are valid. This completes the proof.

\section{Proof of Theorem~\ref{thm:sp-dist}}
\label{app:sp-dist}

We first prove that \eqref{eq:sp-compact-dist} admits a locally unique
equilibrium. Since \eqref{eq:sp-compact-dist} is equivalent to
\eqref{eq:fast-algorithm-dist} via
\(\mathbf u_{\mathrm d}=\mathbf z+\mathbf K_p\mathbf y_{\mathrm d}\), it suffices
to show that \eqref{eq:fast-algorithm-dist} has a locally unique equilibrium in the
\((\mathbf x,\mathbf u_{\mathrm d})\)-coordinates. At an equilibrium,
\(\dot{\mathbf u}_{\mathrm d}=\mathbf 0\) gives
\(\mathbf M_{\mathrm d}(\mathbf x)\mathbf u_{\mathrm d}
+\mathbf A_{\mathrm d}(\mathbf x)\nabla f(\mathbf x)
+\mathbf G_{\mathrm d}(\mathbf y_{\mathrm d})=\mathbf 0\), so by
Lemma~\ref{lem:fl-wp-dist} the input is uniquely determined by
\(\mathbf x\), i.e. \(\mathbf u_{\mathrm d}=\mathbf u_{\mathrm d}^{\rm q}
(\mathbf x)\). Substituting \(\mathbf u_{\mathrm d}=\mathbf u_{\mathrm d}^{\rm q}(\mathbf x)\)
into \(\dot{\mathbf x}=\mathbf 0\) gives the ideal FL dynamics
\eqref{eq:ideal-fl-algorithm-dist} at rest. By Proposition~\ref{prop:fl-dist},
\(\mathbf x_{\mathrm d}^{\ast}\) is a locally exponentially stable equilibrium
of \eqref{eq:ideal-fl-algorithm-dist}, hence a locally isolated one. Therefore
\(\mathbf x=\mathbf x_{\mathrm d}^{\ast}\) is the locally unique equilibrium, and
\(\mathbf u_{\mathrm d}=\mathbf u_{\mathrm d}^{\rm q}(\mathbf x_{\mathrm d}^{\ast})
=\mathbf u_{\mathrm d}^{\ast}\). Thus
\((\mathbf x_{\mathrm d}^{\ast},\mathbf u_{\mathrm d}^{\ast})\) is the locally
unique equilibrium of \eqref{eq:fast-algorithm-dist}, and consequently
\eqref{eq:sp-compact-dist} has the locally unique equilibrium
\((\mathbf x_{\mathrm d}^{\ast},\mathbf z^\ast)\) with
\(\mathbf z^\ast=\mathbf u_{\mathrm d}^{\ast}
-\mathbf K_p\mathbf y_{\mathrm d}(\mathbf x_{\mathrm d}^{\ast})
=\mathbf u_{\mathrm d}^{\ast}\).

Let 
\(\hat{\mathbf u}:=\mathbf u_{\mathrm d}
-\mathbf u_{\mathrm d}^{\rm q}(\mathbf x)\), with
\(\hat{\mathbf u}=\operatorname{col}(\hat{\mathbf v},
\hat{\bm\mu})\). With \(\mathbf x\) frozen and with the fast time \(s:=t/\tau\), the
fast equation in \eqref{eq:fast-algorithm-dist} and the definition of
\(\mathbf u_{\mathrm d}^{\rm q}\) give
\[
    \frac{d\hat{\mathbf u}}{ds}
    =
    -
    \mathbf K_0\mathbf M_{\mathrm d}(\mathbf x)\hat{\mathbf u}.
\]
Define 
\(
    \hat{\mathbf w}
    :=\mathbf B_{\mathrm d}(\mathbf x)\hat{\mathbf u}=
    \hat{\mathbf v}
    +
    \tilde{\mathbf J}^{\mathrm d}(\mathbf x)^{\top}\hat{\bm\mu}
 \in\mathcal N_{\mathrm d}(\mathbf x),
\)
where $ \mathcal N_{\mathrm d}(\mathbf x)
    :=
    \operatorname{Im}\mathbf L+
    \operatorname{Im}\tilde{\mathbf J}^{\mathrm d}(\mathbf x)^{\top}.$ Using the block structure of \(\mathbf M_{\mathrm d}\), we have
\[
    \frac{d\hat{\mathbf w}}{ds}
    =
    -
    \mathbf N_{\mathrm d}(\mathbf x)\hat{\mathbf w},
    \qquad
    \mathbf N_{\mathrm d}(\mathbf x)
    :=
    K_{0c}\mathbf L+
    K_{0h}
    \tilde{\mathbf J}^{\mathrm d}(\mathbf x)^\top
    \tilde{\mathbf J}^{\mathrm d}(\mathbf x).
\]

The matrix \(\mathbf N_{\mathrm d}(\mathbf x)\) is symmetric positive
semidefinite by Assumption \ref{asm:graph}, with
\(\ker\mathbf N_{\mathrm d}(\mathbf x)=\ker\mathbf L\cap
\ker\tilde{\mathbf J}^{\mathrm d}(\mathbf x)\). Hence it is uniformly positive
definite on \(\mathcal N_{\mathrm d}(\mathbf x)\). Thus, for some \(\beta>0\),
\begin{equation}
\label{eq:N-positive-dist}
    \mathbf w^\top
    \mathbf N_{\mathrm d}(\mathbf x)
    \mathbf w
    \ge
    \beta\|\mathbf w\|^2,
    \qquad
    \forall\,\mathbf w\in\mathcal N_{\mathrm d}(\mathbf x).
\end{equation}
Consequently, for \(V_b:=\frac12\|\hat{\mathbf w}\|^2\),
\[
    \frac{dV_b}{ds}
    =
    -
    \hat{\mathbf w}^{\top}
    \mathbf N_{\mathrm d}(\mathbf x)
    \hat{\mathbf w}
    \le
    -2\beta V_b .
\]
This proves the local uniform linear stability of the boundary-layer
system.

The reduced system is obtained by setting
\(\hat{\mathbf w}=\mathbf 0\). According to the proof of Lemma \ref{lem:fl-wp-dist}, there exist constants
\(\underline b,\overline b>0\) such that, for all
\(\ub_{\rm d}\in\mathcal A\),
\begin{equation}
\label{eq:action-equiv-dist}
    \underline b
    \left\|
        \ub_{\rm d}
    \right\|
    \le
    \left\|
        \mathbf B_{\mathrm d}(\mathbf x){\mathbf u}_{\rm d}
    \right\|
    \le
    \overline b
    \left\|
        \ub_{\rm d}
    \right\|.
\end{equation} 
Then \(\hat{\mathbf w}=\mathbf 0\) is equivalent to 
\(\hat{\mathbf u}=\mathbf 0\). Hence
\(\mathbf u_{\mathrm d}=\mathbf u_{\mathrm d}^{\rm q}(\mathbf x)\), which
satisfies \eqref{eq:ideal-fl-relation-dist}. Therefore, the reduced dynamics is
exactly the ideal FL dynamics \eqref{eq:ideal-fl-algorithm-dist}. By Proposition~\ref{prop:fl-dist} and the proof in
Appendix~\ref{app:fl-dist}, for any admissible
\(\alpha_{\mathrm{FL}}\), there is a local Lyapunov function
\(W(\mathbf x)\) such that, for some \(c_1,c_2,c_3>0\),
\begin{equation}
\label{eq:reduced-lyapunov-dist}
\ba
    &c_1 \varepsilon_{\mathrm d}(\mathbf x)^2
    \le
    W(\mathbf x)
    \le
    c_2 \varepsilon_{\mathrm d}(\mathbf x)^2,\\
    &\nabla W^\top
    \dot{\mathbf x}_{\rm red}
    \le
    -2\alpha_{\mathrm{FL}}W,\\
    &\|\nabla W(\mathbf x)\|
    \le
    c_3 \varepsilon_{\mathrm d}(\mathbf x),
\ea
\end{equation}
where \(\dot{\mathbf x}_{\rm red}\) denotes the vector field of the ideal FL
dynamics \eqref{eq:ideal-fl-algorithm-dist}, and
\(\varepsilon_{\mathrm d}(\mathbf x):=
\|\mathbf x-\mathbf x_{\mathrm d}^{\ast}\|
+\|\mathbf L\mathbf x\|
+\|\mathbf h^{\mathrm d}(\mathbf x)\|\).

For the full system, since
\(\hat{\mathbf w}=\mathbf B_{\mathrm d}(\mathbf x)\hat{\mathbf u}\), the full
dynamics of \eqref{eq:fast-algorithm-dist} in the
\((\mathbf x,\hat{\mathbf w})\)-coordinates can be written as
\[
\begin{aligned}
    \dot{\mathbf x}
    &=
    \dot{\mathbf x}_{\rm red}
    -
    \hat{\mathbf w},\\
    \dot{\hat{\mathbf w}}
    &=
    -
    \frac{1}{\tau}
    \mathbf N_{\mathrm d}(\mathbf x)\hat{\mathbf w}
    +
    \bm\Delta(\mathbf x,\hat{\mathbf w}),
\end{aligned}
\]
where \(\dot{\mathbf x}_{\rm red}\) is the ideal FL vector field in
\eqref{eq:ideal-fl-algorithm-dist} and
\(
    \bm\Delta(\mathbf x,\hat{\mathbf w})
    =
    \frac{\partial \mathbf B_{\mathrm d}}{\partial \mathbf x}(\mathbf x)
    [\dot{\mathbf x}]
    \hat{\mathbf u}
    -
    \mathbf B_{\mathrm d}(\mathbf x)
    \frac{\partial \mathbf u_{\mathrm d}^{\rm q}}{\partial \mathbf x}(\mathbf x)
    \dot{\mathbf x}.
\)
Here, the first equation follows from
\(\dot{\mathbf x}
=-\nabla f(\mathbf x)-\mathbf v
-\tilde{\mathbf J}^{\mathrm d}(\mathbf x)^\top\bm\mu
=\dot{\mathbf x}_{\rm red}-\hat{\mathbf w}\). The second equation follows by
differentiating \(\hat{\mathbf w}=\mathbf B_{\mathrm d}(\mathbf x)\hat{\mathbf u}\)
and using
\(\tau\dot{\hat{\mathbf u}}
=-\mathbf K_0\mathbf M_{\mathrm d}(\mathbf x)\hat{\mathbf u}
-\tau\frac{\partial \mathbf u_{\mathrm d}^{\rm q}}{\partial \mathbf x}
(\mathbf x)\dot{\mathbf x}\), together with
\(\mathbf B_{\mathrm d}(\mathbf x)\mathbf K_0
\mathbf M_{\mathrm d}(\mathbf x)\hat{\mathbf u}
=
\mathbf N_{\mathrm d}(\mathbf x)\hat{\mathbf w}\).

We now estimate the two components of the full system. From
\eqref{eq:reduced-lyapunov-dist} and
\(\dot{\mathbf x}=\dot{\mathbf x}_{\rm red}-\hat{\mathbf w}\), one has
\[
    \dot W
    \le
    -2\alpha_{\mathrm{FL}}W
    +
    c_3\varepsilon_{\mathrm d}(\mathbf x)\|\hat{\mathbf w}\|.
\]
Next, fix a sufficiently small radius \(\rho>0\) such that all the preceding
local estimates hold whenever \(\varepsilon_{\mathrm d}(\mathbf x)\le\rho\) and
\(\|\hat{\mathbf w}\|\le\rho\). On this region, choose constants
\(L_r,L_u,L_B>0\) such that
\(\|\dot{\mathbf x}_{\rm red}\|\le L_r \varepsilon_{\mathrm d}(\mathbf x)\),
\(\|\mathbf B_{\mathrm d}(\mathbf x)
\frac{\partial \mathbf u_{\mathrm d}^{\rm q}}{\partial \mathbf x}
(\mathbf x)\xi\|\le L_u\|\xi\|\), and
\(\|\frac{\partial \mathbf B_{\mathrm d}}{\partial \mathbf x}
(\mathbf x)[\xi]\eta\|\le L_B\|\xi\|\,\|\eta\|\). Since
\(\|\dot{\mathbf x}\|\le L_r \varepsilon_{\mathrm d}(\mathbf x)+\|\hat{\mathbf w}\|\)
and \(\|\hat{\mathbf u}\|\le\underline b^{-1}\|\hat{\mathbf w}\|\), the
definition of \(\bm\Delta\) gives
\[
    \|\bm\Delta(\mathbf x,\hat{\mathbf w})\|
    \le
    d_x\varepsilon_{\mathrm d}(\mathbf x)
    +
    d_w\|\hat{\mathbf w}\|,
\]
where \(d_x:=L_uL_r+L_B\underline b^{-1}L_r\rho\) and
\(d_w:=L_u+L_B\underline b^{-1}\rho\). Hence, using
\eqref{eq:N-positive-dist}, \(V_b=\frac12\|\hat{\mathbf w}\|^2\) satisfies
\[
    \dot V_b
    \le
    -
    \frac{2\beta}{\tau}V_b
    +
    d_x\varepsilon_{\mathrm d}(\mathbf x)\|\hat{\mathbf w}\|
    +
    d_w\|\hat{\mathbf w}\|^2 .
\]

Fix any \(\alpha_{\mathrm{SP}}\in(0,\alpha_{\mathrm{FL}})\) and set
\(\gamma:=\alpha_{\mathrm{FL}}-\alpha_{\mathrm{SP}}>0\). Consider the composite
Lyapunov function
\(\bm\nu(\mathbf x,\hat{\mathbf w}):=W(\mathbf x)+V_b\). Combining the preceding
two inequalities gives
\[
\ba
    \dot{\bm\nu}
    &\le
    -2\alpha_{\mathrm{FL}}W
    -
    \frac{2\beta}{\tau}V_b
    +
    (c_3+d_x)\varepsilon_{\mathrm d}(\mathbf x)\|\hat{\mathbf w}\|
    +
    d_w\|\hat{\mathbf w}\|^2 \\
    &\le
    -2\alpha_{\mathrm{FL}}W
    -
    \frac{2\beta}{\tau}V_b
    +
    c_1\gamma \varepsilon_{\mathrm d}(\mathbf x)^2
    +
    \frac{(c_3+d_x)^2}{4c_1\gamma}
    \|\hat{\mathbf w}\|^2
    +
    d_w\|\hat{\mathbf w}\|^2 .
\ea
\]
Since \(W\ge c_1\varepsilon_{\mathrm d}(\mathbf x)^2\), it follows that
\(-2\alpha_{\mathrm{FL}}W+c_1\gamma \varepsilon_{\mathrm d}(\mathbf x)^2
\le -2\alpha_{\mathrm{SP}}W\). Moreover, since
\(\|\hat{\mathbf w}\|^2=2V_b\), if
\[
    0<\tau<\tau^\ast
    :=
    \frac{\beta}{
        \alpha_{\mathrm{SP}}
        +
        d_w
        +
        \frac{(c_3+d_x)^2}{4c_1\gamma}
    },
\]
then
\[
    \frac{\beta}{\tau}
    -
    d_w
    -
    \frac{(c_3+d_x)^2}{4c_1\gamma}
    \ge
    \alpha_{\mathrm{SP}} .
\]
Therefore, \(\dot{\bm\nu}\le -2\alpha_{\mathrm{SP}}\bm\nu\). Let
\(m_\nu:=\min\{c_1,\frac12\}\) and \(M_\nu:=\max\{c_2,\frac12\}\). Then
\(m_\nu(\varepsilon_{\mathrm d}(\mathbf x)^2+\|\hat{\mathbf w}\|^2)\le\bm\nu
\le M_\nu(\varepsilon_{\mathrm d}(\mathbf x)^2+\|\hat{\mathbf w}\|^2)\). Hence
\begin{equation}
\label{eq:conver_SP_DIS}
    \sqrt{
        \varepsilon_{\mathrm d}(\mathbf x(t))^2
        +
        \|\hat{\mathbf w}(t)\|^2
    }
    \le
    \sqrt{\frac{M_\nu}{m_\nu}}
    e^{-\alpha_{\mathrm{SP}}t}
    \sqrt{
        \varepsilon_{\mathrm d}(\mathbf x(0))^2
        +
        \|\hat{\mathbf w}(0)\|^2
    } .
\end{equation}
It remains to translate this estimate back to the
\((\mathbf x,\mathbf z)\)-coordinates. For each fixed
\(\tau\in(0,\tau^\ast)\), choose local constants \(L_y,L_q,C_e>0\) such that
\(\|\mathbf y_{\mathrm d}(\mathbf x)\|\le L_y\varepsilon_{\mathrm d}(\mathbf x)\),
\(\|\mathbf u_{\mathrm d}^{\rm q}(\mathbf x)-\mathbf u_{\mathrm d}^{\ast}\|
\le L_q\varepsilon_{\mathrm d}(\mathbf x)\), and
\(\varepsilon_{\mathrm d}(\mathbf x)\le C_e\|\mathbf x-\mathbf x_{\mathrm d}^{\ast}\|\).
Let \(c_\tau:=\|\mathbf K_p\|L_y+L_q\). Since
\(\mathbf u_{\mathrm d}=\mathbf z+\mathbf K_p\mathbf y_{\mathrm d}(\mathbf x)\),
\(\mathbf z^\ast=\mathbf u_{\mathrm d}^{\ast}\), and
\(\|\hat{\mathbf w}\|\le\overline b\|\hat{\mathbf u}\|\), one has
\[
    \|\hat{\mathbf w}(0)\|
    \le
    \overline b
    \left(
        \|\mathbf z(0)-\mathbf z^\ast\|
        +
        c_\tau\varepsilon_{\mathrm d}(\mathbf x(0))
    \right).
\]
Together with
\(\varepsilon_{\mathrm d}(\mathbf x(0))
\le C_e\|\mathbf x(0)-\mathbf x_{\mathrm d}^{\ast}\|\), this gives
\[
    \sqrt{
        \varepsilon_{\mathrm d}(\mathbf x(0))^2
        +
        \|\hat{\mathbf w}(0)\|^2
    }
    \le
    C_0
    \left(
        \|\mathbf x(0)-\mathbf x_{\mathrm d}^{\ast}\|
        +
        \|\mathbf z(0)-\mathbf z^\ast\|
    \right),
\]
where \(C_0:=\max\{C_e(1+\overline b c_\tau),\overline b\}\).

Next, using
\(\mathbf z=\mathbf u_{\mathrm d}-\mathbf K_p\mathbf y_{\mathrm d}(\mathbf x)\),
\(\mathbf z^\ast=\mathbf u_{\mathrm d}^{\ast}\),
\(\|\hat{\mathbf u}\|\le\underline b^{-1}\|\hat{\mathbf w}\|\), and the
definitions of \(L_y\), \(L_q\), and \(c_\tau\), we obtain
\[
    \|\mathbf z(t)-\mathbf z^\ast\|
    \le
    \underline b^{-1}\|\hat{\mathbf w}(t)\|
    +
    c_\tau\varepsilon_{\mathrm d}(\mathbf x(t)).
\]
Therefore, with \(C_z:=1+c_\tau+\underline b^{-1}\), it follows that
\[
    \varepsilon_{\mathrm d}(\mathbf x(t))
    +
    \|\mathbf z(t)-\mathbf z^\ast\|
    \le
    C_z
    \sqrt{
        \varepsilon_{\mathrm d}(\mathbf x(t))^2
        +
        \|\hat{\mathbf w}(t)\|^2
    }.
\]
Combining this inequality with \eqref{eq:conver_SP_DIS} yields
\[
\begin{aligned}
    &\|\mathbf x(t)-\mathbf x_{\mathrm d}^{\ast}\|
    +
    \|\mathbf z(t)-\mathbf z^\ast\|
    +
    \|\mathbf L\mathbf x(t)\|
    +
    \|\mathbf h^{\mathrm d}(\mathbf x(t))\| \\
    &\qquad\le
    C_{\mathrm{SP}}e^{-\alpha_{\mathrm{SP}}t}
    \left(
        \|\mathbf x(0)-\mathbf x_{\mathrm d}^{\ast}\|
        +
        \|\mathbf z(0)-\mathbf z^\ast\|
    \right),
\end{aligned}
\]
where \(C_{\mathrm{SP}}:=C_z\sqrt{M_\nu/m_\nu}\,C_0\).

Finally, choose
\(\delta_{\mathrm{SP}}:=\min\{\delta_0,\,
C_0^{-1}\sqrt{m_\nu/M_\nu}\rho\}\), where \(\delta_0>0\) is sufficiently small
so that all local estimates above hold. Then
\(\|\mathbf x(0)-\mathbf x_{\mathrm d}^{\ast}\|
+\|\mathbf z(0)-\mathbf z^\ast\|<\delta_{\mathrm{SP}}\) implies
\(
    \sqrt{
        \varepsilon_{\mathrm d}(\mathbf x(0))^2
        +
        \|\hat{\mathbf w}(0)\|^2
    }
    <
    \sqrt{\frac{m_\nu}{M_\nu}}\rho .
\)
Thus \(\bm\nu(0)<m_\nu\rho^2\). Since
\(\dot{\bm\nu}\le-2\alpha_{\mathrm{SP}}\bm\nu\), the set
\(\{\bm\nu\le m_\nu\rho^2\}\) is forward invariant. Hence
\(\varepsilon_{\mathrm d}(\mathbf x(t))\le\rho\) and
\(\|\hat{\mathbf w}(t)\|\le\rho\) for all \(t\ge0\). Therefore, the trajectory
stays in the local region, and the estimate \eqref{eq:sp-bound-dist}
follows. Since \(\gamma=\alpha_{\mathrm{FL}}-\alpha_{\mathrm{SP}}>0\) can be
chosen arbitrarily small, this completes the proof.

\section{Proof of Proposition~\ref{prop:fl-coupled}}
\label{app:fl-coupled}

The proof of Proposition~\ref{prop:fl-coupled} follows the same line of argument as the proof of Proposition~\ref{prop:fl-dist} in Appendix~\ref{app:fl-dist}. Therefore, we present a streamlined proof below, emphasizing only the main differences from Proposition~\ref{prop:fl-dist}.

Since \(\bm\zeta(0)\in\operatorname{Im}\mathbf L_q\) and by \eqref{eq:ideal-fl-algorithm-coupled},
in this proof, we work on the invariant subspace
\(\bm\zeta\in\operatorname{Im}\mathbf L_q\). Since
\(\mathbf x_{\mathrm c}^{\ast}=\mathbf 1_N\otimes\mathbf s_{\mathrm c}^{\ast}\)
and \(\mathbf s_{\mathrm c}^{\ast}\in\Omega_{\mathrm c}\), one has
\((\mathbf 1_N^\top\otimes\mathbf I_q)\mathbf h^{\mathrm c}
(\mathbf x_{\mathrm c}^{\ast})=\mathbf 0\). Hence
\(\mathbf h^{\mathrm c}(\mathbf x_{\mathrm c}^{\ast})\in\operatorname{Im}\mathbf L_q\).
Because \(\mathbf L_q\) is nonsingular on \(\operatorname{Im}\mathbf L_q\),
there exists a unique \(\bm\zeta^\ast\in\operatorname{Im}\mathbf L_q\) such
that
\(\mathbf h^{\mathrm c}(\mathbf x_{\mathrm c}^{\ast})
+\mathbf L_q\bm\zeta^\ast=\mathbf 0\).

Since \(\mathbf G_{\mathrm c}(\mathbf 0)=\mathbf 0\), the ideal FL law gives
\(\dot{\mathbf y}_{\mathrm c}=\mathbf G_{\mathrm c}(\mathbf y_{\mathrm c})\), so
the zero-output set \(\mathcal Z_{\mathrm c}=\{\mathbf y_{\mathrm c}=\mathbf 0\}\)
is a locally invariant manifold of \eqref{eq:ideal-fl-algorithm-coupled}, and by the above design of \(\mathbf G_{\mathrm c}\), the output converges to it exponentially with
rate \(\alpha_G\).
Moreover, one can introduce a reduced coordinate
\(\bm\eta\in\mathbb R^{Nd-m_{\rm c}}\), with
\(m_{\rm c}:=(N-1)d+q\), to parametrize the tangent motion of the
\(\mathbf x\)-component along \(\mathcal Z_{\mathrm c}\). Under this coordinate
representation, the reduced dynamics coincides with the projected gradient
flow of \(F\) on \(\Omega_{\mathrm c}\), and the restricted quadratic growth in
Assumption~\ref{asm:growth} gives the same dissipation estimate as
\eqref{eq:eta-dissipation-dist}, with \(\rho_{\mathrm d}\) replaced by
\(\rho_{\mathrm c}\). In particular, the corresponding reduced convergence rate
is \(\rho_{\mathrm c}/N\).

Combined with the exponentially stable output subsystem
\(\dot{\mathbf y}_{\mathrm c}=\mathbf G_{\mathrm c}(\mathbf y_{\mathrm c})\), the
composite Lyapunov function \(W=\mu V_1+V_2\) gives
\(\dot W\le-2\alpha_{\mathrm{FL}}W\) for any
\(\alpha_{\mathrm{FL}}\in(0,\min\{\rho_{\mathrm c}/N,\alpha_G\})\), exactly as in
Appendix~\ref{app:fl-dist}. The only difference from Proposition~\ref{prop:fl-dist}
lies in the initial estimate. There, the residual
\(\mathbf h^{\mathrm d}(\mathbf x)\) is a function of \(\mathbf x\) alone, so
the entire error is bounded by
\(\|\mathbf x(0)-\mathbf x_{\mathrm d}^{\ast}\|\). Here, the lifted residual
\(\mathbf h^{\mathrm c}(\mathbf x)+\mathbf L_q\bm\zeta\) depends also on
\(\bm\zeta\), and \eqref{eq:diffeo-bound-dist} becomes
\[
    \left\|
        \operatorname{col}({\bm\eta},\mathbf y_{\mathrm c})
    \right\|
    \le
    L_\Phi'
    \big(
        \|\mathbf x(0)-\mathbf x_{\mathrm c}^{\ast}\|
        +
        \|\mathbf h^{\mathrm c}(\mathbf x(0))+\mathbf L_q\bm\zeta(0)\|
    \big)
\]
for some local constant \(L_\Phi'>0\). Consequently,
\begin{equation}
    \label{eq:conver_cou}
\begin{aligned}
    &\|\mathbf x(t)-\mathbf x_{\mathrm c}^{\ast}\|
    +\|\mathbf L\mathbf x(t)\|
    +\|\mathbf h^{\mathrm c}(\mathbf x(t))+\mathbf L_q\bm\zeta(t)\|\\
    &\qquad\le
    \bar C_{\mathrm{FL}}e^{-\alpha_{\mathrm{FL}}t}
    \big(
        \|\mathbf x(0)-\mathbf x_{\mathrm c}^{\ast}\|
        +\|\mathbf h^{\mathrm c}(\mathbf x(0))+\mathbf L_q\bm\zeta(0)\|
    \big),
    \qquad t\ge0,
\end{aligned}
\end{equation}
where \(\bar C_{\mathrm{FL}}:=(L_\Psi+c_y)c_0L_\Phi'\).

It remains to include \(\bm\zeta(t)-\bm\zeta^\ast\) in the estimate. Let
\(\sigma_q>0\) be the smallest positive singular value of \(\mathbf L_q\). Since
\(\bm\zeta(t),\bm\zeta^\ast\in\operatorname{Im}\mathbf L_q\), and
\(\mathbf h^{\mathrm c}\) is locally Lipschitz, there exists \(L_h>0\) such that
\(\|\mathbf h^{\mathrm c}(\mathbf x)-\mathbf h^{\mathrm c}(\mathbf x_{\mathrm c}^{\ast})\|
\le L_h\|\mathbf x-\mathbf x_{\mathrm c}^{\ast}\|\). Using
\(\mathbf h^{\mathrm c}(\mathbf x_{\mathrm c}^{\ast})+\mathbf L_q\bm\zeta^\ast=\mathbf 0\),
we have
\[
    \|\bm\zeta(t)-\bm\zeta^\ast\|
    \le
    \sigma_q^{-1}
    \left(
        \|\mathbf h^{\mathrm c}(\mathbf x(t))+\mathbf L_q\bm\zeta(t)\|
        +
        L_h\|\mathbf x(t)-\mathbf x_{\mathrm c}^{\ast}\|
    \right).
\]
Therefore, with \eqref{eq:conver_cou},
\[
\begin{aligned}
    &\|\mathbf x(t)-\mathbf x_{\mathrm c}^{\ast}\|
    +\|\bm\zeta(t)-\bm\zeta^\ast\|
    +\|\mathbf L\mathbf x(t)\|
    +\|\mathbf h^{\mathrm c}(\mathbf x(t))+\mathbf L_q\bm\zeta(t)\|\\
    &\qquad\le
    c_\zeta \bar C_{\mathrm{FL}}e^{-\alpha_{\mathrm{FL}}t}
    \big(
        \|\mathbf x(0)-\mathbf x_{\mathrm c}^{\ast}\|
        +\|\mathbf h^{\mathrm c}(\mathbf x(0))+\mathbf L_q\bm\zeta(0)\|
    \big),
\end{aligned}
\]
where \(c_\zeta:=1+\sigma_q^{-1}(1+L_h)\). Finally, by local Lipschitz
continuity of \(\mathbf h^{\mathrm c}\) and
\(\mathbf h^{\mathrm c}(\mathbf x_{\mathrm c}^{\ast})+\mathbf L_q\bm\zeta^\ast=\mathbf 0\),
\[
    \|\mathbf h^{\mathrm c}(\mathbf x(0))+\mathbf L_q\bm\zeta(0)\|
    \le
    L_h\|\mathbf x(0)-\mathbf x_{\mathrm c}^{\ast}\|
    +
    \|\mathbf L_q\|\,\|\bm\zeta(0)-\bm\zeta^\ast\|.
\]
Thus the preceding estimate implies
\[
\begin{aligned}
    &\|\mathbf x(t)-\mathbf x_{\mathrm c}^{\ast}\|
    +\|\bm\zeta(t)-\bm\zeta^\ast\|
    +\|\mathbf L\mathbf x(t)\|
    +\|\mathbf h^{\mathrm c}(\mathbf x(t))+\mathbf L_q\bm\zeta(t)\|\\
    &\qquad\le
    C_{\mathrm{FL}}e^{-\alpha_{\mathrm{FL}}t}
    \left(
        \|\mathbf x(0)-\mathbf x_{\mathrm c}^{\ast}\|
        +
        \|\bm\zeta(0)-\bm\zeta^\ast\|
    \right),
    \qquad t\ge0,
\end{aligned}
\]
where
\(
    C_{\mathrm{FL}}
    :=
        c_\zeta \bar C_{\mathrm{FL}}
        \max\{1+L_h,\|\mathbf L_q\|\}.
\)
This completes the proof.

\section{Proof of Theorem~\ref{thm:sp-coupled}}
\label{app:sp-coupled}

We first prove that \eqref{eq:sp-compact-coupled} admits a locally unique
equilibrium. Since \eqref{eq:sp-compact-coupled} is equivalent to
\eqref{eq:fast-algorithm-coupled} via
\(\mathbf u_{\mathrm c}=\mathbf z+\mathbf K_p\mathbf y_{\mathrm c}\), it suffices
to determine the equilibria of \eqref{eq:fast-algorithm-coupled} in the
\((\bm\chi,\mathbf u_{\mathrm c})\)-coordinates, restricted to the invariant
subspace \(\bm\zeta\in\operatorname{Im}\mathbf L_q\). At an equilibrium,
\(\dot{\mathbf u}_{\mathrm c}=\mathbf0\) gives the algebraic relation of the
ideal FL law. Hence, by Lemma~\ref{lem:fl-wp-coupled}, the input is uniquely
determined by \(\bm\chi\), i.e.,
\(\mathbf u_{\mathrm c}=\mathbf u_{\mathrm c}^{\rm q}(\bm\chi)\). Substituting
this relation into \(\dot{\bm\chi}=\mathbf0\) gives the ideal FL dynamics
\eqref{eq:ideal-fl-algorithm-coupled} at rest. By
Proposition~\ref{prop:fl-coupled}, \(\bm\chi_{\mathrm c}^{\ast}:=\operatorname{col}
(\mathbf x_{\mathrm c}^{\ast},\bm\zeta^\ast)\) is a locally exponentially
stable, and hence locally isolated, equilibrium of the ideal FL dynamics. Denote
\(\mathbf u_{\mathrm c}^{\ast}:=\mathbf u_{\mathrm c}^{\rm q}
(\bm\chi_{\mathrm c}^{\ast})\). Consequently, \eqref{eq:sp-compact-coupled}
has the locally unique equilibrium
\((\mathbf x_{\mathrm c}^{\ast},\bm\zeta^\ast,\mathbf z^\ast)\), where
\(\mathbf z^\ast=\mathbf u_{\mathrm c}^{\ast}
-\mathbf K_p\mathbf y_{\mathrm c}(\bm\chi_{\mathrm c}^{\ast})
=\mathbf u_{\mathrm c}^{\ast}\).

Let \(\hat{\mathbf u}:=\mathbf u_{\mathrm c}
-\mathbf u_{\mathrm c}^{\rm q}(\bm\chi)\), with
\(\hat{\mathbf u}=\operatorname{col}(\hat{\mathbf v},\hat{\bm\mu})\). With
\(\bm\chi\) frozen and with the fast time \(s:=t/\tau\), the fast equation in
\eqref{eq:fast-algorithm-coupled} gives
\[
    \frac{d\hat{\mathbf u}}{ds}
    =
    -\mathbf K_0\mathbf M_{\mathrm c}(\bm\chi)\hat{\mathbf u}.
\]
Define the effective fast error as
\(
    \hat{\mathbf w}
    :=
    \operatorname{col}(\hat{\mathbf w}_x,\hat{\mathbf w}_\zeta)
    :=\mathbf B_{\mathrm c}(\bm\chi)\hat{\mathbf u}=
    \operatorname{col}
    \big(
        \hat{\mathbf v}
        +\tilde{\mathbf J}^{\mathrm c}(\mathbf x)^\top\hat{\bm\mu},
        \mathbf L_q\hat{\bm\mu}
    \big)
    \in\mathcal N_{\mathrm c}(\bm\chi),
\)
where
\[
    \mathcal N_{\mathrm c}(\bm\chi)
    :=
    \left\{
    \operatorname{col}
    \big(
        \mathbf v+
        \tilde{\mathbf J}^{\mathrm c}(\mathbf x)^\top\bm\mu,
        \mathbf L_q\bm\mu
    \big):
    \mathbf v\in\operatorname{Im}\mathbf L,\
    \bm\mu\in\mathbb R^{Nq}
    \right\}.
\]
Using the block form of \(\mathbf M_{\mathrm c}\), one obtains
\[
    \frac{d\hat{\mathbf w}}{ds}
    =
    -\mathbf N_{\mathrm c}(\bm\chi)\hat{\mathbf w},
\]
where
\[
\mathbf N_{\mathrm c}(\bm\chi)
:=\mathbf B_{\mathrm c}(\bm\chi)\mathbf K_0\mathbf A_{\mathrm c}(\bm\chi)=
\begin{bmatrix}
K_{0c}\mathbf L
+K_{0h}\tilde{\mathbf J}^{\mathrm c}(\mathbf x)^\top
\tilde{\mathbf J}^{\mathrm c}(\mathbf x)
&
K_{0h}\tilde{\mathbf J}^{\mathrm c}(\mathbf x)^\top\mathbf L_q
\\
K_{0h}\mathbf L_q\tilde{\mathbf J}^{\mathrm c}(\mathbf x)
&
K_{0h}\mathbf L_q^2
\end{bmatrix}.
\]

We next show that $\mathbf N_{\mathrm c}(\bm\chi)$ is positive definite on
\(\mathcal N_{\mathrm c}(\bm\chi)\). Take any
\(\mathbf w=\operatorname{col}(\mathbf w_x,\mathbf w_\zeta)
\in\mathcal N_{\mathrm c}(\bm\chi)\). Then
\[
    \mathbf w^\top\mathbf N_{\mathrm c}(\bm\chi)\mathbf w
    =
    K_{0c}\mathbf w_x^\top\mathbf L\mathbf w_x
    +
    K_{0h}
    \left\|
        \tilde{\mathbf J}^{\mathrm c}(\mathbf x)\mathbf w_x
        +
        \mathbf L_q\mathbf w_\zeta
    \right\|^2
    \ge0 .
\]
If the above quantity is zero, then
\(\mathbf L\mathbf w_x=\mathbf0\) and
\(\tilde{\mathbf J}^{\mathrm c}(\mathbf x)\mathbf w_x
+\mathbf L_q\mathbf w_\zeta=\mathbf0\). Since
\(\mathbf w\in\mathcal N_{\mathrm c}(\bm\chi)\), there exist
\(\mathbf v\in\operatorname{Im}\mathbf L\) and
\(\bm\mu\in\mathbb R^{Nq}\) such that
\(\mathbf w_x=\mathbf v+\tilde{\mathbf J}^{\mathrm c}(\mathbf x)^\top\bm\mu\)
and \(\mathbf w_\zeta=\mathbf L_q\bm\mu\). Therefore,
\[
\begin{aligned}
    \|\mathbf w\|^2
    &=
    \mathbf w_x^\top
    \big(
        \mathbf v+\tilde{\mathbf J}^{\mathrm c}(\mathbf x)^\top\bm\mu
    \big)
    +
    \mathbf w_\zeta^\top\mathbf L_q\bm\mu  \\
    &=
    \mathbf w_x^\top\mathbf v
    +
    \bm\mu^\top
    \big(
        \tilde{\mathbf J}^{\mathrm c}(\mathbf x)\mathbf w_x
        +
        \mathbf L_q\mathbf w_\zeta
    \big)
    =0 .
\end{aligned}
\]
Here we used \(\mathbf w_x\in\ker\mathbf L\), 
\(\mathbf v\in\operatorname{Im}\mathbf L\), and hence
\(\mathbf w_x^\top\mathbf v=0\). Thus \(\mathbf w=\mathbf0\), which proves
that \(\mathbf N_{\mathrm c}(\bm\chi)\) is positive definite on
\(\mathcal N_{\mathrm c}(\bm\chi)\). By continuity, after shrinking the local
neighborhood if necessary, there exists \(\beta_{\mathrm c}>0\) such that
\[
    \mathbf w^\top\mathbf N_{\mathrm c}(\bm\chi)\mathbf w
    \ge
    \beta_{\mathrm c}\|\mathbf w\|^2,
    \qquad
    \forall \mathbf w\in\mathcal N_{\mathrm c}(\bm\chi).
\]
Consequently, for \(V_b:=\frac12\|\hat{\mathbf w}\|^2\),
\[
    \frac{dV_b}{ds}
    =
    -\hat{\mathbf w}^\top
    \mathbf N_{\mathrm c}(\bm\chi)\hat{\mathbf w}
    \le
    -2\beta_{\mathrm c}V_b .
\]
Hence the boundary-layer system is locally uniformly exponentially stable.

We also need the equivalence between \(\hat{\mathbf u}\) and
\(\hat{\mathbf w}\). The map
\(\operatorname{col}(\mathbf v,\bm\mu)\mapsto
\operatorname{col}(\mathbf v+\tilde{\mathbf J}^{\mathrm c}(\mathbf x)^\top\bm\mu,
\mathbf L_q\bm\mu)\) is locally nonsingular on
\(\operatorname{Im}\mathbf L\times\mathbb R^{Nq}\). Indeed, if
\(\mathbf v+\tilde{\mathbf J}^{\mathrm c}(\mathbf x)^\top\bm\mu=\mathbf0\) and
\(\mathbf L_q\bm\mu=\mathbf0\), then
\(\bm\mu=\mathbf1_N\otimes\beta\) for some \(\beta\in\mathbb R^q\). Since
\(\mathbf v\in\operatorname{Im}\mathbf L\), premultiplication by
\(\mathbf1_N^\top\otimes\mathbf I_d\) gives
\(\sum_{i=1}^{N}\mathbf J_i^{\mathrm c}(\mathbf x_i)^\top\beta=\mathbf0\).
By Assumption~\ref{asm:reg-coupled}, this yields \(\beta=\mathbf0\), and hence
\(\bm\mu=\mathbf0\) and \(\mathbf v=\mathbf0\). Therefore, there exist
constants \(\underline b,\overline b>0\) such that, locally, for all
\(\operatorname{col}(\mathbf v,\bm\mu)\in
\operatorname{Im}\mathbf L\times\mathbb R^{Nq}\),
\[
\begin{aligned}
    \underline b
    \left\|
        \operatorname{col}(\mathbf v,\bm\mu)
    \right\|
    &\le
    \left\|
        \operatorname{col}
        \big(
            \mathbf v+\tilde{\mathbf J}^{\mathrm c}(\mathbf x)^\top\bm\mu,
            \mathbf L_q\bm\mu
        \big)
    \right\|  \le
    \overline b
    \left\|
        \operatorname{col}(\mathbf v,\bm\mu)
    \right\|.
\end{aligned}
\]
In particular, \(\hat{\mathbf w}=\mathbf0\) is equivalent to
\(\hat{\mathbf u}=\mathbf0\). Hence the reduced system is obtained by setting
\(\mathbf u_{\mathrm c}=\mathbf u_{\mathrm c}^{\rm q}(\bm\chi)\), and is exactly
the ideal FL dynamics \eqref{eq:ideal-fl-algorithm-coupled}.

By Proposition~\ref{prop:fl-coupled}, for any admissible
\(\alpha_{\mathrm{FL}}\in(0,\min\{\rho_{\mathrm c}/N,\alpha_G\})\), there is a
local Lyapunov function \(W(\bm\chi)\) and constants \(c_1,c_2,c_3>0\) such that,
with
\[
    \varepsilon_{\mathrm c}(\bm\chi)
    :=
    \|\mathbf x-\mathbf x_{\mathrm c}^{\ast}\|
    +
    \|\bm\zeta-\bm\zeta^\ast\|
    +
    \|\mathbf L\mathbf x\|
    +
    \|\mathbf h^{\mathrm c}(\mathbf x)+\mathbf L_q\bm\zeta\|,
\]
one has
\[
\begin{aligned}
    &c_1\varepsilon_{\mathrm c}(\bm\chi)^2
    \le W(\bm\chi)
    \le c_2\varepsilon_{\mathrm c}(\bm\chi)^2,\\
    &\nabla W(\bm\chi)^\top\dot{\bm\chi}_{\rm red}
    \le -2\alpha_{\mathrm{FL}}W(\bm\chi),\\
    &\|\nabla W(\bm\chi)\|
    \le c_3\varepsilon_{\mathrm c}(\bm\chi),
\end{aligned}
\]
where \(\dot{\bm\chi}_{\rm red}\) denotes the vector field of the ideal FL
dynamics.

For the full system, the definition of \(\hat{\mathbf w}\) gives
\[
\begin{aligned}
    \dot{\bm\chi}
    &=
    \dot{\bm\chi}_{\rm red}-\hat{\mathbf w},\\
    \dot{\hat{\mathbf w}}
    &=
    -\frac1\tau\mathbf N_{\mathrm c}(\bm\chi)\hat{\mathbf w}
    +\bm\Delta(\bm\chi,\hat{\mathbf w}),
\end{aligned}
\]
where
\(
    \bm\Delta(\bm\chi,\hat{\mathbf w})
    =
    \frac{\partial \mathbf B_{\mathrm c}}{\partial \bm\chi}(\bm\chi)
    [\dot{\bm\chi}]\hat{\mathbf u}
    -
    \mathbf B_{\mathrm c}(\bm\chi)
    \frac{\partial \mathbf u_{\mathrm c}^{\rm q}}
    {\partial \bm\chi}(\bm\chi)\dot{\bm\chi}.
\)

We now estimate the two components of the full system. From the above Lyapunov
estimate and \(\dot{\bm\chi}=\dot{\bm\chi}_{\rm red}-\hat{\mathbf w}\), one has
\[
    \dot W
    \le
    -2\alpha_{\mathrm{FL}}W
    +
    c_3\varepsilon_{\mathrm c}(\bm\chi)\|\hat{\mathbf w}\|.
\]
Fix a sufficiently small radius \(\rho>0\) such that all the preceding local
estimates hold whenever \(\varepsilon_{\mathrm c}(\bm\chi)\le\rho\) and
\(\|\hat{\mathbf w}\|\le\rho\). On this region, choose constants
\(L_r,L_u,L_B>0\) such that
\(\|\dot{\bm\chi}_{\rm red}\|\le L_r\varepsilon_{\mathrm c}(\bm\chi)\),
\(\|\mathbf B_{\mathrm c}(\bm\chi)
\frac{\partial \mathbf u_{\mathrm c}^{\rm q}}{\partial \bm\chi}(\bm\chi)\xi\|
\le L_u\|\xi\|\), and
\(\|\frac{\partial \mathbf B_{\mathrm c}}{\partial \bm\chi}(\bm\chi)[\xi]\eta\|
\le L_B\|\xi\|\,\|\eta\|\). Since
\(\|\dot{\bm\chi}\|\le L_r\varepsilon_{\mathrm c}(\bm\chi)+\|\hat{\mathbf w}\|\) and
\(\|\hat{\mathbf u}\|\le\underline b^{-1}\|\hat{\mathbf w}\|\), the definition
of \(\bm\Delta\) gives
\[
\begin{aligned}
    \|\bm\Delta(\bm\chi,\hat{\mathbf w})\|
    &\le
    d_x\varepsilon_{\mathrm c}(\bm\chi)
    +
    d_w\|\hat{\mathbf w}\|,
\end{aligned}
\]
where \(d_x:=L_uL_r+L_B\underline b^{-1}L_r\rho\) and
\(d_w:=L_u+L_B\underline b^{-1}\rho\). Using the positive definiteness of
\(\mathbf N_{\mathrm c}\) on \(\mathcal N_{\mathrm c}\), we also have
\[
    \dot V_b
    \le
    -\frac{2\beta}{\tau}V_b
    +
    d_x\varepsilon_{\mathrm c}(\bm\chi)\|\hat{\mathbf w}\|
    +
    d_w\|\hat{\mathbf w}\|^2 .
\]

Fix \(\alpha_{\mathrm{SP}}\in(0,\alpha_{\mathrm{FL}})\) and set
\(\gamma:=\alpha_{\mathrm{FL}}-\alpha_{\mathrm{SP}}>0\). Consider
\(\bm\nu(\bm\chi,\hat{\mathbf w}):=W(\bm\chi)+V_b\). Combining the preceding two
inequalities yields
\[
\begin{aligned}
    \dot{\bm\nu}
    &\le
    -2\alpha_{\mathrm{FL}}W
    -\frac{2\beta}{\tau}V_b
    +(c_3+d_x)\varepsilon_{\mathrm c}(\bm\chi)\|\hat{\mathbf w}\|
    +d_w\|\hat{\mathbf w}\|^2  \\
    &\le
    -2\alpha_{\mathrm{FL}}W
    -\frac{2\beta}{\tau}V_b
    +c_1\gamma \varepsilon_{\mathrm c}(\bm\chi)^2
    +
    \frac{(c_3+d_x)^2}{4c_1\gamma}\|\hat{\mathbf w}\|^2
    +d_w\|\hat{\mathbf w}\|^2 .
\end{aligned}
\]
Since \(W\ge c_1\varepsilon_{\mathrm c}(\bm\chi)^2\), it follows that
\(-2\alpha_{\mathrm{FL}}W+c_1\gamma \varepsilon_{\mathrm c}(\bm\chi)^2
\le -2\alpha_{\mathrm{SP}}W\). Moreover, since
\(\|\hat{\mathbf w}\|^2=2V_b\), if
\[
    0<\tau<\tau^\ast
    :=
    \frac{\beta}{
        \alpha_{\mathrm{SP}}
        +d_w
        +\frac{(c_3+d_x)^2}{4c_1\gamma}
    },
\]
then
\[
    \frac{\beta}{\tau}
    -d_w
    -\frac{(c_3+d_x)^2}{4c_1\gamma}
    \ge
    \alpha_{\mathrm{SP}} .
\]
Therefore, \(\dot{\bm\nu}\le-2\alpha_{\mathrm{SP}}\bm\nu\). Let
\(m_\nu:=\min\{c_1,\frac12\}\) and \(M_\nu:=\max\{c_2,\frac12\}\). Then
\[
    m_\nu
    \big(\varepsilon_{\mathrm c}(\bm\chi)^2+\|\hat{\mathbf w}\|^2\big)
    \le
    \bm\nu
    \le
    M_\nu
    \big(\varepsilon_{\mathrm c}(\bm\chi)^2+\|\hat{\mathbf w}\|^2\big).
\]
Thus
\[
    \sqrt{\varepsilon_{\mathrm c}(\bm\chi(t))^2+\|\hat{\mathbf w}(t)\|^2}
    \le
    \sqrt{\frac{M_\nu}{m_\nu}}
    e^{-\alpha_{\mathrm{SP}}t}
    \sqrt{\varepsilon_{\mathrm c}(\bm\chi(0))^2+\|\hat{\mathbf w}(0)\|^2}.
\]

It remains to translate this estimate back to the
\((\mathbf x,\bm\zeta,\mathbf z)\)-coordinates. For each fixed
\(\tau\in(0,\tau^\ast)\), let \(L_y,L_q>0\) be local constants such that
\(\|\mathbf y_{\mathrm c}(\bm\chi)\|\le L_y\varepsilon_{\mathrm c}(\bm\chi)\) and
\(\|\mathbf u_{\mathrm c}^{\rm q}(\bm\chi)-\mathbf u_{\mathrm c}^{\ast}\|
\le L_q\varepsilon_{\mathrm c}(\bm\chi)\). Put
\(c_\tau:=\|\mathbf K_p\|L_y+L_q\). Since
\(\mathbf u_{\mathrm c}=\mathbf z+\mathbf K_p\mathbf y_{\mathrm c}\) and
\(\mathbf z^\ast=\mathbf u_{\mathrm c}^{\ast}\), the norm equivalence above
implies
\[
\begin{aligned}
    \|\hat{\mathbf w}(0)\|
    &\le
    \overline b
    \|\mathbf u_{\mathrm c}(0)-
    \mathbf u_{\mathrm c}^{\rm q}(\bm\chi(0))\| \\
    &\le
    \overline b
    \Big(
        \|\mathbf z(0)-\mathbf z^\ast\|
        +
        c_\tau \varepsilon_{\mathrm c}(\bm\chi(0))
    \Big).
\end{aligned}
\]
Moreover, by local Lipschitz continuity of \(\mathbf h^{\mathrm c}\), there
exists \(C_e>0\) such that, locally,
\[
    \varepsilon_{\mathrm c}(\bm\chi)
    \le
    C_e
    \left(
        \|\mathbf x-\mathbf x_{\mathrm c}^{\ast}\|
        +
        \|\bm\zeta-\bm\zeta^\ast\|
    \right).
\]
Hence
\[
\begin{aligned}
    \sqrt{\varepsilon_{\mathrm c}(\bm\chi(0))^2+
    \|\hat{\mathbf w}(0)\|^2}
    &\le
    \varepsilon_{\mathrm c}(\bm\chi(0))+
    \|\hat{\mathbf w}(0)\| \\
    &\le
    C_0
    \left(
        \|\mathbf x(0)-\mathbf x_{\mathrm c}^{\ast}\|
        +
        \|\bm\zeta(0)-\bm\zeta^\ast\|
        +
        \|\mathbf z(0)-\mathbf z^\ast\|
    \right),
\end{aligned}
\]
where
\(C_0:=\max\{C_e(1+\overline b c_\tau),\overline b\}\).

Furthermore,
\[
\begin{aligned}
    \|\mathbf z(t)-\mathbf z^\ast\|
    &\le
    \|\hat{\mathbf u}(t)\|
    +
    \|\mathbf u_{\mathrm c}^{\rm q}(\bm\chi(t))-
      \mathbf u_{\mathrm c}^{\ast}\|
    +
    \|\mathbf K_p\|\,\|\mathbf y_{\mathrm c}(\bm\chi(t))\| \\
    &\le
    \underline b^{-1}\|\hat{\mathbf w}(t)\|
    +
    c_\tau \varepsilon_{\mathrm c}(\bm\chi(t)).
\end{aligned}
\]
Therefore, with \(C_z:=1+c_\tau+\underline b^{-1}\), one has
\[
    \varepsilon_{\mathrm c}(\bm\chi(t))
    +
    \|\mathbf z(t)-\mathbf z^\ast\|
    \le
    C_z
    \sqrt{
        \varepsilon_{\mathrm c}(\bm\chi(t))^2
        +
        \|\hat{\mathbf w}(t)\|^2
    }.
\]
Combining the preceding estimates gives
\[
\begin{aligned}
    &\|\mathbf x(t)-\mathbf x_{\mathrm c}^{\ast}\|
    +\|\bm\zeta(t)-\bm\zeta^\ast\|
    +\|\mathbf z(t)-\mathbf z^\ast\|
    +\|\mathbf L\mathbf x(t)\|
    +\|\mathbf h^{\mathrm c}(\mathbf x(t))+\mathbf L_q\bm\zeta(t)\| \\
    &\qquad\le
    C_{\mathrm{SP}}e^{-\alpha_{\mathrm{SP}}t}
    \left(
        \|\mathbf x(0)-\mathbf x_{\mathrm c}^{\ast}\|
        +
        \|\bm\zeta(0)-\bm\zeta^\ast\|
        +
        \|\mathbf z(0)-\mathbf z^\ast\|
    \right),
\end{aligned}
\]
where 
\(
    C_{\mathrm{SP}}
    :=
        C_z\sqrt{\frac{M_\nu}{m_\nu}}\,C_0.
\)

Finally, let \(\delta_{\mathrm{SP}}:=\min\{\delta_0,\,
C_0^{-1}\sqrt{m_\nu/M_\nu}\rho\}\), where \(\delta_0>0\) is sufficiently
small so that all local estimates above hold. Then
\(\|\mathbf x(0)-\mathbf x_{\mathrm c}^{\ast}\|
+\|\bm\zeta(0)-\bm\zeta^\ast\|
+\|\mathbf z(0)-\mathbf z^\ast\|<\delta_{\mathrm{SP}}\) implies
\[
    \sqrt{
        \varepsilon_{\mathrm c}(\bm\chi(0))^2+
        \|\hat{\mathbf w}(0)\|^2
    }
    <
    \sqrt{\frac{m_\nu}{M_\nu}}\rho .
\]
Thus \(\bm\nu(0)<m_\nu\rho^2\). Since
\(\dot{\bm\nu}\le-2\alpha_{\mathrm{SP}}\bm\nu\), the set
\(\{\bm\nu\le m_\nu\rho^2\}\) is forward invariant. Hence
\(\varepsilon_{\mathrm c}(\bm\chi(t))\le\rho\) and
\(\|\hat{\mathbf w}(t)\|\le\rho\) for all \(t\ge0\). Therefore, the trajectory
stays in the local region, and the estimate \eqref{eq:sp-bound-coupled}
follows. Since $\gamma=\alpha_{\mathrm{FL}}-\alpha_{\mathrm{SP}}>0$ can be chosen arbitrarily small, this completes the proof.

\section{Proof of Theorem \ref{thm:sp-dist-enlarged-roa}}

The proof consists of two steps. We first show that the ideal FL dynamics \eqref{eq:ideal-fl-algorithm-dist}  has
an enlarged region of attraction in a tubular neighborhood of the feasible
manifold. Then we invoke the same singular-perturbation argument for \eqref{eq:sp-compact-dist} as in the proof
of Theorem~\ref{thm:sp-dist}.
Since the proof follows the same route as that of Theorem \ref{thm:sp-dist}, we only outline the main steps below and omit detailed arguments such as the specific parameter choices.

Extending the result in Lemma \ref{lem:fl-wp-dist} to \(\xb\in\mathcal N_{\mathrm d}(\varepsilon)\), 
by Assumptions~\ref{asm:graph}and \ref{asm:tube-reg-dist}, the FL algebraic relation \eqref{eq:ideal-fl-relation-dist} admits a
unique admissible solution for all
\(\xb\in\mathcal N_{\mathrm d}(\varepsilon)\), denoted by
\(\ub_{\mathrm d}^{\mathrm q}(\xb)\). Then, for the ideal FL dynamics \eqref{eq:ideal-fl-algorithm-dist} obtained
by setting
\(\ub_{\mathrm d}=\ub_{\mathrm d}^{\mathrm q}(\xb)\), one has
\[
    \dot{\yb}_{\mathrm d}
    =
    \Gb_{\mathrm d}(\yb_{\mathrm d}).
\]
Therefore, by the design of \(\mathbf G_{\mathrm d}\), there exists a Lyapunov
function \(V_y(\yb_{\mathrm d})\) for the output subsystem such that, for some
positive constants \(c_1,c_2\),
\[
    c_1\|\yb_{\mathrm d}\|^2
    \le
    V_y(\yb_{\mathrm d})
    \le
    c_2\|\yb_{\mathrm d}\|^2,
    \qquad
    \dot V_y
    \le
    -2\alpha_GV_y .
\]

It remains to analyze the tangential motion. On the zero-output manifold
\(\mathcal X_{\mathrm d}^{\ast}\), each state can be written as
\(\xb=\mathbf 1_N\otimes\sbf\), where
\(\sbf\in\Omega_{\mathrm d}^{\ast}\). Since the ideal FL input cancels the
normal component of \(-\nabla f(\xb)\) relative to the zero-output manifold,
the zero dynamics is the projected gradient flow
\[
    \dot{\sbf}
    =
    -\frac{1}{N}
    \Pi_{T_{\sbf}\Omega_{\mathrm d}}
    \nabla F(\sbf).
\]
Let
\(
    V_F(\sbf)
    :=
    F(\sbf)-F(\sbf_{\mathrm d}^{\ast}).
\)
Along the zero dynamics, Assumption~\ref{asm:manifold-pl-dist} gives
\[
\begin{aligned}
    \dot V_F
    &=
    -\frac{1}{N}
    \left\|
        \Pi_{T_{\sbf}\Omega_{\mathrm d}}
        \nabla F(\sbf)
    \right\|^2     \\
    &\le
    -\frac{2\mu_{\mathrm d}}{N}
    \left(
        F(\sbf)-F(\sbf_{\mathrm d}^{\ast})
    \right)
    =
    -\frac{2\mu_{\mathrm d}}{N}V_F .
\end{aligned}
\]
Hence the zero dynamics converges exponentially in objective residual on
\(\Omega_{\mathrm d}^{\ast}\).

We next consider the case \(\yb_{\mathrm d}\neq0\) but
\(\xb\in\mathcal N_{\mathrm d}(\varepsilon)\). By
Assumptions~\ref{asm:tube-smooth-dist} and \ref{asm:tube-reg-dist}, the
tubular neighborhood admits a smooth decomposition into tangential and normal
components with uniformly bounded derivatives. Therefore, the tangential
component of the ideal FL dynamics can be written as
\[
    \dot{\sbf}
    =
    -\frac{1}{N}
    \Pi_{T_{\sbf}\Omega_{\mathrm d}}
    \nabla F(\sbf)
    +
    \Delta_s(\sbf,\yb_{\mathrm d}),
\]
where the perturbation term satisfies
\(
    \|\Delta_s(\sbf,\yb_{\mathrm d})\|
    \le
    c_\Delta \|\yb_{\mathrm d}\|
\)
for some constant \(c_\Delta>0\). Consequently,
\[
\begin{aligned}
    \dot V_F
    &=
    \nabla_{\Omega_{\mathrm d}}F(\sbf)^\top\dot{\sbf}  \\
    &\le
    -\frac{1}{N}
    \|\nabla_{\Omega_{\mathrm d}}F(\sbf)\|^2
    +
    c_\Delta
    \|\nabla_{\Omega_{\mathrm d}}F(\sbf)\|
    \|\yb_{\mathrm d}\|                                  \\
    &\le
    -\left(\frac{2\mu_{\mathrm d}}{N}-\varepsilon\right)V_F
    +
    c_\Delta' \|\yb_{\mathrm d}\|^2
\end{aligned}
\]
for arbitrarily small $\varepsilon>0$ and some constant \(c_\Delta'>0\), where
\(\nabla_{\Omega_{\mathrm d}}F(\sbf)
:=
\Pi_{T_{\sbf}\Omega_{\mathrm d}}\nabla F(\sbf)\) and the last inequality uses
Assumption~\ref{asm:manifold-pl-dist} together with Young's inequality.

Define the composite Lyapunov function
\(
    W_{\mathrm{FL}}
    :=
    V_F+\gamma V_y ,
\)
where \(\gamma>0\) is sufficiently large. Combining the preceding
estimates yields
\[
    \dot W_{\mathrm{FL}}
    \le
    -2\alpha_{\mathrm{FL}}W_{\mathrm{FL}}
\]
for any
\(
    \alpha_{\mathrm{FL}}
    \in
    \left(
        0,
        \min
        \left\{
            \frac{2\mu_{\mathrm d}}{N},
            \alpha_G
        \right\}
    \right).
\)
Thus, the ideal FL dynamics is exponentially stable with respect to the
objective residual and the output residual in the tubular neighborhood
\(\mathcal N_{\mathrm d}(\varepsilon)\).

Since \(F\) satisfies the local quadratic-growth condition in
Assumption~\ref{asm:growth}, the objective residual controls the distance to
\(\sbf_{\mathrm d}^{\ast}\) once the trajectory enters the neighborhood
\(\mathcal U_{\mathrm d}\). Because \(V_F(t)\) decays exponentially and
\(\sbf_{\mathrm d}^{\ast}\) is the unique minimizer on
\(\Omega_{\mathrm d}^{\ast}\), the trajectory enters this local neighborhood
in finite time. Thus, for \eqref{eq:ideal-fl-algorithm-dist} under the assumptions in this theorem, we obtain
\[
    \|\xb(t)-\xb_{\mathrm d}^{\ast}\|
    +
    \|\mathbf L\xb(t)\|
    +
    \|\mathbf h^{\mathrm d}(\xb(t))\|
    \le
    C_{\mathrm{FL}}e^{-\alpha_{\mathrm{FL}}t}
    \Big(
        \|\xb(0)-\xb_{\mathrm d}^{\ast}\|
        +
        \|\mathbf L\xb(0)\|
        +
        \|\mathbf h^{\mathrm d}(\xb(0))\|
    \Big).
\]
for some $C_{\rm FL}>0$.

We now turn to the SP realization \eqref{eq:sp-compact-dist}. Let
\(
    \tilde{\ub}_{\mathrm d}
    :=
    \ub_{\mathrm d}-\ub_{\mathrm d}^{\mathrm q}(\xb).
\)
Recalling that \(\ub_{\mathrm d}=\mathbf K_p\yb_{\mathrm d}+\zb\), we have
\[
    \tilde{\ub}_{\mathrm d}
    =
    \zb-\zb_{\mathrm d}^{\mathrm q}(\xb),
    \qquad
    \zb_{\mathrm d}^{\mathrm q}(\xb)
    :=
    \ub_{\mathrm d}^{\mathrm q}(\xb)-\mathbf K_p\yb_{\mathrm d}(\xb).
\]
Thus, the natural fast error in the SP coordinates is the distance from
\(\zb\) to the state-dependent quasi-steady graph
\(\zb_{\mathrm d}^{\mathrm q}(\xb)\). This explains the initialization
condition
\(
    \|\zb(0)-\zb_{\mathrm d}^{\mathrm q}(\xb(0))\|<\delta_z,
\)
which is equivalent to requiring
\(
    \|\ub_{\mathrm d}(0)-\ub_{\mathrm d}^{\mathrm q}(\xb(0))\|<\delta_z.
\)

By Assumptions~\ref{asm:tube-smooth-dist} and \ref{asm:tube-reg-dist}, the
quasi-steady map \(\ub_{\mathrm d}^{\mathrm q}(\xb)\) is well defined and
Lipschitz on \(\mathcal N_{\mathrm d}(\varepsilon)\), and the boundary-layer
system associated with \(\tilde{\ub}_{\mathrm d}\) is uniformly exponentially
stable on the admissible space \(\mathcal A_{\mathrm d}\). Hence, using the
same singular-perturbation argument as in the proof of
Theorem~\ref{thm:sp-dist}, the exponential convergence of the ideal FL dynamics
established above is preserved by the SP realization
\eqref{eq:sp-compact-dist} for all sufficiently small \(\tau\).
The proof is complete.

\section{Proof of Theorem~\ref{thm:euler-sp-dist}}
\label{app:euler-sp-dist}

Let \(\bm\xi_{\mathrm d}:=\operatorname{col}(\mathbf x,\mathbf z_c,\mathbf z_h)\) and
\(\bm\xi_{\mathrm d}^{\ast}:=\operatorname{col}(\mathbf x_{\mathrm d}^{\ast},\mathbf z^\ast)\). We work on the invariant subspace
\(\mathcal S_{\mathrm d}:=\{\operatorname{col}(\mathbf x,\mathbf z_c,\mathbf z_h):
\mathbf x\in\mathbb R^{Nd},\mathbf z_c\in\operatorname{Im}\mathbf L,
\mathbf z_h\in\mathbb R^r\}\). On this subspace, the continuous-time SP
realization \eqref{eq:sp-compact-dist} can be written compactly as
\(\dot{\bm\xi}_{\mathrm d}=\Phi_{\mathrm d}(\bm\xi_{\mathrm d})\), where
\(\Phi_{\mathrm d}\) is continuously differentiable in a neighborhood of
\(\bm\xi_{\mathrm d}^{\ast}\). FS-NDE is the Euler
map \(\bm\xi_{\mathrm d}^{k+1}=\Psi_{\mathrm d}(\bm\xi_{\mathrm d}^{k})
:=\bm\xi_{\mathrm d}^{k}+T\Phi_{\mathrm d}(\bm\xi_{\mathrm d}^{k})\). Since
\(\Psi_{\mathrm d}(\bm\xi_{\mathrm d})=\bm\xi_{\mathrm d}\) if and only if
\(\Phi_{\mathrm d}(\bm\xi_{\mathrm d})=\mathbf 0\), the fixed points of the
Euler map coincide with the equilibria of the continuous-time SP system. By
Theorem~\ref{thm:sp-dist}, \(\bm\xi_{\mathrm d}^{\ast}\) is the locally unique
equilibrium of \eqref{eq:sp-compact-dist}; hence it is also the locally unique
equilibrium of FS-NDE.

Let \(\mathbf R_{\mathrm d}:=\frac{\partial \Phi_{\mathrm d}}
{\partial \bm\xi_{\mathrm d}}(\bm\xi_{\mathrm d}^{\ast})\) be the Jacobian
restricted to \(\mathcal S_{\mathrm d}\). By Theorem~\ref{thm:sp-dist},
\(\bm\xi_{\mathrm d}^{\ast}\) is locally exponentially stable for the
continuous-time SP system with admissible rate \(\alpha_{\mathrm{SP}}\). Hence
\(\mathbf R_{\mathrm d}\) is Hurwitz. Fix any
\(\alpha_{\mathrm E}\in(0,\alpha_{\mathrm{SP}})\) and set
\(\bar\alpha_{\mathrm E}:=(\alpha_{\mathrm E}+\alpha_{\mathrm{SP}})/2\). Then
there exists \(T_{\mathrm d}^\ast>0\) such that, for every
\(T\in(0,T_{\mathrm d}^\ast)\), the matrix
\(\mathbf R_{\mathrm d,T}:=\mathbf I+T\mathbf R_{\mathrm d}\) satisfies
\(\rho(\mathbf R_{\mathrm d,T})<e^{-\bar\alpha_{\mathrm E}T}<e^{-\alpha_{\mathrm E}T}\).
For such a fixed \(T\), let \(\mathbf P_{\mathrm d,T}\succ0\) be the unique
solution of
\[
    \mathbf R_{\mathrm d,T}^{\top}\mathbf P_{\mathrm d,T}\mathbf R_{\mathrm d,T}
    -e^{-2\bar\alpha_{\mathrm E}T}\mathbf P_{\mathrm d,T}
    =
    -\mathbf I .
\]
Denote \(m_{\mathrm d,T}:=\lambda_{\min}(\mathbf P_{\mathrm d,T})\),
\(M_{\mathrm d,T}:=\lambda_{\max}(\mathbf P_{\mathrm d,T})\), and
\(\bar C_{\mathrm d,T}:=\sqrt{M_{\mathrm d,T}/m_{\mathrm d,T}}\). Since
\(\Psi_{\mathrm d}\) is continuously differentiable, for
\(\varepsilon_{\mathrm d,T}:=\sqrt{m_{\mathrm d,T}/M_{\mathrm d,T}}
(e^{-\alpha_{\mathrm E}T}-e^{-\bar\alpha_{\mathrm E}T})>0\), there exists
\(r_{\mathrm d,T}>0\) such that, whenever
\(\|\bm\xi_{\mathrm d}-\bm\xi_{\mathrm d}^{\ast}\|\le r_{\mathrm d,T}\), one has
\[
    \left\|
    \Psi_{\mathrm d}(\bm\xi_{\mathrm d})-\bm\xi_{\mathrm d}^{\ast}
    -\mathbf R_{\mathrm d,T}(\bm\xi_{\mathrm d}-\bm\xi_{\mathrm d}^{\ast})
    \right\|
    \le
    \varepsilon_{\mathrm d,T}
    \|\bm\xi_{\mathrm d}-\bm\xi_{\mathrm d}^{\ast}\|.
\]
Using the norm \(\|\eta\|_{\mathbf P_{\mathrm d,T}}:=
\sqrt{\eta^\top\mathbf P_{\mathrm d,T}\eta}\), it follows that, as long as
\(\bm\xi_{\mathrm d}^{k}\) remains in this neighborhood,
\[
\begin{aligned}
    \|\bm\xi_{\mathrm d}^{k+1}-\bm\xi_{\mathrm d}^{\ast}\|_{\mathbf P_{\mathrm d,T}}
    &\le
    e^{-\bar\alpha_{\mathrm E}T}
    \|\bm\xi_{\mathrm d}^{k}-\bm\xi_{\mathrm d}^{\ast}\|_{\mathbf P_{\mathrm d,T}}
    +
    \sqrt{M_{\mathrm d,T}}\varepsilon_{\mathrm d,T}
    \|\bm\xi_{\mathrm d}^{k}-\bm\xi_{\mathrm d}^{\ast}\|  \\
    &\le
    e^{-\alpha_{\mathrm E}T}
    \|\bm\xi_{\mathrm d}^{k}-\bm\xi_{\mathrm d}^{\ast}\|_{\mathbf P_{\mathrm d,T}} .
\end{aligned}
\]
Therefore,
\[
    \|\bm\xi_{\mathrm d}^{k}-\bm\xi_{\mathrm d}^{\ast}\|
    \le
    \bar C_{\mathrm d,T}e^{-\alpha_{\mathrm E}kT}
    \|\bm\xi_{\mathrm d}^{0}-\bm\xi_{\mathrm d}^{\ast}\|,
    \qquad k=0,1,2,\ldots .
\]

It remains to estimate the residual terms appearing in the theorem. Let
\(r_{\mathrm h,d}>0\) be sufficiently small and define
\(L_{\mathrm h,d}:=\sup_{\|\mathbf x-\mathbf x_{\mathrm d}^{\ast}\|\le r_{\mathrm h,d}}
\|\tilde{\mathbf J}^{\mathrm d}(\mathbf x)\|\). Since
\(\mathbf h^{\mathrm d}(\mathbf x_{\mathrm d}^{\ast})=\mathbf 0\), one has
\(\|\mathbf h^{\mathrm d}(\mathbf x)\|\le
L_{\mathrm h,d}\|\mathbf x-\mathbf x_{\mathrm d}^{\ast}\|\) locally. Moreover,
\(\mathbf L\mathbf x_{\mathrm d}^{\ast}=\mathbf 0\), and hence
\[
\begin{aligned}
    &\|\mathbf x^k-\mathbf x_{\mathrm d}^{\ast}\|
    +\|\mathbf z^k-\mathbf z^\ast\|
    +\|\mathbf L\mathbf x^k\|
    +\|\mathbf h^{\mathrm d}(\mathbf x^k)\|  \\
    &\qquad\le
    \left(\sqrt{2}+\|\mathbf L\|+L_{\mathrm h,d}\right)
    \|\bm\xi_{\mathrm d}^{k}-\bm\xi_{\mathrm d}^{\ast}\|.
\end{aligned}
\]
Combining this bound with the preceding estimate gives the desired inequality
with
\[
    C_{\mathrm E}
    =
    \left(\sqrt{2}+\|\mathbf L\|+L_{\mathrm h,d}\right)
    \bar C_{\mathrm d,T}
    =
    \left(\sqrt{2}+\|\mathbf L\|+L_{\mathrm h,d}\right)
    \sqrt{\frac{M_{\mathrm d,T}}{m_{\mathrm d,T}}}.
\]
Finally, we let
\(\delta_{\mathrm E}:=\min\{r_{\mathrm d,T}/\bar C_{\mathrm d,T},
r_{\mathrm h,d}/\bar C_{\mathrm d,T}\}\) so that the
trajectory remains in the local neighborhood where all estimates above hold.
This completes the proof.

\section{Proof of Theorem~\ref{thm:euler-sp-coupled}}
\label{app:euler-sp-coupled}

Let \(\bm\xi_{\mathrm c}:=\operatorname{col}(\mathbf x,\bm\zeta,\mathbf z_c,\mathbf z_h)\) and
\(\bm\xi_{\mathrm c}^{\ast}:=\operatorname{col}
(\mathbf x_{\mathrm c}^{\ast},\bm\zeta^\ast,\mathbf z^\ast)\). We work on
the invariant subspace
\(\mathcal S_{\mathrm c}:=\{\operatorname{col}(\mathbf x,\bm\zeta,\mathbf z_c,\mathbf z_h):
\mathbf x\in\mathbb R^{Nd},\bm\zeta\in\operatorname{Im}\mathbf L_q,
\mathbf z_c\in\operatorname{Im}\mathbf L,\mathbf z_h\in\mathbb R^{Nq}\}\).
On this subspace, the continuous-time SP realization
\eqref{eq:sp-compact-coupled} can be written compactly as
\(\dot{\bm\xi}_{\mathrm c}=\Phi_{\mathrm c}(\bm\xi_{\mathrm c})\), where
\(\Phi_{\mathrm c}\) is continuously differentiable in a neighborhood of
\(\bm\xi_{\mathrm c}^{\ast}\). FS-NCE is the
Euler map \(\bm\xi_{\mathrm c}^{k+1}=\Psi_{\mathrm c}(\bm\xi_{\mathrm c}^{k})
:=\bm\xi_{\mathrm c}^{k}+T\Phi_{\mathrm c}(\bm\xi_{\mathrm c}^{k})\). Since
\(\Psi_{\mathrm c}(\bm\xi_{\mathrm c})=\bm\xi_{\mathrm c}\) if and only if
\(\Phi_{\mathrm c}(\bm\xi_{\mathrm c})=\mathbf 0\), the fixed points of the
Euler map coincide with the equilibria of the continuous-time SP system. By
Theorem~\ref{thm:sp-coupled}, \(\bm\xi_{\mathrm c}^{\ast}\) is the locally
unique equilibrium of \eqref{eq:sp-compact-coupled}; hence it is also the
locally unique equilibrium of FS-NCE.

Let \(\mathbf R_{\mathrm c}:=\frac{\partial \Phi_{\mathrm c}}
{\partial \bm\xi_{\mathrm c}}(\bm\xi_{\mathrm c}^{\ast})\) be the Jacobian
restricted to \(\mathcal S_{\mathrm c}\). By Theorem~\ref{thm:sp-coupled},
\(\bm\xi_{\mathrm c}^{\ast}\) is locally exponentially stable for the
continuous-time SP system with admissible rate \(\alpha_{\mathrm{SP}}\). Hence
\(\mathbf R_{\mathrm c}\) is Hurwitz. Fix any
\(\alpha_{\mathrm E}\in(0,\alpha_{\mathrm{SP}})\) and set
\(\bar\alpha_{\mathrm E}:=(\alpha_{\mathrm E}+\alpha_{\mathrm{SP}})/2\). Then
there exists \(T_{\mathrm c}^\ast>0\) such that, for every
\(T\in(0,T_{\mathrm c}^\ast)\), the matrix
\(\mathbf R_{\mathrm c,T}:=\mathbf I+T\mathbf R_{\mathrm c}\) satisfies
\(\rho(\mathbf R_{\mathrm c,T})<e^{-\bar\alpha_{\mathrm E}T}<e^{-\alpha_{\mathrm E}T}\).
For such a fixed \(T\), let \(\mathbf P_{\mathrm c,T}\succ0\) be the unique
solution of
\[
    \mathbf R_{\mathrm c,T}^{\top}\mathbf P_{\mathrm c,T}\mathbf R_{\mathrm c,T}
    -e^{-2\bar\alpha_{\mathrm E}T}\mathbf P_{\mathrm c,T}
    =
    -\mathbf I .
\]
Denote \(m_{\mathrm c,T}:=\lambda_{\min}(\mathbf P_{\mathrm c,T})\),
\(M_{\mathrm c,T}:=\lambda_{\max}(\mathbf P_{\mathrm c,T})\), and
\(\bar C_{\mathrm c,T}:=\sqrt{M_{\mathrm c,T}/m_{\mathrm c,T}}\). Since
\(\Psi_{\mathrm c}\) is continuously differentiable, for
\(\varepsilon_{\mathrm c,T}:=\sqrt{m_{\mathrm c,T}/M_{\mathrm c,T}}
(e^{-\alpha_{\mathrm E}T}-e^{-\bar\alpha_{\mathrm E}T})>0\), there exists
\(r_{\mathrm c,T}>0\) such that, whenever
\(\|\bm\xi_{\mathrm c}-\bm\xi_{\mathrm c}^{\ast}\|\le r_{\mathrm c,T}\), one has
\[
    \left\|
    \Psi_{\mathrm c}(\bm\xi_{\mathrm c})-\bm\xi_{\mathrm c}^{\ast}
    -\mathbf R_{\mathrm c,T}(\bm\xi_{\mathrm c}-\bm\xi_{\mathrm c}^{\ast})
    \right\|
    \le
    \varepsilon_{\mathrm c,T}
    \|\bm\xi_{\mathrm c}-\bm\xi_{\mathrm c}^{\ast}\|.
\]
Using the norm \(\|\eta\|_{\mathbf P_{\mathrm c,T}}:=
\sqrt{\eta^\top\mathbf P_{\mathrm c,T}\eta}\), it follows that, as long as
\(\bm\xi_{\mathrm c}^{k}\) remains in this neighborhood,
\[
\begin{aligned}
    \|\bm\xi_{\mathrm c}^{k+1}-\bm\xi_{\mathrm c}^{\ast}\|_{\mathbf P_{\mathrm c,T}}
    &\le
    e^{-\bar\alpha_{\mathrm E}T}
    \|\bm\xi_{\mathrm c}^{k}-\bm\xi_{\mathrm c}^{\ast}\|_{\mathbf P_{\mathrm c,T}}
    +
    \sqrt{M_{\mathrm c,T}}\varepsilon_{\mathrm c,T}
    \|\bm\xi_{\mathrm c}^{k}-\bm\xi_{\mathrm c}^{\ast}\|  \\
    &\le
    e^{-\alpha_{\mathrm E}T}
    \|\bm\xi_{\mathrm c}^{k}-\bm\xi_{\mathrm c}^{\ast}\|_{\mathbf P_{\mathrm c,T}} .
\end{aligned}
\]
Therefore,
\[
    \|\bm\xi_{\mathrm c}^{k}-\bm\xi_{\mathrm c}^{\ast}\|
    \le
    \bar C_{\mathrm c,T}e^{-\alpha_{\mathrm E}kT}
    \|\bm\xi_{\mathrm c}^{0}-\bm\xi_{\mathrm c}^{\ast}\|,
    \qquad k=0,1,2,\ldots .
\]

It remains to estimate the residual terms appearing in the theorem. Let
\(r_{\mathrm h,c}>0\) be sufficiently small and define
\(L_{\mathrm h,c}:=\sup_{\|\mathbf x-\mathbf x_{\mathrm c}^{\ast}\|\le r_{\mathrm h,c}}
\|\tilde{\mathbf J}^{\mathrm c}(\mathbf x)\|\). Since
\(\mathbf h^{\mathrm c}(\mathbf x_{\mathrm c}^{\ast})+\mathbf L_q\bm\zeta^\ast=\mathbf 0\),
one has locally
\[
    \|\mathbf h^{\mathrm c}(\mathbf x)+\mathbf L_q\bm\zeta\|
    \le
    L_{\mathrm h,c}\|\mathbf x-\mathbf x_{\mathrm c}^{\ast}\|
    +
    \|\mathbf L_q\|\|\bm\zeta-\bm\zeta^\ast\|.
\]
Moreover, \(\mathbf L\mathbf x_{\mathrm c}^{\ast}=\mathbf 0\), and hence
\[
\begin{aligned}
    &\|\mathbf x^k-\mathbf x_{\mathrm c}^{\ast}\|
    +\|\bm\zeta^k-\bm\zeta^\ast\|
    +\|\mathbf z^k-\mathbf z^\ast\|
    +\|\mathbf L\mathbf x^k\| 
    +\|\mathbf h^{\mathrm c}(\mathbf x^k)+\mathbf L_q\bm\zeta^k\|  \\
    &\qquad\le
    \left(\sqrt{3}+\|\mathbf L\|+L_{\mathrm h,c}+\|\mathbf L_q\|\right)
    \|\bm\xi_{\mathrm c}^{k}-\bm\xi_{\mathrm c}^{\ast}\|.
\end{aligned}
\]
Combining this bound with the preceding estimate gives the desired inequality
with
\[
    C_{\mathrm E}
    =
    \left(\sqrt{3}+\|\mathbf L\|+L_{\mathrm h,c}+\|\mathbf L_q\|\right)
    \bar C_{\mathrm c,T}
    =
    \left(\sqrt{3}+\|\mathbf L\|+L_{\mathrm h,c}+\|\mathbf L_q\|\right)
    \sqrt{\frac{M_{\mathrm c,T}}{m_{\mathrm c,T}}}.
\]
Finally, we let
\(\delta_{\mathrm E}:=\min\{r_{\mathrm c,T}/\bar C_{\mathrm c,T},
r_{\mathrm h,c}/\bar C_{\mathrm c,T}\}\) so that the
trajectory remains in the local neighborhood where all estimates above hold.
This completes the proof.

\bibliographystyle{agsm}
\bibliography{FL_DO/reference}  

\end{document}